\documentclass[11pt]{article}
  \usepackage{amsmath}
  \usepackage{amsthm}
  \usepackage{latexsym}
  \usepackage{lscape}
  \usepackage{amscd}
  \usepackage{amssymb}
  \usepackage{txfonts}
  \usepackage{bussproofs}
	\usepackage{array}

\newtheorem{theorem}{Theorem}[section]
 \newtheorem{lemma}[theorem]{Lemma}
 \newtheorem{prop}[theorem]{Proposition}

 \theoremstyle{definition}
 \newtheorem{definition}[theorem]{Definition}
 \newtheorem{example}[theorem]{Example}

 \newtheorem{remark}[theorem]{Remark}

 \theoremstyle{remark}

\usepackage[left=2cm,top=3cm,right=2cm,bottom=3cm]{geometry}
\newcommand{\comment}[1]{}

\newcommand{\bigamp}{\mathop{\mbox{\Large \&}}}

\newcommand{\amp}{\mathop{\&}}

\newcommand{\marginnote}[1]{\marginpar{\raggedright\tiny{#1}}}

\newcommand{\starfor}{{/\!\!}_{\star}}
\newcommand{\circfor}{{/\!}_{\circ}}

\newcommand{\starback}{\backslash_{\star}}
\newcommand{\circback}{\backslash_{\circ}}
\newcommand{\blhd}{\blacktriangleleft}
\newcommand{\brhd}{\blacktriangleright}

\newcommand{\bba}{\mathbb{A}}
\newcommand{\bbA}{\mathbb{A}}
\newcommand{\bbas}{\mathbb{A}^{\delta}}
\newcommand{\kbbas}{K(\mathbb{A}^{\delta})}
\newcommand{\obbas}{O(\mathbb{A}^{\delta})}

\newcommand{\Cc}{\mathbb{C}}
\newcommand{\p}{\mathcal{P}}

\newcommand{\jir}{J^{\infty}(\bba^{\delta})}
\newcommand{\mir}{M^{\infty}(\bba^{\delta})}

\newcommand{\jira}{J^{\infty}(\bba)}
\newcommand{\mira}{M^{\infty}(\bba)}
\newcommand{\jty}{J^{\infty}}
\newcommand{\mty}{M^{\infty}}

\newcommand{\nomi}{\mathbf{i}}
\newcommand{\nomj}{\mathbf{j}}
\newcommand{\nomk}{\mathbf{k}}

\newcommand{\cnomm}{\mathbf{m}}
\newcommand{\cnomn}{\mathbf{n}}

\renewcommand{\phi}{\varphi}
\renewcommand{\emptyset}{\varnothing}

\newcommand{\X}{\mathbb{X}}

\newcommand{\z}{\mathbf{z}}

\newcommand{\STP}{\mathrm{ST}^{(+)}}
\newcommand{\STN}{\mathrm{ST}^{(-)}}

\title{Algorithmic correspondence and canonicity for non-distributive logics}
\author{Willem Conradie\thanks{The research of the first author was supported by the National Research Foundation of South Africa.} \ and Alessandra Palmigiano\thanks{The research of the second author was supported by the Vidi grant 016.138.314 of the Netherlands Organization for Scientific Research
    (NWO), by the NWO Aspasia grant 015.008.054, and by a Delft Technology Fellowship awarded in 2013.} }

\begin{document}
\maketitle

\begin{abstract}%
We extend the theory of unified correspondence to a very broad class of logics with algebraic semantics given by varieties of normal lattice expansions (LEs), also known as `lattices with operators'. Specifically, we introduce a very general syntactic definition of the class of Sahlqvist formulas and inequalities, which applies uniformly to each LE-signature  and is given purely in terms of the order-theoretic properties of the algebraic interpretations of the logical connectives. Together with this, we introduce a variant of the algorithm ALBA, specific to the setting of LEs, which effectively computes first-order correspondents of LE-inequalities, and is guaranteed to succeed on a wide class of inequalities (the so-called inductive inequalities) which significantly extend the Sahlqvist class. Further, we show that every inequality on which ALBA succeeds is  canonical. 
The projection of these results yields state-of-the-art correspondence theory for many well known substructural logics, such as the Lambek calculus and its extensions, the Lambek-Grishin calculus, the logic of (not necessarily distributive) de Morgan lattices, and the multiplicative-additive fragment of linear logic. 
{\em Math.\ Subject Class.}  03B45, 03B47, 03B60, 06D50, 06D10, 03G10, 06E15.
\end{abstract}

\section*{Introduction}

Sahqvist theory has a long and distinguished history within modal logic, going back to \cite{Sah75} and \cite{vB83}. Sahlqvist's theorem \cite{Sah75} gives a syntactic definition of a class of modal formulas, the {\em Sahlqvist class}, each member of which defines an elementary class of frames and is canonical. These are two highly desirable properties: the canonicity of an axiomatization guarantees the strong Kripke completeness of its associated logic, while elementarity brings with it all the computational and theoretical advantages of first-order logic over second-order logic. As it turns out, both these properties (singularly and in combination) are algorithmically undecidable \cite{Chagrov:Chagrova:2006}, so a decidable approximation, like the Sahlqvist class, is very desirable.

Over the years, many extensions, variations and analogues of this result have appeared, including alternative proofs in e.g.\ \cite{Sambin:Vacarro:89}, generalizations to arbitrary modal signatures \cite{DeRijke:Venema:95}, variations of the correspondence language \cite{Ohlbach:Schmidt:97, vanBenthem:Fix:Points:2006}, Sahlqvist-type results for hybrid logics  \cite{TenCate:Marx:Viana},  various substructural logics \cite{Kurtonina,Suzuki:RSL:2013,Gehrke}, mu calculus  \cite{Van:Bent:Bezh:Hodk:Studia}, enlargements of the Sahlqvist class to e.g.\ the {\em inductive} formulas of \cite{Goranko:Vakarelov:2006}, to mention but a few.

However, this literature displays very different and uneven degrees of development of Sahlqvist theory for logics not based on classical normal modal logic. More fundamentally, what is lacking is an explicit, overarching theoretical framework which would provide a mathematically grounded way to compare (hierarchies of) Sahlqvist-type classes belonging to different logical settings.
\footnote{Being able to compare Sahlqvist classes of different logical signatures is particularly useful when focusing on substructural logics, since in this setting it is common to treat fragments or expansions of a given initial logical signature. In Section \ref{Inductive:Fmls:Section}, we will discuss  a comparison between Sahlqvist classes of logical settings which differ not only in the signature  but also in the underlying axiomatization.} Along with this lack of {\em uniformity}, each Sahlqvist-type results is tied to a particular choice of relational semantics for the relevant logic. Such a choice could be motivated by the fact that a logic has a uniquely established set-based semantics, but for many logics, like substructural logics, this is not the case.
%
%
Hence, it is desirable to have a {\em modular} Sahlqvist theory that would distinguish core characteristics from incidental details relating to a particular choice of relational semantics.
%

A theory which subsumes the previous results and which satisfies the desiderata of uniformity and modularity is currently emerging, and has been dubbed \emph{unified correspondence} \cite{UnifiedCor}. It is built on duality-theoretic insights \cite{ConPalSur} 
and uniformly exports the state-of-the-art in Sahlqvist theory from normal modal logic to a wide range of logics which include, among others, intuitionistic and distributive lattice-based (normal modal) logics \cite{ALBAPaper}, non-normal (regular) modal logics of arbitrary modal signature \cite{PaSoZh15r}, hybrid logics \cite{ConRob}, and mu-calculus \cite{CoCr14,CFPS}.

	
	The breadth of this work has also stimulated many and varied applications. Some are closely related to the core concerns of the theory itself, such as the understanding of the relationship between different methodologies for obtaining canonicity results \cite{PaSoZh14,CPZ:constructive}, or of the phenomenon of pseudo-correspondence \cite{CGPSZ14}. Other, possibly surprising applications include the dual characterizations of classes of finite lattices \cite{FrPaSa14}, and the identification of the syntactic shape of axioms which can be translated into analytic structural rules of a proper display calculus \cite{GMPTZ}. Finally, the insights of unified correspondence theory have made it possible to determine the extent to which the Sahlqvist theory of classes of normal DLEs can be reduced to the Sahlqvist theory of normal Boolean expansions, by means of G\"{o}del-type translations \cite{CPZ:Trans}.

The most important technical tools of unified correspondence are: (a) very general syntactic definitions of the class of Sahlqvist formulas and of the strict superclass of inductive formulas, which apply uniformly to all logical signatures; 
(b) the algorithm ALBA (an acronym for `Ackermann Lemma Based Algorithm'), uniformly based on the order theoretic properties of the connectives of each logical signature, and which  effectively computes first-order correspondents of all inductive formulas while simultaneously proving their canonicity.

The synergy between duality theory and algebraic semantics provides the mathematical underpinning of these tools, and makes it possible to address both  desiderata of uniformity and   modularity. Indeed, thanks to duality,  correspondence theory can be transferred from the model-theoretic setting of its origins to  the environment of perfect normal lattice expansion (cf.\ Definition \ref{def:perfect LE}). In this environment, the mechanisms of correspondence are laid bare and can be very perspicuously explained in terms of the order-theoretic properties of the interpretation of the logical connectives (cf.\ \cite{ConPalSur}). This  allows for a treatment which abstracts away from any specific signature, while at the same time providing the ground for comparisons between signatures based on these order-theoretic properties.
Moreover, adopting this perspective makes it possible to divide the computation of first-order correspondents neatly into two stages: first, a {\em reduction} stage, computed by ALBA, in which input formulas or inequalities are equivalently transformed into sets of quasi-inequalities in an expanded  language, naturally interpreted on perfect algebras of suitable signature; second, a {\em translation} stage, in which each quasi-inequality in the output of the reduction stage is translated into a formula of the first-order correspondence language, via a suitably defined standard translation. While the latter
stage depends on the particular choice of relational semantics, the former does not, and hence can be performed
once and for all state-based settings, thus making the treatment modular.  The successful completion of the first stage alone is also enough to guarantee canonicity, both in the ordinary sense and constructively (cf.\ Section \ref{sec:conclusions}).


The contributions of the present paper are core to the research program of unified correspondence. We extend the two main tools of unified correspondence (namely, the uniform definition of Sahlqvist and inductive inequalities, and the algorithm ALBA) to a vast class of logics (referred to as {\em LE-logics}) which are captured algebraically by varieties of so-called {\em normal lattice expansions} (cf.\ Definition \ref{def:DLE}). This class includes the axiomatic extensions of basic orthomodular logic \cite{Goldblatt:Ortho:74}, the logic of the non-distributive de Morgan algebras \cite{balbes}, the Lambek-Grishin calculus \cite{Grishin1983, Moortgat}, the multiplicative-additive fragment of linear logic (MALL) \cite{GaJiKoOn07}.  While being  based on the same fundamental engine (the Ackermann lemma), the version of ALBA defined here generalizes those in \cite{ALBAPaper} and \cite{Conradie:et:al:SQEMAI} in a non trivial way. Indeed, due to the fact that in the setting of general lattices the completely join-irreducible (resp.\ meet-irreducible) elements are not necessarily completely join-prime (resp.\ meet-prime), the so-called {\em approximation rules} in the style of those defined  in \cite{ALBAPaper} and \cite{UnifiedCor} are not sound anymore. Hence, a significantly different type of approach needs to be adopted (for more discussion see Example \ref{Examp:Hopeless} and Remark \ref{rem: towards constructive can}).   
The contributions of the present paper include a uniform proof of the canonicity of any LE-inequality on which ALBA succeeds.

\paragraph{Structure of the paper.} In section \ref{Prelim:Section}, we introduce the syntax and semantics of the basic LE-logics, 
together with some other necessary preliminaries. In  section \ref{subsec: dual frames}, we outline two set-based semantic environments for LE-languages, and for each of them we provide the corresponding standard translation, with a special focus on the languages of LML (defined in Section \ref{subsec: dual frames}) and the Lambek-Grishin calculus.
In section \ref{Inductive:Fmls:Section}, we define the inductive and Sahlqvist inequalities for any basic LE-logic. 
We show how the Sahlqvist and inductive class of the Lambek-Grishin calculus project appropriately onto the corresponding classes in \cite{GNV} and \cite{ALBAPaper}. The non-distributive ALBA algorithm is introduced in section \ref{Spec:Alg:Section}. ALBA attempts to eliminate all propositional variables from inequalities in favour of special variables ranging over the join- and meet-irreducible elements of perfect lattices. This is done by applying rewrite rules which exploit the residuated and distributive behaviour of the operations corresponding to the connectives of the language. Section \ref{Examples:Section} provides examples of ALBA-reductions of Sahlqvist, inductive, and non-inductive inequalities in various signatures.
In section \ref{Crrctnss:Section} we prove that ALBA is correct, i.e., that the outputs returned by it are indeed equivalent to the input in the appropriate sense.
In section \ref{Complete:For:Inductive:Section} we prove that ALBA successfully reduces all inductive inequalities.  This, together with the canonicity of all inequalities suitably reducible by ALBA, which is proved in section \ref{section:canonicity}, implies that all inductive inequalities are elementary and canonical. Final remarks on constructive canonicity are collected in section \ref{sec:conclusions}.  Technical lemmas are relegated to section \ref{appendix}.

\section{Preliminaries}\label{Prelim:Section}

In this section we present the languages under consideration together with their associated minimal logics. We collect various preliminaries related to the algebraic semantics of these languages.

\subsection{Language and axioms}
Our base language is an unspecified but fixed language $\mathcal{L}_\mathrm{LE}$, to be interpreted over  lattice expansions of compatible similarity type. As mentioned in the introduction, this setting uniformly accounts for many well known logical systems.
	
	In our treatment, we will make heavy use of the following auxiliary definition: an {\em order-type} over $n\in \mathbb{N}$\footnote{Throughout the paper, order-types will be typically associated with arrays of variables $\vec p: = (p_1,\ldots, p_n)$. When the order of the variables in $\vec p$ is not specified, we will sometimes abuse notation and write $\varepsilon(p) = 1$ or $\varepsilon(p) = \partial$.} is an $n$-tuple $\epsilon\in \{1, \partial\}^n$. For every order type $\epsilon$, we denote its {\em opposite} order type by $\epsilon^\partial$, that is, $\epsilon^\partial_i = 1$ iff $\epsilon_i=\partial$ for every $1 \leq i \leq n$. For any lattice $\bba$, we let $\bba^1: = \bba$ and $\bba^\partial$ be the dual lattice, that is, the lattice associated with the converse partial order of $\bba$. For any order type $\varepsilon$, we let $\bba^\varepsilon: = \Pi_{i = 1}^n \bba^{\varepsilon_i}$.
	
	The language $\mathcal{L}_\mathrm{LE}(\mathcal{F}, \mathcal{G})$ (from now on abbreviated as $\mathcal{L}_\mathrm{LE}$) takes as parameters: 1) a denumerable set $\mathsf{PROP}$ of proposition letters, elements of which are denoted $p,q,r$, possibly with indexes; 2) disjoint sets of connectives $\mathcal{F}$ and $\mathcal{G}$. Each $f\in \mathcal{F}$ and $g\in \mathcal{G}$ has arity $n_f\in \mathbb{N}$ (resp.\ $n_g\in \mathbb{N}$) and is associated with some order-type $\varepsilon_f$ over $n_f$ (resp.\ $\varepsilon_g$ over $n_g$).\footnote{Unary $f$ (resp.\ $g$) will be sometimes denoted as $\Diamond$ (resp.\ $\Box$) if the order-type is 1, and $\lhd$ (resp.\ $\rhd$) if the order-type is $\partial$.} The terms (formulas) of $\mathcal{L}_\mathrm{LE}$ are defined recursively as follows:
	\[
	\phi ::= p \mid \bot \mid \top \mid \phi \wedge \phi \mid \phi \vee \phi \mid f(\overline{\phi}) \mid g(\overline{\phi})
	\]
	where $p \in \mathsf{AtProp}$, $f \in \mathcal{F}$, $g \in \mathcal{G}$. Terms in $\mathcal{L}_\mathrm{LE}$ will be denoted either by $s,t$, or by lowercase Greek letters such as $\varphi, \psi, \gamma$ etc. The set of all $\mathcal{L}_\mathrm{LE}$-inequalities $\phi \leq \psi$ where $\phi, \psi$ are $\mathcal{L}_\mathrm{LE}$-terms will be denoted  $\mathrm{LE}$. The set of all $\mathcal{L}_\mathrm{LE}$-quasi-inequalities, i.e., expressions of the form $(\phi_1 \leq \psi_1 \amp \cdots \amp \phi_n \leq \psi_n) \Rightarrow \phi \leq \psi$ where $\phi_1, \ldots, \phi_n, \psi_1, \ldots \psi_n, \phi, \psi \in \mathcal{L}_\mathrm{LE}$, will be denoted by $\mathcal{L}_\mathrm{LE}^{\mathit{quasi}}$.

The formulas of \emph{distributive modal logic} (cf.\ \cite{GNV}, \cite{ALBAPaper}), denoted $\mathrm{DML}_{\mathit{term}}$, are obtained by instantiating $\mathcal{F}: = \{\Diamond, {\lhd}\}$ with $n_\Diamond = n_\lhd = 1$, $\varepsilon_\Diamond = 1$  and $\varepsilon_\lhd = \partial$, and   $\mathcal{G} = \{\Box, {\rhd}\}$ with $n_\Box = n_\rhd = 1$, $\varepsilon_\Box = 1$  and $\varepsilon_\rhd = \partial$. The formulas of the Full Lambek calculus \cite{la61} are obtained by instantiating $\mathcal{F}: = \{\circ\}$ with $n_\circ  = 2$, $\varepsilon_\circ = (1, 1)$   and   $\mathcal{G} = \{\backslash, /\}$ with $n_\backslash = n_/ = 2$, $\varepsilon_\backslash = (\partial, 1)$  and $\varepsilon_/ = (1, \partial)$. The formulas of the  Lambek-Grishin calculus (cf.\ \cite{Moortgat}) 
are obtained by instantiating $\mathcal{F}: = \{\circ, \starfor, \starback \}$ with $n_\circ = n_{\starback} = n_{\starfor} = 2$, $\varepsilon_\circ = (1, 1)$, $\varepsilon_{\starback} = (\partial, 1)$, $\varepsilon_{\starfor} = (1, \partial)$   and   $\mathcal{G} := \{\star, \circfor, \circback\}$ with $n_\star = n_{\circfor} = n_{\circback} = 2$, $\varepsilon_\star = (1, 1)$,  $\varepsilon_{\circback} = (\partial, 1)$, $\varepsilon_{\circfor} = (1, \partial)$.

\subsection{Normal lattice expansions, and their canonical extensions}

	\begin{definition}
		\label{def:DLE}
		For any tuple $(\mathcal{F}, \mathcal{G})$ of disjoint sets of function symbols as above, a {\em  lattice expansion} (abbreviated as LE) is a tuple $\bba = (L, \mathcal{F}^\bbA, \mathcal{G}^\bbA)$ such that $L$ is a bounded  lattice, $\mathcal{F}^\bbA = \{f^\bbA\mid f\in \mathcal{F}\}$ and $\mathcal{G}^\bbA = \{g^\bbA\mid g\in \mathcal{G}\}$, such that every $f^\bbA\in\mathcal{F}^\bbA$ (resp.\ $g^\bbA\in\mathcal{G}^\bbA$) is an $n_f$-ary (resp.\ $n_g$-ary) operation on $\bbA$. An LE is {\em normal} if every $f^\bbA\in\mathcal{F}^\bbA$ (resp.\ $g^\bbA\in\mathcal{G}^\bbA$) preserves finite (hence also empty) joins (resp.\ meets) in each coordinate with $\epsilon_f(i)=1$ (resp.\ $\epsilon_g(i)=1$) and reverses finite (hence also empty) meets (resp.\ joins) in each coordinate with $\epsilon_f(i)=\partial$ (resp.\ $\epsilon_g(i)=\partial$).\footnote{\label{footnote:DLE vs DLO} Normal LEs are sometimes referred to as {\em  lattices with operators} (LOs). This terminology derives from the setting of Boolean algebras with operators, in which operators are understood as operations which preserve finite (hence also empty) joins in each coordinate. Thanks to the Boolean negation, operators are typically taken as primitive connectives, and all the other operations are reduced to these. However, this terminology results somewhat ambiguous in the lattice setting, in which primitive operations are typically maps which are operators if seen as $\bbA^\epsilon\to \bbA^\eta$ for some order-type $\epsilon$ on $n$ and some order-type $\eta\in \{1, \partial\}$. Rather than speaking of lattices with $(\varepsilon, \eta)$-operators, we then speak of normal LEs.} Let $\mathbb{LE}$ be the class of LEs. Sometimes we will refer to certain LEs as $\mathcal{L}_\mathrm{LE}$-algebras when we wish to emphasize that these algebras have a compatible signature with the logical language we have fixed.
	\end{definition}
In the remainder of the paper,
we will abuse notation and write e.g.\ $f$ for $f^\bbA$ when this causes no confusion.
Normal LEs constitute the main semantic environment of the present paper. Henceforth, since every LE is assumed to be normal, the adjective will be typically dropped. The class of all LEs is equational, and can be axiomatized by the usual lattice identities and the following equations for any $f\in \mathcal{F}$ (resp.\ $g\in \mathcal{G}$) and $1\leq i\leq n_f$ (resp.\ for each $1\leq j\leq n_g$):
	\begin{itemize}
		\item if $\varepsilon_f(i) = 1$, then $f(p_1,\ldots, p\vee q,\ldots,p_{n_f}) = f(p_1,\ldots, p,\ldots,p_{n_f})\vee f(p_1,\ldots, q,\ldots,p_{n_f})$ and\\ $f(p_1,\ldots, \bot,\ldots,p_{n_f}) = \bot$,
		\item if $\varepsilon_f(i) = \partial$, then $f(p_1,\ldots, p\wedge q,\ldots,p_{n_f}) = f(p_1,\ldots, p,\ldots,p_{n_f})\vee f(p_1,\ldots, q,\ldots,p_{n_f})$ and\\ $f(p_1,\ldots, \top,\ldots,p_{n_f}) = \bot$,
		\item if $\varepsilon_g(j) = 1$, then $g(p_1,\ldots, p\wedge q,\ldots,p_{n_g}) = g(p_1,\ldots, p,\ldots,p_{n_g})\wedge g(p_1,\ldots, q,\ldots,p_{n_g})$ and\\ $g(p_1,\ldots, \top,\ldots,p_{n_g}) = \top$,
		\item if $\varepsilon_g(j) = \partial$, then $g(p_1,\ldots, p\vee q,\ldots,p_{n_g}) = g(p_1,\ldots, p,\ldots,p_{n_g})\wedge g(p_1,\ldots, q,\ldots,p_{n_g})$ and\\ $g(p_1,\ldots, \bot,\ldots,p_{n_g}) = \top$.
	\end{itemize}
	Each language $\mathcal{L}_\mathrm{LE}$ is interpreted in the appropriate class of LEs. In particular, for every LE $\bba$, each operation $f^\bba\in \mathcal{F}^\bbA$ (resp.\ $g^\bba\in \mathcal{G}^\bbA$) is finitely join-preserving (resp.\ meet-preserving) in each coordinate when regarded as a map $f^\bba: \bba^{\varepsilon_f}\to \bba$ (resp.\ $g^\bba: \bba^{\varepsilon_g}\to \bba$).
Typically, lattice-based logics of this kind are not expressive enough to allow an implication-like term to be defined out of the primitive connectives. Therefore the entailment relation cannot be recovered from the set of tautologies, hence the deducibility has to be defined in terms of sequents. This motivates the following:	
	\begin{definition}
		\label{def:DLE:logic:general}
		For any language $\mathcal{L}_\mathrm{LE} = \mathcal{L}_\mathrm{LE}(\mathcal{F}, \mathcal{G})$, the {\em basic}, or {\em minimal} $\mathcal{L}_\mathrm{LE}$-{\em logic} is a set of sequents $\phi\vdash\psi$, with $\phi,\psi\in\mathcal{L}_\mathrm{LE}$, which contains the following axioms:
		\begin{itemize}
			\item Sequents for lattice operations:
			\begin{align*}
				&p\vdash p, && \bot\vdash p, && p\vdash \top, & &  &\\
				&p\vdash p\vee q, && q\vdash p\vee q, && p\wedge q\vdash p, && p\wedge q\vdash q, &
			\end{align*}
			\item Sequents for additional connectives:
			\begin{align*}
				& f(p_1,\ldots, \bot,\ldots,p_{n_f}) \vdash \bot,~\mathrm{for}~ \varepsilon_f(i) = 1,\\
				& f(p_1,\ldots, \top,\ldots,p_{n_f}) \vdash \bot,~\mathrm{for}~ \varepsilon_f(i) = \partial,\\
				&\top\vdash g(p_1,\ldots, \top,\ldots,p_{n_g}),~\mathrm{for}~ \varepsilon_g(i) = 1,\\
				&\top\vdash g(p_1,\ldots, \bot,\ldots,p_{n_g}),~\mathrm{for}~ \varepsilon_g(i) = \partial,\\
				&f(p_1,\ldots, p\vee q,\ldots,p_{n_f}) \vdash f(p_1,\ldots, p,\ldots,p_{n_f})\vee f(p_1,\ldots, q,\ldots,p_{n_f}),~\mathrm{for}~ \varepsilon_f(i) = 1,\\
				&f(p_1,\ldots, p\wedge q,\ldots,p_{n_f}) \vdash f(p_1,\ldots, p,\ldots,p_{n_f})\vee f(p_1,\ldots, q,\ldots,p_{n_f}),~\mathrm{for}~ \varepsilon_f(i) = \partial,\\
				& g(p_1,\ldots, p,\ldots,p_{n_g})\wedge g(p_1,\ldots, q,\ldots,p_{n_g})\vdash g(p_1,\ldots, p\wedge q,\ldots,p_{n_g}),~\mathrm{for}~ \varepsilon_g(i) = 1,\\
				& g(p_1,\ldots, p,\ldots,p_{n_g})\wedge g(p_1,\ldots, q,\ldots,p_{n_g})\vdash g(p_1,\ldots, p\vee q,\ldots,p_{n_g}),~\mathrm{for}~ \varepsilon_g(i) = \partial,
			\end{align*}
		\end{itemize}
		and is closed under the following inference rules:
		\begin{displaymath}
			\frac{\phi\vdash \chi\quad \chi\vdash \psi}{\phi\vdash \psi}
			\quad
			\frac{\phi\vdash \psi}{\phi(\chi/p)\vdash\psi(\chi/p)}
			\quad
			\frac{\chi\vdash\phi\quad \chi\vdash\psi}{\chi\vdash \phi\wedge\psi}
			\quad
			\frac{\phi\vdash\chi\quad \psi\vdash\chi}{\phi\vee\psi\vdash\chi}
		\end{displaymath}
		\begin{displaymath}
			\frac{\phi\vdash\psi}{f(p_1,\ldots,\phi,\ldots,p_n)\vdash f(p_1,\ldots,\psi,\ldots,p_n)}{~(\varepsilon_f(i) = 1)}
		\end{displaymath}
		\begin{displaymath}
			\frac{\phi\vdash\psi}{f(p_1,\ldots,\psi,\ldots,p_n)\vdash f(p_1,\ldots,\phi,\ldots,p_n)}{~(\varepsilon_f(i) = \partial)}
		\end{displaymath}
		\begin{displaymath}
			\frac{\phi\vdash\psi}{g(p_1,\ldots,\phi,\ldots,p_n)\vdash g(p_1,\ldots,\psi,\ldots,p_n)}{~(\varepsilon_g(i) = 1)}
		\end{displaymath}
		\begin{displaymath}
			\frac{\phi\vdash\psi}{g(p_1,\ldots,\psi,\ldots,p_n)\vdash g(p_1,\ldots,\phi,\ldots,p_n)}{~(\varepsilon_g(i) = \partial)}.
		\end{displaymath}
		The minimal LE-logic is denoted by $\mathbf{L}_\mathrm{LE}$. For any LE-language $\mathcal{L}_{\mathrm{LE}}$, by an {\em $\mathrm{LE}$-logic} we understand any axiomatic extension of the basic $\mathcal{L}_{\mathrm{LE}}$-logic in $\mathcal{L}_{\mathrm{LE}}$.
	\end{definition}
	
	For every LE $\bba$, the symbol $\vdash$ is interpreted as the lattice order $\leq$. A sequent $\phi\vdash\psi$ is valid in $\bba$ if $h(\phi)\leq h(\psi)$ for every homomorphism $h$ from the $\mathcal{L}_\mathrm{LE}$-algebra of formulas over $\mathsf{PROP}$ to $\bba$. The notation $\mathbb{LE}\models\phi\vdash\psi$ indicates that $\phi\vdash\psi$ is valid in every LE. Then, by means of a routine Lindenbaum-Tarski construction, it can be shown that the minimal LE-logic $\mathbf{L}_\mathrm{LE}$ is sound and complete with respect to its correspondent class of algebras $\mathbb{LE}$, i.e.\ that any sequent $\phi\vdash\psi$ is provable in $\mathbf{L}_\mathrm{LE}$ iff $\mathbb{LE}\models\phi\vdash\psi$. 
	
	\subsection{The `tense' language $\mathcal{L}_\mathrm{LE}^*$}
	\label{ssec:expanded tense language}
	Any given language $\mathcal{L}_\mathrm{LE} = \mathcal{L}_\mathrm{LE}(\mathcal{F}, \mathcal{G})$ can be associated with the language $\mathcal{L}_\mathrm{LE}^* = \mathcal{L}_\mathrm{LE}(\mathcal{F}^*, \mathcal{G}^*)$, where $\mathcal{F}^*\supseteq \mathcal{F}$ and $\mathcal{G}^*\supseteq \mathcal{G}$ are obtained by expanding $\mathcal{L}_\mathrm{LE}$ with the following connectives:
	\begin{enumerate}
		\item the $n_f$-ary connective $f^\sharp_i$ for $0\leq i\leq n_f$, the intended interpretation of which is the right residual of $f\in\mathcal{F}$ in its $i$th coordinate if $\varepsilon_f(i) = 1$ (resp.\ its Galois-adjoint if $\varepsilon_f(i) = \partial$);
		\item the $n_g$-ary connective $g^\flat_i$ for $0\leq i\leq n_g$, the intended interpretation of which is the left residual of $g\in\mathcal{G}$ in its $i$th coordinate if $\varepsilon_g(i) = 1$ (resp.\ its Galois-adjoint if $\varepsilon_g(i) = \partial$).
		\footnote{The adjoints of the unary connectives $\Box$, $\Diamond$, $\lhd$ and $\rhd$ are denoted $\Diamondblack$, $\blacksquare$, $\blhd$ and $\brhd$, respectively.}
	\end{enumerate}
	We stipulate that
	$f^\sharp_i\in\mathcal{G}^*$ if $\varepsilon_f(i) = 1$, and $f^\sharp_i\in\mathcal{F}^*$ if $\varepsilon_f(i) = \partial$. Dually, $g^\flat_i\in\mathcal{F}^*$ if $\varepsilon_g(i) = 1$, and $g^\flat_i\in\mathcal{G}^*$ if $\varepsilon_g(i) = \partial$. The order-type assigned to the additional connectives is predicated on the order-type of their intended interpretations. That is, for any $f\in \mathcal{F}$ and $g\in\mathcal{G}$,
	\begin{enumerate}
		\item if $\epsilon_f(i) = 1$, then $\epsilon_{f_i^\sharp}(i) = 1$ and $\epsilon_{f_i^\sharp}(j) = (\epsilon_f(j))^\partial$ for any $j\neq i$.
		\item if $\epsilon_f(i) = \partial$, then $\epsilon_{f_i^\sharp}(i) = \partial$ and $\epsilon_{f_i^\sharp}(j) = \epsilon_f(j)$ for any $j\neq i$.
		\item if $\epsilon_g(i) = 1$, then $\epsilon_{g_i^\flat}(i) = 1$ and $\epsilon_{g_i^\flat}(j) = (\epsilon_g(j))^\partial$ for any $j\neq i$.
		\item if $\epsilon_g(i) = \partial$, then $\epsilon_{g_i^\flat}(i) = \partial$ and $\epsilon_{g_i^\flat}(j) = \epsilon_g(j)$ for any $j\neq i$.
	\end{enumerate}
	
	For instance, if $f$ and $g$ are binary connectives such that $\varepsilon_f = (1, \partial)$ and $\varepsilon_g = (\partial, 1)$, then $\varepsilon_{f^\sharp_1} = (1, 1)$, $\varepsilon_{f^\sharp_2} = (1, \partial)$, $\varepsilon_{g^\flat_1} = (\partial, 1)$ and $\varepsilon_{g^\flat_2} = (1, 1)$.\footnote{Warning: notice that this notation heavily depends from the connective which is taken as primitive, and needs to be carefully adapted to well known cases. For instance, consider the  `fusion' connective $\circ$ (which, when denoted  as $f$, is such that $\varepsilon_f = (1, 1)$). Its residuals
$f_1^\sharp$ and $f_2^\sharp$ are commonly denoted $/$ and
$\backslash$ respectively. However, if $\backslash$ is taken as the primitive connective $g$, then $g_2^\flat$ is $\circ = f$, and
$g_1^\flat(x_1, x_2): = x_2/x_1 = f_1^\sharp (x_2, x_1)$. This example shows
that, when identifying $g_1^\flat$ and $f_1^\sharp$, the conventional order of the coordinates is not preserved, and depends of which connective
is taken as primitive.}

	\begin{definition}
		For any language $\mathcal{L}_\mathrm{LE}(\mathcal{F}, \mathcal{G})$, the {\em basic `tense'} $\mathcal{L}_\mathrm{LE}$-{\em logic} is defined by specializing Definition \ref{def:DLE:logic:general} to the language $\mathcal{L}_\mathrm{LE}^* = \mathcal{L}_\mathrm{LE}(\mathcal{F}^*, \mathcal{G}^*)$ 
		and closing under the following residuation rules for each $f\in \mathcal{F}$ and $g\in \mathcal{G}$:
			$$
			\begin{array}{cc}
			\AxiomC{$f(\varphi_1,\ldots,\phi,\ldots, \varphi_{n_f}) \vdash \psi$}
			\doubleLine
			\LeftLabel{$(\epsilon_f(i) = 1)$}
			\UnaryInfC{$\phi\vdash f^\sharp_i(\varphi_1,\ldots,\psi,\ldots,\varphi_{n_f})$}
			\DisplayProof
			&
			\AxiomC{$\phi \vdash g(\varphi_1,\ldots,\psi,\ldots,\varphi_{n_g})$}
			\doubleLine
			\RightLabel{$(\epsilon_g(i) = 1)$}
			\UnaryInfC{$g^\flat_i(\varphi_1,\ldots, \phi,\ldots, \varphi_{n_g})\vdash \psi$}
			\DisplayProof
			\end{array}
			$$
			$$
			\begin{array}{cc}
			\AxiomC{$f(\varphi_1,\ldots,\phi,\ldots, \varphi_{n_f}) \vdash \psi$}
			\doubleLine
			\LeftLabel{$(\epsilon_f(i) = \partial)$}
			 \UnaryInfC{$f^\sharp_i(\varphi_1,\ldots,\psi,\ldots,\varphi_{n_f})\vdash \phi$}
			\DisplayProof
			&
			\AxiomC{$\phi \vdash g(\varphi_1,\ldots,\psi,\ldots,\varphi_{n_g})$}
			\doubleLine
			\RightLabel{($\epsilon_g(i) = \partial)$}
			\UnaryInfC{$\psi\vdash g^\flat_i(\varphi_1,\ldots, \phi,\ldots, \varphi_{n_g})$}
			\DisplayProof
			\end{array}
			$$
		The double line in each rule above indicates that the rule should be read both top-to-bottom and bottom-to-top.
		Let $\mathbf{L}_\mathrm{LE}^*$ be the minimal  `tense' $\mathcal{L}_\mathrm{LE}$-logic. 
For any language $\mathcal{L}_{\mathrm{LE}}$, by a {\em tense $\mathrm{LE}$-logic} we understand any axiomatic extension of the basic tense  $\mathcal{L}_{\mathrm{LE}}$-logic in $\mathcal{L}^*_{\mathrm{LE}}$.
	\end{definition}
	
	The algebraic semantics of $\mathbf{L}_\mathrm{LE}^*$ is given by the class of `tense' $\mathcal{L}_\mathrm{LE}$-algebras, defined as tuples $\bba = (L, \mathcal{F}^*, \mathcal{G}^*)$ such that $L$ is a lattice, and moreover,
		\begin{enumerate}
			
			\item for every $f\in \mathcal{F}$ s.t.\ $n_f\geq 1$, all $a_1,\ldots,a_{n_f}\in L$ and $b\in L$, and each $1\leq i\leq n_f$,
			\begin{itemize}
				\item
				if $\epsilon_f(i) = 1$, then $f(a_1,\ldots,a_i,\ldots a_{n_f})\leq b$ iff $a_i\leq f^\sharp_i(a_1,\ldots,b,\ldots,a_{n_f})$;
				\item
				if $\epsilon_f(i) = \partial$, then $f(a_1,\ldots,a_i,\ldots a_{n_f})\leq b$ iff $a_i\leq^\partial f^\sharp_i(a_1,\ldots,b,\ldots,a_{n_f})$.
			\end{itemize}
			\item for every $g\in \mathcal{G}$ s.t.\ $n_g\geq 1$, any $a_1,\ldots,a_{n_g}\in D$ and $b\in L$, and each $1\leq i\leq n_g$,
			\begin{itemize}
				\item if $\epsilon_g(i) = 1$, then $b\leq g(a_1,\ldots,a_i,\ldots a_{n_g})$ iff $g^\flat_i(a_1,\ldots,b,\ldots,a_{n_g})\leq a_i$.
				\item
				if $\epsilon_g(i) = \partial$, then $b\leq g(a_1,\ldots,a_i,\ldots a_{n_g})$ iff $g^\flat_i(a_1,\ldots,b,\ldots,a_{n_g})\leq^\partial a_i$.
			\end{itemize}
		\end{enumerate}
		It is also routine to prove using the Lindenbaum-Tarski construction that $\mathbf{L}_\mathrm{LE}^*$ (as well as any of its sound axiomatic extensions) is sound and complete w.r.t.\ the class of  `tense' $\mathcal{L}_\mathrm{LE}$-algebras (w.r.t.\ the suitably defined equational subclass, respectively). 

\subsection{Perfect algebras and canonical extensions}
The way the algebraic and the relational semantics of any classical modal logic are linked to one another is very well known: every Boolean algebra with operators (BAO) can be associated with its ultrafilter frame, and with every Kripke frame is associated its complex algebra. To close this triangle, the J\'onsson-Tarski expansion of Stone representation theorem states that every BAO $\bba$ canonically embeds in the complex algebra of its ultrafilter frame. This complex algebra, which is called the {\em perfect}, or {\em canonical extension} of $\bba$, has several additional properties, both intrinsic to it (for instance, it is a {\em powerset} algebra, and not just an algebra of sets) and also relative to its embedded subalgebra. These properties can be expressed purely algebraically, hence independently of the ultrafilter frame construction, and characterize the canonical extension up to an isomorphism fixing the embedded algebra. Analogously well behaved constructions can be performed also for normal (distributive) lattice expansions ((D)LEs), of which we will not give a full account here. Interestingly, whereas the counterparts of the ultrafilter frames look rather different from their Boolean versions, the canonical extension of an LE is defined exactly as the one of a BAO, and the definition is based on the following:
\begin{definition}
Let $\bba$ be a (bounded) sublattice of a complete lattice $\bba'$.
\begin{enumerate}
\item  $\bba$ is {\em dense} in $\bba'$ if every element of $\bba'$ can be expressed both as a join of meets and as
a meet of joins of elements from $\bba$.
\item $\bba$ is {\em compact} in $\bba'$ if, for all $S, T \subseteq \bba'$, if $\bigvee S\leq \bigwedge T$ then $\bigvee S'\leq \bigwedge T'$ for some finite $S'\subseteq S$ and $T'\subseteq T$.
\item The {\em canonical extension} of a lattice $\bba$ is a complete lattice $\bbas$ containing $\bba$
as a dense and compact sublattice.
\end{enumerate}
\end{definition}

Given a lattice $\bba$, its canonical extension, besides being unique up to an isomorphism fixing $\bba$, always exists \cite[Propositions 2.6 and 2.7]{GH01}.
Early on, we mentioned the   intrinsic special properties of canonical extensions. Just like the canonical extension of a Boolean algebra (BA) can be shown to be a {\em perfect} BA (i.e.\ a BA  isomorphic to the powerset algebra of some set), the canonical extension of any lattice is a perfect lattice \cite[Corollary 2.10]{DGP}:
\begin{definition}
\label{def: perfect lattice}
A  lattice $\bba$ is {\em perfect} if 
$\bba$ is complete, and is both completely join-generated by the set $\jira$ of the completely
join-irreducible elements of $\bba$, and completely meet-generated by the set $\mira$ of
the completely meet-irreducible elements of $\bba$.
\end{definition}

The density implies that $\jir$ is contained in the meet closure $\kbbas$ of $\bba$ in $\bbas$ and that $\mir$ is contained in the join closure $\obbas$ of $\bba$ in $\bbas$ \cite{DGP}.\label{Page:JIr:Clsd:MIr:Opn} The elements of $\kbbas$ are referred to as {\em closed} elements, and elements of $\obbas$ as {\em open} elements.
The canonical extension of an LE $\bba$ will be defined as a suitable expansion of the canonical extension of the underlying lattice of $\bba$.
Before turning to this definition, recall that
taking the canonical extension of a lattice  commutes with
taking order-duals and products, namely:
${(\bba^\partial)}^\delta = {(\bbas)}^\partial$ and ${(\bba_1\times \bba_2)}^\delta = \bba_1^\delta\times \bba_2^\delta$ (cf.\ \cite[Theorem 2.8]{DGP}).
Hence,  ${(\bba^\partial)}^\delta$ can be  identified with ${(\bbas)}^\partial$,  ${(\bba^n)}^\delta$ with ${(\bbas)}^n$, and
${(\bba^\varepsilon)}^\delta$ with ${(\bbas)}^\varepsilon$ for any order type $\varepsilon$. Thanks to these identifications,
in order to extend operations of any arity which are monotone or antitone in each coordinate from a lattice $\bba$ to its canonical extension, treating the case
of {\em monotone} and {\em unary} operations suffices:
\begin{definition}
For every unary, order-preserving operation $f : \bba \to \bba$, the $\sigma$-{\em extension} of $f$ is defined firstly by declaring, for every $k\in \kbbas$,
$$f^\sigma(k):= \bigwedge\{ f(a)\mid a\in \bba\mbox{ and } k\leq a\},$$ and then, for every $u\in \bbas$,
$$f^\sigma(u):= \bigvee\{ f^\sigma(k)\mid k\in \kbbas\mbox{ and } k\leq u\}.$$
The $\pi$-{\em extension} of $f$ is defined firstly by declaring, for every $o\in \obbas$,
$$f^\pi(o):= \bigvee\{ f(a)\mid a\in \bba\mbox{ and } a\leq o\},$$ and then, for every $u\in \bbas$,
$$f^\pi(u):= \bigwedge\{ f^\pi(o)\mid o\in \obbas\mbox{ and } u\leq o\}.$$
\end{definition}
It is easy to see that the $\sigma$- and $\pi$-extensions of $\varepsilon$-monotone maps are $\varepsilon$-monotone. More remarkably,
the $\sigma$-extension of a map which sends (finite) joins or meets in the domain to (finite) joins in the
codomain sends {\em arbitrary} joins
or meets in the domain to {\em arbitrary} joins in the codomain. Dually, the $\pi$-extension of a map which sends (finite) joins or meets in the domain to (finite) meets in the
codomain sends {\em arbitrary} joins
or meets in the domain to {\em arbitrary} meets in the codomain.
Therefore, depending on the properties of the original operation, it is more convenient to use one or the other extension. This justifies the following
\begin{definition}
The canonical extension of an
$\mathcal{L}_\mathrm{LE}$-algebra $\bbA = (L, \mathcal{F}^\bbA, \mathcal{G}^\bbA)$ is the   $\mathcal{L}_\mathrm{LE}$-algebra
$\bbA^\delta: = (L^\delta, \mathcal{F}^{\bbA^\delta}, \mathcal{G}^{\bbA^\delta})$ such that $f^{\bbA^\delta}$ and $g^{\bbA^\delta}$ are defined as the
$\sigma$-extension of $f^{\bbA}$ and as the $\pi$-extension of $g^{\bbA}$ respectively, for all $f\in \mathcal{F}$ and $g\in \mathcal{G}$.
\end{definition}
The canonical extension of an LE $\bba$ can be shown to be a {\em perfect} LE:
\begin{definition}
\label{def:perfect LE}
An LE $\bbA = (L, \mathcal{F}^\bbA, \mathcal{G}^\bbA)$ is perfect if $L$ is a perfect lattice (cf.\ Definition \ref{def: perfect lattice}), and moreover the following infinitary distribution laws are satisfied for each $f\in \mathcal{F}$, $g\in \mathcal{G}$, $1\leq i\leq n_f$ and $1\leq j\leq n_g$: for every $S\subseteq L$,
\begin{center}
\begin{tabular}{c c }
$f(x_1,\ldots, \bigvee S, \ldots, x_{n_f}) =\bigvee \{ f(x_1,\ldots, x, \ldots, x_{n_f}) \mid x\in S \}$  & if $\varepsilon_f(i) = 1$\\

$f(x_1,\ldots, \bigwedge S, \ldots, x_{n_f}) =\bigvee \{ f(x_1,\ldots, x, \ldots, x_{n_f}) \mid x\in S \}$  & if $\varepsilon_f(i) = \partial$\\

$g(x_1,\ldots, \bigwedge S, \ldots, x_{n_g}) =\bigwedge \{ g(x_1,\ldots, x, \ldots, x_{n_g}) \mid x\in S \}$  & if $\varepsilon_g(i) = 1$\\

$g(x_1,\ldots, \bigvee S, \ldots, x_{n_g}) =\bigwedge \{ g(x_1,\ldots, x, \ldots, x_{n_g}) \mid x\in S \}$  & if $\varepsilon_g(i) = \partial$.\\

\end{tabular}
\end{center}

\end{definition}
Before finishing the present subsection, let us spell out and further simplify the definitions of the extended operations.
First of all, we recall that taking the order-dual interchanges closed and open elements:
$K({(\bbas)}^\partial) = O(\bbas)$ and $O({(\bbas)}^\partial) =\kbbas$;  similarly, $K({(\bba^n)}^\delta) =\kbbas^n$, and $O({(\bba^n)}^\delta) =\obbas^n$. Hence,  $K({(\bbas)}^\epsilon) =\prod_i K(\bbas)^{\epsilon(i)}$ and $O({(\bbas)}^\epsilon) =\prod_i O(\bbas)^{\epsilon(i)}$ for every LE $\bba$ and every order-type $\epsilon$ on any $n\in \mathbb{N}$, where
\begin{center}
\begin{tabular}{cc}
$K(\bbas)^{\epsilon(i)}: =\begin{cases}
K(\bbas) & \mbox{if } \epsilon(i) = 1\\
O(\bbas) & \mbox{if } \epsilon(i) = \partial\\
\end{cases}
$ &
$O(\bbas)^{\epsilon(i)}: =\begin{cases}
O(\bbas) & \mbox{if } \epsilon(i) = 1\\
K(\bbas) & \mbox{if } \epsilon(i) = \partial.\\
\end{cases}
$\\
\end{tabular}
\end{center}
Denoting by $\leq^\epsilon$ the product order on $(\bbas)^\epsilon$, we have for every $f\in \mathcal{F}$, $g\in \mathcal{G}$,  $\overline{u}\in (\bbas)^{n_f}$ and $\overline{v}\in (\bbas)^{n_g}$,
\begin{center}
\begin{tabular}{l l}
$f^\sigma (\overline{k}):= \bigwedge\{ f( \overline{a})\mid \overline{a}\in (\bbas)^{\epsilon_f}\mbox{ and } \overline{k}\leq^{\epsilon_f} \overline{a}\}$ & $f^\sigma (\overline{u}):= \bigvee\{ f^\sigma( \overline{k})\mid \overline{k}\in K({(\bbas)}^{\epsilon_f})\mbox{ and } \overline{k}\leq^{\epsilon_f} \overline{u}\}$ \\
$g^\pi (\overline{o}):= \bigvee\{ g( \overline{a})\mid \overline{a}\in (\bbas)^{\epsilon_g}\mbox{ and } \overline{a}\leq^{\epsilon_g} \overline{o}\}$ & $g^\pi (\overline{v}):= \bigwedge\{ g^\pi( \overline{o})\mid \overline{o}\in O({(\bbas)}^{\epsilon_g})\mbox{ and } \overline{v}\leq^{\epsilon_g} \overline{o}\}$. \\
\end{tabular}
%
\end{center}

Notice that the algebraic completeness of the logics $\mathbf{L}_\mathrm{LE}$ and $\mathbf{L}_\mathrm{LE}^*$ and the canonical embedding of LEs into their canonical extensions immediately give completeness of $\mathbf{L}_\mathrm{LE}$ and $\mathbf{L}_\mathrm{LE}^*$ w.r.t.\ the appropriate class of perfect LEs.

\subsection{The expanded language of perfect LEs}\label{Subsec:Expanded:Land}

The enhanced environment of perfect LEs has two, strictly related features which will be critical for the development of correspondence theory. Firstly, since perfect LEs are in particular complete lattices, and since the  operations of a perfect LE satisfy the additional infinitary distribution laws, 
by general and well known order-theoretic facts, these operations have (coordinatewise) {\em adjoints}. Namely, any perfect $\mathcal{L}_\mathrm{LE}$-algebra is endowed with the additional structure needed to support the interpretation of  the  language $\mathcal{L}_\mathrm{LE}^*$ (cf.\ Section \ref{ssec:expanded tense language}).

\noindent Secondly, in a perfect LE $\bba$, the sets $\jira$ and $\mira$ play the same  role as the set of {\em atoms}
in a perfect BAO. Namely, the elements of $\jira$ (resp.\ $\mira$) join-generate (resp.\ meet-generate) $\bba$, and form the set of {\em states} (resp.\ {\em co-states}) of its associated relational structure\footnote{Restricted to the Boolean setting, the relational structure associated with a perfect LO $(\bba, \Diamond)$ can be identified with its  {\em atom-structure} $\bba_+$, defined as the Kripke structure $(At(\bba), R)$, such that $At(\bba)$ is the set of atoms of $\bba$, and  $b Ra$ iff $b\leq \Diamond a$ for every $a, b\in At(\bba)$.}
$\bba_+$ (the construction of which will be outlined in the next subsection).  In the Boolean and distributive setting, the role of the completely meet-irreducible elements is typically left implicit because then $\mira$ is order-isomorphic to $\jira$. However, in the setting of  general lattices, the isomorphism between  $\jira$ and $\mira$ falls apart and so both families are essential.

The expanded language of perfect LEs will include the connectives corresponding to all the
adjoint operations, as well as a denumerably infinite set of sorted variables $\mathsf{NOM}$ called {\em nominals},
ranging over the completely join-irreducible elements of perfect LEs, and a denumerably infinite set of
sorted variables $\mathsf{CO\text{-}NOM}$, called {\em co-nominals}, ranging over the completely meet-irreducible elements of
perfect LEs. The elements of $\mathsf{NOM}$ will be denoted with with $\nomi, \nomj$, possibly indexed, and those of
$\mathsf{CO\text{-}NOM}$ with $\cnomm, \cnomn$, possibly indexed.

Let us introduce the expanded language formally: the \emph{formulas} $\phi$ of $\mathcal{L}_\mathrm{LE}^{+}$ are given by the following recursive definition:
\begin{center}
\begin{tabular}{r c |c|c|c|c|c|c c c c c c c}
$\phi ::= $ &$\nomj$ & $\cnomm$ & $\psi$ & $\phi\wedge\phi$ & $\phi\vee\phi$ & $f(\overline{\phi})$ &$g(\overline{\phi})$
\end{tabular}
\end{center}
with $\psi  \in \mathcal{L}_\mathrm{LE}$, $\nomj \in \mathsf{NOM}$ and $\cnomm \in \mathsf{CO\text{-}NOM}$,  $f\in \mathcal{F}^*$ and $g\in \mathcal{G}^*$.

As in the case of $\mathcal{L}_\mathrm{LE}$, we can form inequalities and quasi-inequalities based on $\mathcal{L}_\mathrm{LE}^{+}$. Let $\mathcal{L}_\mathrm{LE}^{+\leq}$ and $\mathcal{L}_\mathrm{LE}^{+\mathit{quasi}}$ respectively denote the set of inequalities between terms in $\mathcal{L}_\mathrm{LE}^+$, and  the set of quasi-inequalities formed out of $\mathcal{L}_\mathrm{LE}^{+\leq}$.
Members of $\mathcal{L}_\mathrm{LE}^+$, $\mathcal{L}_\mathrm{LE}^{+\leq}$, and $\mathcal{L}_\mathrm{LE}^{+\mathit{quasi}}$  not containing any propositional variables (but possibly containing nominals and co-nominals) will be called \emph{pure}.

\noindent Summing up, we will be working with six sets of syntactic objects, as reported in the following table:

\begin{center}
\begin{tabular}{| c || l | l |}
\hline &Base language & Expanded Language\\
\hline \hline Formulas / terms &${\mathcal{L}_\mathrm{LE}}$ &${\mathcal{L}_\mathrm{LE}^+}$\\
\hline Inequalities &$\mathcal{L}_\mathrm{LE}^{\leq}$ &$\mathcal{L}_\mathrm{LE}^{+\leq}$\\
\hline Quasi-inequalities &$\mathcal{L}_\mathrm{LE}^{\mathit{quasi}}$ &$\mathcal{L}_\mathrm{LE}^{+\mathit{quasi}}$\\
\hline
\end{tabular}
\end{center}

If $\mathbb{A}$ is a perfect LE,  then an \emph{assignment} for ${\mathcal{L}_\mathrm{LE}^+}$ on $\mathbb{A}$ is a map $v: \mathsf{PROP} \cup \mathsf{NOM} \cup \mathsf{CO\mbox{-}NOM} \rightarrow \bbA$ sending propositional variables to elements of $\mathbb{A}$, sending nominals to $\jty(\mathbb{A})$ and co-nominals to $\mty(\mathbb{A})$. For any LE $\mathbb{A}$, an \emph{admissible assignment}\label{admissible:assignment} for ${\mathcal{L}_\mathrm{LE}^+}$  on $\mathbb{A}$ is an assignment $v$ for ${\mathcal{L}_\mathrm{LE}^+}$ on $\mathbb{A}^{\sigma}$, such that $v(p) \in \bba$ for each $p \in \mathsf{PROP}$. In other words, the assignment $v$ sends propositional variables to elements of the subalgebra $\bba$, while nominals and co-nominals get sent to the completely join-irreducible and the completely meet-irreducible elements of $\bbas$, respectively. This means that the value of ${\mathcal{L}_\mathrm{LE}}$-terms under an admissible assignment will belong to $\bba$, whereas ${\mathcal{L}_\mathrm{LE}^+}$-terms in general will not.

\section{Modular correspondence: two case studies in relational semantics}
\label{subsec: dual frames}
As discussed in \cite{UnifiedCor}, the unified correspondence approach bases (the soundness of) the core algorithmic computation on  algebras rather than on state-based, relational structures. Specifically, the algorithmic reductions will take place in the language $\mathcal{L}_\mathrm{LE}^+$, which as mentioned earlier is naturally interpreted on perfect LEs. Successful outputs, in the form of  pure $\mathcal{L}_\mathrm{LE}^+$-sentences,  can be further translated into the first-order language of the state-based structures associated  with perfect LEs via some duality. For LE-logics, this modular approach  proves particularly advantageous. Indeed, there is no uniquely established  relational semantics for LE-logics, as there is for logics canonically associated with Boolean expansions, which uniformly accounts for all signatures and for all axiomatic extensions in any given signature. On the contrary,  various alternative proposals  of state-based semantics exist in the literature for the same LE-logic, or for different axiomatic extensions of a given LE-logic (cf.\ e.g.\ \cite{dunn1991gaggle}). The unified correspondence approach, being inherently algebraic,  makes it possible to introduce a neat division of labour between (a) {\em reduction} to pure  $\mathcal{L}_\mathrm{LE}^+$-quasi-inequalities, and  (b) {\em translation} into the correspondence language appropriate to each type of relational structure. While the latter stage depends on the different state-based settings, the former does not, and hence can be performed once and for all state-based settings. Namely, if an LE-inequality can be equivalently transformed by ALBA into the conjunction of pure $\mathcal{L}_\mathrm{LE}^+$-quasi inequalities, the same ALBA output can be then interpreted in the correspondence language associated with each different state-based setting. The present paper  mainly focuses on the reduction stage. However, in the present section, we will outline two different proposals of state-based semantics for the specific LE-logic LML, defined by instantiating $\mathcal{F}: = \{\Diamond, {\lhd}, \circ\}$, with $n_\Diamond = n_{\lhd} = 1$ and $n_{\circ} = 2$ and $\mathcal{G}: = \{\Box, {\rhd}, \star\}$, with $n_\Box = n_{\rhd} = 1$ and $n_{\star} = 2$, and $\epsilon_\Diamond = \epsilon_\Box = 1$, $\epsilon_\lhd = \epsilon_\rhd = \partial$ and $\epsilon_\circ = \epsilon_\star = (1, 1)$. For each state-based setting, we will introduce its  relative standard translation, and provide examples of translations of  the ALBA-outputs of well known axioms  into first-order conditions in each  semantic setting. None of the material in the present section will be used  in following sections, which can hence be read independently.

\subsection{RS-frames}
RS-frames are the first type of state-based semantics  we are going to report on. They are based on structures, referred to as {\em RS-polarities}, which are closely related to the structures dual to general lattices in Urquhart and Hartonas' dualities. In this context, the duality between perfect lattices and RS-frames  can be understood as the `discrete case' of \cite{hartonas1997stone}, in the same way that the duality between perfect distributive lattices and posets is the `discrete Priestley duality'.  
The most prominent feature of these structures  is that they are based not on one but on {\em two} domains, each of which providing a natural interpretation for a first-order language in which individual variables  come in {\em two sorts}. Indeed, the relational structures associated with perfect LEs are 
based on {\em polarities} (cf.\ Definition \ref{def:polarity} below). 


\begin{definition}
\label{def:polarity} A {\em polarity} is a triple $(X, Y,R)$ where $X$ and $Y$ are non-empty sets and $R \subseteq X \times Y$
is a binary relation.
\end{definition}
The relation $R$ induces the following ``specialization'' preorders  on both $X$ and $Y$: for all $x, x'\in X$ and $y, y'\in Y$,
$$x\leq x' \mbox{ iff } R[x']\subseteq R[x] \quad \mbox{ and }\quad y\leq y' \mbox{ iff } R^{-1}[y]\subseteq R^{-1}[y'].$$
For any perfect lattice $\bba$, let $\bba_+: = (\jira, \mira, \leq)$ be its associated polarity, where $\leq$ is the restriction of the lattice order to $\jira\times \mira.$ Conversely, a complete lattice $\mathcal{G}(P)$ can be associated with any  polarity $(X, Y, R)$, arising as the complete sub $\bigwedge$-semilattice of $(\mathcal{P}(X), \bigcap)$,
consisting of the Galois-stable\footnote{The {\em Galois-stable} elements of a Galois connection $(u: P\to Q; \ell: Q\to P)$ are those $x\in P$ s.t.\ $\ell(u(x)) = x$ and those $y\in Q$ s.t.\ $u(\ell(y)) = y$.} elements of the Galois connection $(u: \mathcal{P}(X)\to\mathcal{P}(Y); \ell: \mathcal{P}(Y)\to\mathcal{P}(X))$, defined as follows: for every $U\subseteq X, V\subseteq Y$, $$u(U) = \bigcap \{R[x]\mid x\in U\} = \{y\mid \forall x[x\in U\rightarrow xRy]\} \quad \ell(V) = \bigcap \{R^{-1}[y]\mid y\in V\} = \{x\mid \forall y[y\in V\rightarrow xRy]\},$$
or as the dual of the complete sub $\bigwedge$-semilattice of $(\mathcal{P}(Y), \bigcap)$, consisting of the Galois-stable elements of the same Galois connection. However, it is well known that {\em any} complete lattice (and not just the perfect ones) is isomorphic to one arising from some polarity. In \cite{DGP,Ge06}, the polarities dually corresponding to perfect lattices have been characterized as follows (here we report on the presentation in \cite{Ge06}):
\begin{definition}
A polarity $(X, Y, R)$ is:
\begin{enumerate}
\item {\em separating} if the following conditions are satisfied:
\begin{enumerate}
\item for all $x, x'\in X$, if $x\neq x'$ then $R[x]\neq R[x']$, and
\item for all $y, y'\in Y$, if $y\neq y'$ then $R^{-1}[y]\neq R^{-1}[y']$.
\end{enumerate}
\item {\em reduced} if the following conditions are satisfied:
\begin{enumerate}
\item for every $x\in X$, some $y\in Y$ exists s.t.\ $x$ is a minimal element in $\{x'\in X\mid (x', y)\notin R \}$.
\item for every $y\in Y$, some $x\in X$ exists s.t.\ $y$ is a maximal element in $\{y'\in Y\mid (x, y')\notin R\}$.
\end{enumerate}
\item an {\em RS-polarity}\footnote{In \cite{Ge06}, RS-polarities are referred to as RS-frames. Here we reserve the term RS-frame for RS-polarities endowed with extra relations used to interpret the operations of the lattice expansion. } if it is separating and reduced.
\end{enumerate}
\end{definition}

\subsubsection{From algebraic to RS-semantics} In the Boolean setting and some distributive settings (e.g.\ that of intuitionistic logic), duality bridges between pre-existing and independently established algebraic and state-based semantics for a given logic. In the present subsection, we discuss how RS-semantics  for LML (i.e.\ both the structures referred to as RS-frames for LML and the interpretation of LML-formulas in those structures) can be defined from  perfect LML-algebras as algebraic models of LML, via the duality between perfect lattices and RS-polarities. This illustrates a strategy that is particularly useful for obtaining state-based semantics for arbitrary LE-logics in a uniform and modular way.

The  specification of the RS-semantics of LML hinges on the  dual characterization of any homomorphic assignment $v: \mathrm{LML} \rightarrow \Cc$, where $\Cc$ is any perfect LML-algebra,  as a pair of relations $(\Vdash, \succ)$ such that $\Vdash\ \subseteq\  \jty(\Cc)\times \mathrm{LML}$ and  $\succ\ \subseteq\ \mty(\Cc)\times \mathrm{LML}$.
Before expanding on how the relations $\Vdash$ and  $\succ$ are obtained, let us briefly recall the better known converse direction, from $\Vdash$ to homomorphic assignments, in the Boolean and distributive settings.

In the Boolean and distributive settings, for any given modal propositional language $\mathrm{ML}$,  any Kripke structure $\mathbb{F}$ for ML, and any satisfaction relation $\Vdash \ \subseteq \ W\times \mathrm{ML}$ between states of $\mathbb{F}$ and formulas, an interpretation $\overline{v}: \mathrm{ML}\to \mathbb{F}^+$ can be defined, which is an $\mathrm{ML}$-homomorphism, and is obtained as the unique homomorphic extension of the equivalent functional representation of the relation $\Vdash$ as a map $v: \mathsf{PROP}\to \mathbb{F}^+$, defined as $v(p) = {\Vdash}^{-1}[p]$\footnote{Notice that in order for this equivalent functional representation to be well defined, we need to assume that the relation $\Vdash$ is $\mathbb{F}^+$-{\em compatible}, i.e.\ that ${\Vdash}^{-1}[p]\in \mathbb{F}^+$ for every $p\in \mathsf{PROP}$. In the Boolean case, every relation from $W$ to $\mathrm{LML}$ is clearly $\mathbb{F}^+$-compatible, but already in the distributive case this is not so: indeed ${\Vdash}^{-1}[p]$ needs to be an upward- or downward-closed subset of $\mathbb{F}$. This gives rise to the persistency condition, e.g.\ in the relational semantics of intuitionistic logic.}. In this way, interpretations can be derived from satisfaction relations, so that for every $x\in \jty(\mathbb{F}^+)$ and every formula $\phi$, \begin{equation}\label{eq: desiderata satisfaction} x\Vdash \phi\quad \mbox{ iff }\quad x\leq \overline{v}(\phi),\end{equation}

\noindent where, on the left-hand side, $x\in \jty(\mathbb{F}^+)$ is identified with a state of $\mathbb{F}$ via the isomorphism $\mathbb{F}\cong (\mathbb{F}^+)_+$. 
Conversely, given the perfect LML-algebra $\Cc$ as the complex algebra $\mathbb{F}^+$ of some relational structure $\mathbb{F}$ based on an RS-polarity, and given an interpretation $v: \mathrm{LML}\to \mathbb{F}^+$, the satisfaction relation $\Vdash$ on $\mathbb{F}$ we aim to define should satisfy the condition (\ref{eq: desiderata satisfaction}).



First of all, let us recall  how  the usual satisfaction relation clauses can be retrieved from the algebraic interpretation in the Boolean and distributive case.

Suppose for instance that our signature $\mathrm{ML}$ contains a unary diamond $\Diamond$.
Since $\overline{v}: \mathrm{ML}\to \mathbb{F}^+$ is a homomorphism,  $\overline{v}(\Diamond \phi) = \Diamond^{\mathbb{F}^+}\overline{v}(\phi)$. By the general semantics of distributive lattice-based modal logics (cf.\ e.g.\ \cite{ALBAPaper}), it holds that $\mathbb{F}^+ = \p^{\uparrow}(W)$, that  $\jty(\mathbb{F}^+) = \{x{\uparrow}\mid x\in W\}$, and that $\Diamond^{\mathbb{F}^+}u = R^{-1}[u]$ for every $u\in \mathbb{F}^+$.

The satisfaction relation $\Vdash_{\overline{v}}$ needs to be defined inductively, taking equation (\ref{eq: desiderata satisfaction}) into account. 
As for the basic step,  for every $x\in \jty(\mathbb{F}^+)$ and every $p\in \mathsf{PROP} \cup \{\top, \bot \}$, we define \begin{equation}\label{eq: basic case} x\Vdash_{\overline{v}} p\quad \mbox{ iff }\quad x\leq \overline{v}(p).\end{equation}

For the inductive step, suppose that $\phi = \Diamond \psi$ and   (\ref{eq: desiderata satisfaction}) holds for any $\psi$ of strictly lower complexity than $\phi$. Since $\mathbb{F}^+$ is a perfect distributive lattice and $\overline{v}(\Diamond\psi)\in \mathbb{F}^+$, we get $\overline{v}(\psi) = \bigvee\{x'\in \jty(\mathbb{F}^+)\mid x'\leq \overline{v}(\psi)\} = \bigvee\{x'\in \jty(\mathbb{F}^+)\mid x'\Vdash \psi\}$. Since by assumption $\overline{v}$ is a homomorphism, $\overline{v}(\Diamond \psi) = \Diamond^{\mathbb{F}^+}\overline{v}(\psi) = \Diamond^{\mathbb{F}^+}(\bigvee\{x'\in \jty(\mathbb{F}^+)\mid x'\Vdash \psi\})$, and since $\Diamond^{\mathbb{F}^+}$ is completely join-preserving, we get:

$$\overline{v}(\Diamond \psi) = \bigvee\{\Diamond^{\mathbb{F}^+}x'\mid x'\in \jty(\mathbb{F}^+)\mbox{ and } x'\Vdash \psi\}.$$

\noindent Hence, for any $x\in \jty(\mathbb{F}^+)$ (recall that $x$ is completely join-prime, and we are identifying $\mathbb{F}$ and $ (\mathbb{F}^+)_+$),
\begin{center}
\begin{tabular}{r c l l}
$x\leq \overline{v}(\Diamond\psi)$& iff & $x\leq \bigvee\{\Diamond^{\mathbb{F}^+}x'\mid x'\in \jty(\mathbb{F}^+)\mbox{ and } x'\Vdash \psi\}$ &\\
& iff & $\exists x'( x'\Vdash \psi\ \&\ x\leq \Diamond^{\mathbb{F}^+}x')$ & ($x$ completely join-prime)\\
& iff & $\exists x'( x'\Vdash \psi\ \&\ x{\uparrow}\subseteq R^{-1}[x'])$ & ($\mathbb{F}\cong (\mathbb{F}^+)_+$)\\
& iff & $\exists x'( x'\Vdash \psi\ \&\ x\in R^{-1}[x'])$ & ($R^{-1}[x'] = \Diamond^{\mathbb{F}^+}x' \in \p^{\uparrow}(W)$)\\
& iff & $\exists x'( x'\Vdash \psi\ \&\ x Rx')$.\\
\end{tabular}
\end{center}

\noindent The chain of equivalences above is an instance of a dual characterization argument. By this argument, in the distributive and Boolean setting, the usual satisfaction clause for $\Diamond$-formulas has been recovered from a given algebraic interpretation. This satisfaction clause can then form part of the inductive definition of the corresponding satisfaction relation. The same can be done in the general lattice case. However, the chain of equivalences above breaks down in the second step, since $x\in \jty(\mathbb{F}^+)$ are not  in general completely join-{\em prime} anymore, but only completely join-irreducible.  However, we can obtain a reduction also in this case, by crucially making use of the elements $y\in \mty(\mathbb{F}^+)$:

\begin{center}
\begin{tabular}{r c l }
$x\leq \overline{v}(\Diamond\psi)$& iff & $x\leq \bigwedge\{y\in \mty(\mathbb{F}^+)\mid  \overline{v}(\Diamond\psi)\leq y\}$\\
& iff & $\forall y[ \overline{v}(\Diamond\psi)\leq y\ \Rightarrow\  x\leq y]$\\
& iff & $\forall y[ \bigvee\{\Diamond^{\mathbb{F}^+}x'\mid x'\in \jty(\mathbb{F}^+)\mbox{ and } x'\leq \overline{v}(\psi)\} \leq y\ \Rightarrow\ x\leq y]$\\
& iff & $\forall y[\forall  x'[x'\leq \overline{v}(\psi)\ \Rightarrow\ \Diamond^{\mathbb{F}^+}x'\leq y]\ \Rightarrow\ x\leq y]$\\
& iff & $\forall y[\forall  x'[x'\Vdash\psi\ \Rightarrow\ y R_{\Diamond}x']\ \Rightarrow\ x\leq y]$,\\
\end{tabular}
\end{center}
\noindent where we take  $\Diamond^{\mathbb{F}^+}x\leq y$ as the definition of $y R_{\Diamond}x$. By the argument above we were able to dually characterize the condition $x\leq \overline{v}(\Diamond\psi)$ also in the perfect lattice setting, and hence we can take the last line displayed above as the definition of $x\Vdash \Diamond \psi$:
\begin{equation}
\label{eq:satisf diamond}
x\Vdash \Diamond\psi\ \ \mbox{ iff } \ \ \forall y[\forall  x'[x'\Vdash\psi\ \Rightarrow\ y R_{\Diamond}x']\ \Rightarrow\ x\leq y].\end{equation}

\noindent The calculation above also provides  the definition of the relation $R_{\Diamond}$ which is dual to the given $\Diamond$ operation.
Moreover, it  provides the inductive step for $\Diamond$-formulas in the recursive definition of the relation $\succ\ \subseteq \mty(\mathbb{F}^+)\times \mathrm{LML}$. In the Boolean and distributive settings, $\succ$ is completely determined by $\Vdash$, and is hence not mentioned explicitly there. Indeed, since every $\overline{v}(\psi)$ is an element of a perfect lattice, we would want to define  $\succ$ inductively in such a way that for every $y\in \mty(\mathbb{F}^+)$ and every $\psi\in \mathrm{LML}$,
\begin{equation}\label{eq:co-satisf} y\succ \overline{v}(\psi)\quad \mbox{ iff }\quad \overline{v}(\psi)\leq y.\end{equation}

Again, the base case is
\begin{equation}\label{eq:co-satisf basic case} y\succ \overline{v}(p)\quad \mbox{ iff }\quad \overline{v}(p)\leq y.\end{equation}

Specializing the clause above to powerset algebras $\mathcal{P}(W)$, we would have $y \succ_{V} p$ iff $V(p) \leq y$ iff $V(p) \subseteq W/\{ x \}$ for some $x \in W$ iff $\{ x\} \not\subseteq V(p)$ iff $x \notin V(p)$ iff $x \not \Vdash p$, which shows that  the relation $\succ$ can be regarded as an upside-down description of the satisfaction relation $\Vdash$, namely a {\em co-satisfaction}, or {\em refutation}.

As mentioned, the inductive case for $\Diamond$-formulas has already been developed within the chain of equivalences above:

\begin{equation}
\label{eq:co-satisf diamond}
y\succ\Diamond\psi\ \ \mbox{ iff }\ \ \forall  x'[x'\Vdash\psi\ \Rightarrow\ y R_{\Diamond}x'].\end{equation}

Summing up, any interpretation $\overline{v}:\mathrm{LML}\to \Cc$ into a perfect LML-algebra can be dually characterized as a pair $(\Vdash, \succ)$ where $\Vdash\ \subseteq \ \jty(\Cc)\times \mathrm{LML}$ is a satisfaction relation between states and formulas, and  $\succ\ \subseteq \ \mty(\Cc)\times \mathrm{LML}$ is a {\em co-satisfaction relation} between co-states and formulas. This dual characterization is in fact itself a correspondence argument, where moreover the crucial steps are based on the same order-theoretic facts guaranteeing the soundness of ALBA rules.  The RS-semantics for the full LML signature, together with the definition of the relational duals of each operation,  can be derived by  arguments similar to the one given above. The definition of RS-frames for LML is motivated by these arguments and is reported below. The remaining satisfaction and co-satisfaction clauses are collected in Definition \ref{def:RS:model}.

\begin{definition}\label{Def:RSFrame}
An \emph{RS-frame} for LML is a structure $\mathbb{F} = (X,Y,R,\mathcal{R})$ where $(X,Y,R)$ is an RS-polarity, and $\mathcal{R} = (R_{\Diamond}, R_{\Box}, R_{\lhd}, R_{\rhd}, R_{\circ}, R_{\star})$ is a tuple of additional relations (encoding unary and binary modal operations, respectively) such that $R_{\Diamond}\subseteq Y\times X$,  $R_{\Box}\subseteq X\times Y$, $R_{\lhd}\subseteq Y\times Y$, $R_{\rhd}\subseteq X\times X$, $R_{\circ}\subseteq Y\times X\times X$, and $R_{\star}\subseteq X\times Y\times Y$,  satisfying additional compatibility conditions guaranteeing that the operations associated with the relations in $\mathcal{R}$ map Galois-stable sets to Galois-stable sets.
\end{definition}

Because the specifics of the compatibility conditions play no role in the development of the present paper, we will not discuss them. However, every RS-frame is isomorphic to one arising from a perfect LML algebra $\Cc$. In this case, the following relations, arising from the remaining computations, provide a sufficient specification for the purposes of the present development. For all $x, x', x_1, x_2 \in \jty(\Cc)$ and all $y, y', y_1, y_2 \in \mty(\Cc)$,
\begin{center}
\begin{tabular}{rclcrclcrclcr}
$y R_{\Diamond}x$ &iff &$y \geq \Diamond x$ &  &$x R_{\Box} y$ &iff &$x \leq \Box y$ & &$yR_{\lhd}y'$ &iff &$y \geq {\lhd} y'$ & &\\
$x R_{\rhd} x'$ &iff &$x \leq {\rhd} x'$ & &$R_{\circ}(y, x_1, x_2)$ &iff &$y \geq x_1 \circ x_2$ & &$R_{\star}(x, y_1, y_2)$ &iff &$x \leq y_1 \star y_2$ & &
\end{tabular}
\end{center}

Let $\mathrm{LML}^+$ denote the expanded language $\mathcal{L}_{\mathrm{LE}}^+(\mathcal{F}, \mathcal{G})$ (cf.\ Section \ref{Subsec:Expanded:Land}) for $\mathcal{F}$ and $\mathcal{G}$ specified at the beginning of the present section. Clearly, the dual characterization procedure described above applies also to the connectives and variables of  $\mathrm{LML}^+$, which motivates the fact that this language can be interpreted in LML-models as specified below.
\begin{definition}\label{def:RS:model}
An \emph{RS-model} for LML is a tuple $\mathbb{M} = (\mathbb{F}, V)$ where $\mathbb{F}$ is an RS-frame for LML and $V$ maps any $p \in \mathsf{PROP}$ to a pair $V(p) = (P_1, P_2)$ such that $P_1 \subseteq X$, $P_2 \subseteq Y$ and moreover $u(P_1) = P_2$ and $\ell(P_2) = P_1$.

An \emph{RS-model} for LML$^+$ is a tuple $\mathbb{M}$ as above such that the map $V$ is extended to $\mathsf{NOM}$ and $\mathsf{CO}$-$\mathsf{NOM}$ as follows:  any nominal $\nomi$ is assigned to a pair $(V_1(\nomj), V_2(\nomj)) = (\ell(u(x)), u(x))$ for some $x\in X$ (cf.\ notation introduced at the beginning of the present section) and any conominal $\cnomm$ to a pair $(V_1(\cnomm), V_2(\cnomm)) =(\ell(y), u(\ell(y)))$ for some $y\in Y$.
\end{definition}

For every RS-model $\mathbb{M} = (X,Y,R,\mathcal{R},V)$, $x \in X$, $y \in Y$ and any LML$^+$-formula $\phi$, the satisfaction and co-satisfaction relations $\mathbb{M}, x \Vdash \phi$ and $\mathbb{M}, y \succ \phi$ are defined by simultaneous recursion as follows:
\begin{flushleft}
\begin{tabular}{lllllll}
$\mathbb{M}, x \Vdash \bot$ && never & &$\mathbb{M}, y \succ \bot$ && always\\
$\mathbb{M}, x \Vdash \top$ &&always & &$\mathbb{M}, y \succ \top$ &&never\\
$\mathbb{M}, x \Vdash p$ & iff & $x\in V_1(p)$ & &$\mathbb{M}, y \succ p$ & iff & $y\in V_2(p)$\\
$\mathbb{M}, x \Vdash \nomi$ & iff & $x\in V_1(\nomi)$ & &$\mathbb{M}, y \succ \nomi$ & iff & $y\in V_2(\nomi)$\\
$\mathbb{M}, x \Vdash \cnomm$ & iff & $x\in V_1(\cnomm)$ & &$\mathbb{M}, y \succ \cnomm$ & iff & $y\in V_2(\cnomm)$\\
\end{tabular}

\begin{tabular}{m{2cm}lll}
$\mathbb{M}, x \Vdash \phi\wedge \psi$ & iff & $\mathbb{M}, x \Vdash \phi$ and $\mathbb{M}, x \Vdash  \psi$ & \\
$\mathbb{M}, y \succ \phi\wedge \psi$ & iff & for all $x\in X$, if $\mathbb{M}, x \Vdash \phi\wedge \psi$, then $x R y$\\
$\mathbb{M}, x \Vdash \phi\vee \psi$ & iff & for all $y\in Y$, if $\mathbb{M}, y \succ \phi\wedge \psi$, then $x R y$  & \\
$\mathbb{M}, y \succ \phi\vee \psi$ & iff &  $\mathbb{M}, y \succ \phi$ and $\mathbb{M}, y \succ  \psi$ &\\
$\mathbb{M}, x \Vdash \Diamond\phi$ & iff & for all $y\in Y$, if $\mathbb{M}, y \succ \Diamond\phi$, then $x R y$  & \\
$\mathbb{M}, y \succ \Diamond\phi$ & iff &  for all $x\in X$, if $\mathbb{M}, x \Vdash \phi$, then $y R_{\Diamond} x$ \\
$\mathbb{M}, x \Vdash \Box\phi$ & iff & for all $y\in Y$, if $\mathbb{M}, y \succ \phi$, then $x R_{\Box} y$& \\
$\mathbb{M}, y \succ \Box\phi$ & iff & for all $x\in X$, if $\mathbb{M}, x \Vdash \Box\phi$, then $x R y$\\
$\mathbb{M}, x \Vdash \lhd\phi$ & iff & for all $y\in Y$, if $\mathbb{M}, y \succ \lhd\phi$, then $x R y$  & \\
$\mathbb{M}, y \succ \lhd\phi$ & iff &  for all $y'\in Y$, if $\mathbb{M}, y' \succ \phi$, then $y R_{\lhd} y'$ \\
$\mathbb{M}, x \Vdash \rhd\phi$ & iff & for all $x'\in X$, if $\mathbb{M}, x' \Vdash \phi$, then $x R_{\rhd} x'$& \\
$\mathbb{M}, y \succ \rhd\phi$ & iff & for all $x\in X$, if $\mathbb{M}, x \Vdash \rhd\phi$, then $x R y$\\
$\mathbb{M}, x \Vdash \phi \circ \psi$ & iff & for all $y\in Y$, if $\mathbb{M}, y \succ \phi \circ \psi$, then $x R y$  & \\
$\mathbb{M}, y \succ \phi \circ \psi$ & iff  & for all $x_1, x_2\in X$, if $\mathbb{M}, x_1 \Vdash \phi$ and $\mathbb{M}, x_2 \Vdash \psi$, then $R_{\circ}(y, x_1, x_2)$ \\
$\mathbb{M}, x \Vdash \phi \star \psi$ & iff & for all $y_1, y_2\in Y$, if $\mathbb{M}, y_1 \succ \phi$ and $\mathbb{M}, y_2 \succ \psi$, then $R_{\star}(x, y_1, y_2) $& \\
$\mathbb{M}, y \succ \phi \star \psi$ & iff & for all $x\in X$, if $\mathbb{M}, x \Vdash \phi \star \psi$, then $x R y$\\
\end{tabular}

\begin{tabular}{m{2cm}llll}
$\mathbb{M}, x \Vdash \Diamondblack\phi$ & iff & for all $y\in Y$, if $\mathbb{M}, y \succ \Diamondblack\phi$, then $x R y$   &\\
$\mathbb{M}, y \succ \Diamondblack\phi$ & iff &  for all $x\in X$, if $\mathbb{M}, x \Vdash \phi$, then $x R_{\Box} y$ \\
$\mathbb{M}, x \Vdash \blacksquare\phi$ & iff & for all $y\in Y$, if $\mathbb{M}, y \succ \phi$, then $y R_{\Diamond} x$& \\
$\mathbb{M}, y \succ \blacksquare\phi$ & iff & for all $x\in X$, if $\mathbb{M}, x \Vdash \blacksquare\phi$, then $x R y$\\
$\mathbb{M}, x \Vdash {\blacktriangleleft}\phi$ & iff & for all $y\in Y$, if $\mathbb{M}, y \succ {\blacktriangleleft}\phi$, then $x R y$  & \\
$\mathbb{M}, y \succ {\blacktriangleleft}\phi$ & iff &  for all $y'\in Y$, if $\mathbb{M}, y' \succ \phi$, then $y' R_{\lhd} y$ \\
$\mathbb{M}, x \Vdash {\blacktriangleright}\phi$ & iff & for all $x'\in X$, if $\mathbb{M}, x' \Vdash \phi$, then $x' R_{\rhd} x$& \\
$\mathbb{M}, y \succ {\blacktriangleright}\phi$ & iff & for all $x\in X$, if $\mathbb{M}, x \Vdash {\blacktriangleright}\phi$, then $x R y$\\
\end{tabular}

\begin{tabular}{lllll}
$\mathbb{M}, x \Vdash \phi \starfor \psi$ & iff & for all $y\in Y$, if $\mathbb{M}, y \succ \phi \starfor \psi$, then $x R y$  & \\
$\mathbb{M}, y \succ \phi \starfor \psi$ & iff  & for all $x_1\in X$, $y_2\in Y$, if $\mathbb{M}, x_1 \Vdash \phi$ and $\mathbb{M}, y_2 \succ \psi$, then $R_{\star}(x_1, y, y_2)$ \\
$\mathbb{M}, x \Vdash \phi \starback \psi$ & iff & for all $y\in Y$, if $\mathbb{M}, y \succ \phi \starback \psi$, then $x R y$  & \\
$\mathbb{M}, y \succ \phi \starback \psi$ & iff &  for all $y_1\in Y$, $x_2\in X$, if $\mathbb{M}, y_1 \succ \phi$ and $\mathbb{M}, x_2 \Vdash \psi$, then $R_{\star}(x_2, y_1, y)$ \\
$\mathbb{M}, x \Vdash \phi \circfor \psi$ & iff & for all $y_1\in Y$, $x_2\in X$, if $\mathbb{M}, y_1 \succ \phi$ and $\mathbb{M}, x_2 \Vdash \psi$, then $R_{\circ}(y_1, x, x_2) $& \\
$\mathbb{M}, y \succ \phi \circfor \psi$ & iff & for all $x\in X$, if $\mathbb{M}, x \Vdash \phi \circfor \psi$, then $x R y$\\
$\mathbb{M}, x \Vdash \phi \circback \psi$ & iff & for all $x_1\in X$, $y_2\in Y$, if $\mathbb{M}, x_1 \Vdash \phi$ and $\mathbb{M}, y_2 \succ \psi$, then $R_{\circ}(y_2, x_1, x) $& \\
$\mathbb{M}, y \succ \phi \circback \psi$ & iff & for all $x\in X$, if $\mathbb{M}, x \Vdash \phi \circback \psi$, then $x R y$\\
\end{tabular}
\end{flushleft}

An $\mathrm{LML}^{+}$-inequality $\phi \leq \psi$ is \emph{true} in $\mathbb{M}$, denoted $\mathbb{M} \Vdash \phi \leq \psi$, if for all $x \in X$ and all $y \in Y$, if $\mathbb{M}, x \Vdash \phi$ and $\mathbb{M}, y \succ \psi$ then $xRy$. An $\mathrm{LML}^{+}$ $\phi \leq \psi$ is \emph{valid} in $\mathbb{F}$ if it is true in every model based on $\mathbb{F}$.

\subsubsection{Standard translation on RS-frames for $\mathrm{LML}^{+}$}

As in the Boolean case, each RS-model $\mathbb{M}$ for LML can be seen as a two-sorted first-order structure. Accordingly, we define correspondence languages as follows.

Let $L_1$ be the two-sorted first-order language with equality  built over the denumerable and disjoint sets of individual variables $X$ and $Y$, with binary relation symbols $\leq$, $R_{\Diamond}$, $R_{\Box}$, $R_{\lhd}$, $R_{\rhd}$, ternary relation symbols $R_{\circ}$, $R_{\star}$ and two unary predicate symbols $P_1, P_2$ for each propositional variable $p \in \mathsf{PROP}$.\footnote{The intended interpretation links $P_1$ and $P_2$ in the way suggested by the definition of LML-valuations. Indeed, every $p \in \mathsf{PROP}$ is mapped to a pair $(V_1(p), V_2(p))$ of Galois-stable sets as indicated in Definition \ref{def:RS:model}. Accordingly, the interpretation of pairs $(P_1, P_2)$ of predicate symbols is restricted to such pairs of Galois-stable sets, and hence the interpretation of universal second-order quantification is also restricted to range over such sets.}

We will further assume that $L_1$ contains denumerably many individual variables $i, j, \ldots$ corresponding to the nominals $\nomi, \nomj, \ldots \in \mathsf{NOM}$ and $n, m, \ldots$ corresponding to the co-nominals $\cnomn, \cnomm \in \mathsf{CO\text{-}NOM}$. Let $L_0$ be the sub-language which does not contain the unary predicate symbols corresponding to the propositional variables.
Let us now define the \emph{standard translation} of $\mathrm{LML}^+$ into $L_1$ recursively:
\begin{center}
\begin{tabular}{ll}
$\mathrm{ST}_x(\bot) := x \neq x$ & $\mathrm{ST}_y(\bot) := y  = y$ \\
$\mathrm{ST}_x(\top) := x = x$ & $\mathrm{ST}_y(\top) := y \neq y$\\
$\mathrm{ST}_x(p) := P_1(x)$& $\mathrm{ST}_y(p) :=  P_2(y)$\\
$\mathrm{ST}_x(\nomj) := \forall y[j\leq y \rightarrow x\leq y]$& $\mathrm{ST}_y(\nomj) := j \leq y$  \\
$\mathrm{ST}_x(\cnomm) := x \leq m$ & $\mathrm{ST}_y(\cnomm) := \forall x[x\leq m \rightarrow x\leq y]$\\
$\mathrm{ST}_x(\phi \vee \psi) := \forall y[\mathrm{ST}_y(\phi\vee \psi)\rightarrow x\leq y]$& $\mathrm{ST}_y(\phi \vee \psi) := \mathrm{ST}_y(\phi) \wedge \mathrm{ST}_y(\psi)$\\
$\mathrm{ST}_x(\phi \wedge \psi) := \mathrm{ST}_x(\phi) \wedge \mathrm{ST}_x(\psi)$& $\mathrm{ST}_y(\phi \wedge \psi) := \forall x[\mathrm{ST}_x(\phi\wedge \psi)\rightarrow x\leq y]$\\
$\mathrm{ST}_x(\Diamond \phi) := \forall y [\mathrm{ST}_y(\Diamond \phi)\rightarrow x\leq y]$& $\mathrm{ST}_y(\Diamond \phi) := \forall x[\mathrm{ST}_x(\phi)\rightarrow yR_{\Diamond} x]$\\
 $\mathrm{ST}_x(\Box \phi) := \forall y [\mathrm{ST}_y(\phi)  \rightarrow xR_{\Box} y]$ & $\mathrm{ST}_y(\Box \phi) := \forall x [\mathrm{ST}_x(\Box \phi)\rightarrow x\leq y]$ \\
$\mathrm{ST}_x(\lhd \phi) := \forall y [\mathrm{ST}_y(\lhd \phi) \rightarrow x \leq y]$ &  $\mathrm{ST}_y(\lhd \phi) := \forall y' [\mathrm{ST}_{y'}(\phi) \rightarrow y R_{\lhd} y']$\\
$\mathrm{ST}_x(\rhd \phi) := \forall x' [\mathrm{ST}_{x'}(\phi)  \rightarrow xR_{\rhd}x']$ & $\mathrm{ST}_y(\rhd \phi) := \forall x [\mathrm{ST}_x(\rhd \phi) \rightarrow x \leq y]$\\
$\mathrm{ST}_x(\phi \circ \psi) := \forall y [\mathrm{ST}_{y}(\phi \circ \psi)  \rightarrow x \leq y]$ & $\mathrm{ST}_y(\phi \circ \psi) := \forall x_1 \forall x_2 [\mathrm{ST}_{x_1}(\phi) \wedge \mathrm{ST}_{x_2}(\psi) \rightarrow R_{\circ}(y,x_1, x_2)]$\\
$\mathrm{ST}_x(\phi \star \psi) := \forall y_1 \forall y_2 [\mathrm{ST}_{y_1}(\phi) \wedge \mathrm{ST}_{y_1}(\psi)  \rightarrow R_{\star}(x, y_1, y_2)]$ & $\mathrm{ST}_y(\phi \star \psi) := \forall x [\mathrm{ST}_{x}(\phi \star \psi)  x \leq y]$\\
\end{tabular}
\begin{tabular}{ll}
$\mathrm{ST}_x(\Diamondblack \phi) := \forall y [\mathrm{ST}_{y}(\Diamondblack \phi)\rightarrow x \leq y]$& $\mathrm{ST}_y(\Diamondblack \phi) := \forall x[\mathrm{ST}_x(\phi)\rightarrow yR_{\Box}^{-1} x]$\\
$\mathrm{ST}_x({\blacksquare} \phi) := \forall y [\mathrm{ST}_y(\phi)  \rightarrow xR^{-1}_{\Diamond} y]$ & $\mathrm{ST}_y({\blacksquare} \phi) := \forall x [\mathrm{ST}_x({\blacksquare} \phi)\rightarrow x\leq y]$ \\
$\mathrm{ST}_x({\blacktriangleleft} \phi) := \forall y [\mathrm{ST}_y({\blacktriangleleft} \phi) \rightarrow x \leq y]$ &  $\mathrm{ST}_y({\blacktriangleleft} \phi) := \forall y' [\mathrm{ST}_{y'}(\phi) \rightarrow y R^{-1}_{\lhd} y']$\\
$\mathrm{ST}_x({\blacktriangleright} \phi) := \forall x' [\mathrm{ST}_{x'}(\phi)  \rightarrow xR^{-1}_{\rhd}x']$ & $\mathrm{ST}_y({\blacktriangleright} \phi) := \forall x [\mathrm{ST}_x({\blacktriangleright} \phi) \rightarrow x \leq y]$\\
$\mathrm{ST}_x(\phi \circback \psi) := \forall x_1 \forall y_2 [\mathrm{ST}_{x_1}(\phi) \wedge \mathrm{ST}_{y_2}(\psi)  \rightarrow R^{-2}_{\circ}(x, x_1, y_2)]$ & $\mathrm{ST}_y(\phi \circback \psi) := \forall x [\mathrm{ST}_{x}(\phi \circback \psi)  x \leq y]$\\
$\mathrm{ST}_x(\phi \circfor \psi) := \forall y_1 \forall x_2 [\mathrm{ST}_{y_1}(\phi) \wedge \mathrm{ST}_{x_2}(\psi)  \rightarrow R^{-1}_{\circ}(x, y_1, x_2)]$ & $\mathrm{ST}_y(\phi \circfor \psi) := \forall x [\mathrm{ST}_{x}(\phi \circfor \psi)  x \leq y]$\\
$\mathrm{ST}_x(\phi \starback \psi) := \forall y [\mathrm{ST}_{y}(\phi \starback \psi)  \rightarrow x \leq y]$ & $\mathrm{ST}_y(\phi \starback \psi) := \forall y_1 \forall x_2 [\mathrm{ST}_{y_1}(\phi) \wedge \mathrm{ST}_{x_2}(\psi) \rightarrow R^{-2}_{\star}(y,y_1, x_2)]$\\
$\mathrm{ST}_x(\phi \starfor \psi) := \forall y [\mathrm{ST}_{y}(\phi \starfor \psi)  \rightarrow x \leq y]$ & $\mathrm{ST}_y(\phi \starfor \psi) := \forall x_1 \forall y_2 [\mathrm{ST}_{x_1}(\phi) \wedge \mathrm{ST}_{y_2}(\psi) \rightarrow R^{-1}_{\star}(y,x_1, y_2)]$\\
\end{tabular}
\end{center}


Observe that if $\mathrm{ST}_x$ is applied to pure formulas (cf.\ Section \ref{Subsec:Expanded:Land}), it produces formulas in the sublanguage $L_0$. This will be most important for our purposes. The following lemma is proved by a routine induction.

\begin{lemma}\label{lemma:Standard:Translation}
For any RS-model $\mathbb{M} = (X,Y,R,\mathcal{R}, V)$ for LML, any RS-frame $\mathbb{F}$ for LML, , $w \in X$, $v \in Y$ and for all $LML^+$-formulas $\phi$ and $\psi$, the following hold:
\begin{enumerate}

\item $\mathbb{M}, w \Vdash \phi$ iff $\mathbb{M} \models \mathrm{ST}_x(\phi)[x:= w]$

\item $\mathbb{M}, v \succ \psi$ iff $\mathbb{M} \models \mathrm{ST}_y(\psi)[y:= v]$

\item
\begin{eqnarray*}
\mathbb{M} \Vdash \phi \leq \psi &\mbox{iff} &\mathbb{M} \models \forall x \forall y [(\mathrm{ST}_x(\phi) \wedge  \mathrm{ST}_y(\psi)) \rightarrow xRy]\\
&\mbox{iff} &\mathbb{M} \models \forall x  [\mathrm{ST}_x(\phi) \rightarrow  \forall y( \mathrm{ST}_y(\psi) \rightarrow xRy)]\\
&\mbox{iff} &\mathbb{M} \models \forall y [\mathrm{ST}_y(\psi) \rightarrow  \forall x( \mathrm{ST}_x(\phi) \rightarrow xRy)].
\end{eqnarray*}

\item\begin{eqnarray*}
\mathbb{F} \Vdash \phi \leq \psi &\mbox{iff} &\mathbb{F} \models \forall \overline{P} \forall \overline{j} \forall \overline{m}\forall x \forall y [(\mathrm{ST}_x(\phi) \wedge  \mathrm{ST}_y(\psi)) \rightarrow xRy]\\
&\mbox{iff} &\mathbb{F} \models \forall\overline{P} \forall \overline{j} \forall \overline{m} \forall x  [\mathrm{ST}_x(\phi) \rightarrow  \forall y( \mathrm{ST}_y(\psi) \rightarrow xRy)]\\
&\mbox{iff} &\mathbb{F} \models \forall\overline{P} \forall \overline{j} \forall \overline{m} \forall y [\mathrm{ST}_y(\psi) \rightarrow  \forall x( \mathrm{ST}_x(\phi) \rightarrow xRy)].
\end{eqnarray*}
where $\overline{P}$, $\overline{j}$, and $\overline{m}$ are, respectively, the vectors of all predicate symbols corresponding to propositional variables, individual variables corresponding to nominals, and individual variables corresponding to co-nominals, occurring in $\mathrm{ST}_x(\phi)$ and $\mathrm{ST}_y(\psi)$.

\end{enumerate}

\end{lemma}

\begin{example}
Consider the LML-inequality $p\leq \Diamond p$. The algorithm $\mathsf{ALBA}$ introduced in Section \ref{Spec:Alg:Section} produces an equivalent pure quasi-inequality which simplifies to $\nomi \leq \Diamond \nomi$. We illustrate the standard translation by instantiating the statement of Lemma \ref{lemma:Standard:Translation}.4 with $\nomi$ for $\phi$ and $\Diamond \nomi$ for $\psi$:

\begin{center}
\begin{tabular}{c l}
&$\mathbb{F} \Vdash \nomi \leq \Diamond \nomi$\\
iff &$\mathbb{F} \models \forall x \forall y \forall i [(\mathrm{ST}_{x}(\nomi) \wedge (\mathrm{ST}_{y}(\Diamond \nomi)) \rightarrow x \leq y]$\\
iff &$\mathbb{F} \models \forall x \forall y \forall i [(\forall y'[i \leq y' \rightarrow x \leq y'] \wedge \forall x'[\mathrm{ST}_{x'}(\nomi) \rightarrow y R_{\Diamond} x' ]) \rightarrow x \leq y]$\\
iff &$\mathbb{F} \models \forall x \forall y \forall i [(\forall y'[i \leq y' \rightarrow x \leq y'] \wedge \forall x'[\forall y''[i \leq y'' \rightarrow x' \leq y''] \rightarrow y R_{\Diamond} x' ]) \rightarrow x \leq y]$.
\end{tabular}
\end{center}

\comment{
$\quad$\\
\begin{center}
\begin{tabular}{c l}
& $\forall p[p\leq \Diamond p]$\\
iff & $\forall p\forall \nomi\forall \cnomm[(\nomi\leq p\ \&\ \Diamond p\leq \cnomm )\Rightarrow i\leq \cnomm]$\\
iff & $\forall \nomi\forall \cnomm[\Diamond \nomi\leq \cnomm \Rightarrow i\leq \cnomm]$\\
iff & $\forall \nomi[\nomi \leq \Diamond \nomi]$\\
iff & $\forall x[\overline{x} \subseteq \langle R_{\Diamond}\rangle \overline{x}]$\\
iff & $\forall x\forall x'[x'\in \overline{x} \Rightarrow x'\in  \langle R_{\Diamond}\rangle \overline{x}]$\\
iff & $\forall x\forall x'[\forall y[xRy\Rightarrow x'Ry] \Rightarrow
\forall y[\forall  x_1[x_1\in \overline{x}\ \Rightarrow\ y R_{\Diamond}x_1]\ \Rightarrow\ x'\leq y]]$\\
iff & $\forall x\forall x'[\forall y[xRy\Rightarrow x'Ry] \Rightarrow
\forall y[\forall  x_1[\forall y[xRy\Rightarrow x_1Ry] \Rightarrow\ y R_{\Diamond}x_1]\ \Rightarrow\ x'\leq y]]$\\

\end{tabular}
\end{center}
}
The  theory developed in the present paper (cf.\ Theorems  \ref{Crctns:Theorem} and \ref{Thm:ALBA:Canonicity}) guarantees that the LE-logic obtained by adding $p\leq \Diamond p$ to the basic logic $\mathbf{L}_\mathrm{LML}$ (cf.\ Definition \ref{def:DLE:logic:general}) is sound and complete w.r.t.\ the elementary class of RS-frames for LML defined by the sentence above.
\end{example}

\begin{example}
\label{example:linear involution}
The treatment for $\mathrm{LML}$ and its expanded language $\mathrm{LML}^+$ can be applied also to inequalities in the language of the multiplicative-additive fragment of linear logic (MALL).\footnote{In \cite{Gehrke},  MALL is understood as the logic of  a class of normal LEs $(L, \circ, \slash, \backslash, \bot, 1)$ such that $(L, \circ, \slash, \backslash)$ is a residuated lattice, $\circ$ is commutative (hence $\slash$ and $\backslash$ can be identified), associative and with unit $1$, and the identity
$a = a^{\bot\bot}$ holds.} In \cite{Gehrke}, the first-order correspondent of  the inequality $p^{\bot\bot} \leq p$, where ${(\cdot)}^\bot$ is defined as ${(\cdot)}\backslash \bot$, has been computed, together with those of other inequalities concurring to the axiomatization of MALL, by means of a dual characterization argument, although in absence of a general theory identifying Sahlqvist and inductive inequalities across different signatures, or providing a formal definition of standard translation. In this computation, the two stages of reduction and translation mentioned at the beginning of the present section can be clearly recognized. In \cite{Gehrke}, the reduction stage is argued as follows:

\begin{quote}
{\it Furthermore, the mapping [$a \mapsto a^{\bot\bot}$] is completely
join-preserving and therefore it again suffices to consider completely join-irreducible elements.}
\end{quote}
In the notation of the present paper, this corresponds to performing an ALBA-reduction on $p^{\bot\bot} \leq p$ which delivers $\nomi^{\bot\bot} \leq \nomi$. The ensuing translation stage computed in \cite{Gehrke} takes (the notational counterpart of) the  pure inequality $\nomi^{\bot\bot} \leq \nomi$ as input.
The quoted claim above is not explicitly justified, but is a consequence of the following observation, which appears early on in the text \cite{Gehrke}:
\begin{quote}
{\it Implication sends joins in the
first coordinate to meets, hence ${(\cdot)}^\bot$ sends joins to meets. As ${(\cdot)}^\bot$ is a bijection, it follows that it
is a (bijective) lattice homomorphism $L \to L^\partial$, where $L^\partial$ is the lattice obtained by reversing the
order in $L$.}
\end{quote}
If $L$ is a perfect lattice, the fact that ${(\cdot)}^\bot: L \to L^\partial$ is a complete lattice homomorphism implies that it is both completely join- and meet-reversing, from which it follows that the assignment $a \mapsto a^{\bot\bot}$ is indeed completely join-preserving.
Moreover,  not only ${(\cdot)}^{\bot\bot}$ being completely join-preserving follows from the validity of the identity $a = a^{\bot\bot}$, but  also does not follow from the remaining MALL-axioms. Indeed, all the axioms of classical linear algebras (identifying $\circ$ and $\wedge$)  hold on    Heyting algebras but  the identity $a = a^{\bot\bot}$, and the 5-element Heyting algebra with two atoms 
is an example of one in which ${(\cdot)}^\bot$ is not a bijection and  ${(\cdot)}^{\bot\bot}$ is not join-preserving (indeed, letting $a$ be the meet-irreducible element which is not join-irreducible, it is easy to verify that $a^\bot = \bot = \top^\bot$, which implies that ${(\cdot)}^{\bot\bot}$ does not preserve the join of the two join-irreducible elements below $a$). Thus, in their reduction stage for the inequality $p^{\bot\bot} \leq p$, the authors of \cite{Gehrke} rely on a property which depends on the same inequality $p^{\bot\bot} \leq p$. This is different from the common practice of relying on the additional properties provided by other, logically independent axioms. Thus, the first-order formula they obtain is not the equivalent first-order correspondent of the input inequality over the class of frames defined by the first-order correspondents of the remaining axioms, but rather, over a strictly smaller class. Hence, the duality result claimed in  \cite[Theorem 15]{Gehrke} does not hold as stated. This theorem can be easily amended. The amendment hinges on performing an alternative  ALBA-reduction step  of the inequality $p^{\bot\bot} \leq p$ without relying on the additional assumption of $a \mapsto a^{\bot\bot}$ being completely join-preserving on perfect lattices. Such a  reduction, performed in Example \ref{example: reduction involution}, yields $\cnomm^{\bot\bot} \leq \cnomm$. Below, we shall compute the standard translation of this inequality, which provides the required amendment to  \cite[Theorem 15]{Gehrke}. To simplify the computation, we will make use of the fact that $\mathrm{ST}_{y}(\bot)$ is defined as the tautology $y=y$ for every $y$ in $Y$.

\begin{center}
\begin{tabular}{c l}
&$\mathbb{F} \Vdash \cnomm^{\bot\bot} \leq \cnomm$\\
iff &$\mathbb{F} \models \forall x \forall y \forall m [(\mathrm{ST}_{x}(\cnomm^{\bot\bot}) \wedge \mathrm{ST}_{y}(\cnomm)) \rightarrow x \leq y]$\\
iff &$\mathbb{F} \models \forall x \forall y \forall m [(\mathrm{ST}_{x}((\cnomm\backslash \bot)\backslash \bot) \wedge \forall x'[x'\leq m \rightarrow x'\leq y]) \rightarrow x \leq y]$\\
iff &$\mathbb{F} \models \forall x \forall y \forall m [(\forall x_1 \forall y_2 [(\mathrm{ST}_{x_1}(\cnomm\backslash \bot) \wedge \mathrm{ST}_{y_2}(\bot)) \rightarrow R^{-2}_{\circ}(x, x_1, y_2)] \wedge$\\
  & $\quad\quad \forall x'[x'\leq m \rightarrow x'\leq y]) \rightarrow x \leq y]$\\
iff &$\mathbb{F} \models \forall x \forall y \forall m [(\forall x_1 [\mathrm{ST}_{x_1}(\cnomm\backslash \bot)  \rightarrow \forall y_2 [R^{-2}_{\circ}(x, x_1, y_2)]] \wedge$\\
  & $\quad\quad \forall x'[x'\leq m \rightarrow x'\leq y]) \rightarrow x \leq y]$\\
iff &$\mathbb{F} \models \forall x \forall y \forall m [(\forall x_1 [\forall x_1' \forall y_2' [(\mathrm{ST}_{x_1'}(\cnomm) \wedge \mathrm{ST}_{y_2'}(\bot))\rightarrow  R^{-2}_{\circ}(x_1, x_1', y_2')]$\\
& $\quad\quad\rightarrow \forall y_2[R^{-2}_{\circ}(x, x_1, y_2)]] \wedge \forall x'[x'\leq m \rightarrow x'\leq y]) \rightarrow x \leq y]$\\
iff &$\mathbb{F} \models \forall x \forall y \forall m [(\forall x_1 [\forall x_1'  [x_1'\leq m \rightarrow \forall y_2'[R^{-2}_{\circ}(x_1, x_1', y_2')]]$\\
& $\quad\quad \rightarrow \forall y_2 [R^{-2}_{\circ}(x, x_1, y_2)]] \wedge \forall x'[x'\leq m \rightarrow x'\leq y]) \rightarrow x \leq y]$.\\
\end{tabular}
\end{center}

\end{example}

\subsection{TiRS graph semantics}
\label{ssec:TiRS}

TiRS-graphs are the second type of state-based semantics  we are going to report on. They are based on structures, referred to as {\em TiRS-graphs}, which are closely related to the topological structures dual to general lattices in Plo\v{s}\v{c}ica's duality \cite{Pl95}. 
Indeed, the definition of TiRS-graphs describes the structures obtained by forgetting the topology of the dual spaces in  Plo\v{s}\v{c}ica's duality. Hence, TiRS-graphs are the natural candidates for a `discrete Plo\v{s}\v{c}ica's duality', which has been explored in \cite{CGH15}, 
 and which provides an alternative to finite RS-polarities as duals of finite lattices.
Unlike RS-polarities, TiRS-graphs are based  on one  domain:
\begin{definition}
A \emph{TiRS graph} is a reflexive directed graph $\X=(X,E)$ that satisfies the following conditions:

\begin{enumerate}
\item[(S)] for every $x, y \in X$, if $x \neq y$ then $xE\neq yE$ or $Ex \neq Ey$;
\item[(R)]
\begin{itemize}
\item[(i)] for all $x, z \in X$, if $zE\subsetneq xE$ then $(z,x) \notin E$;
\item[(ii)] for all $y, z \in X$, if $Ez \subsetneq Ey$ then $(y,z) \notin E$;
\end{itemize}
\item[(Ti)$'$] for all $x,y \in X$, if $(x,y) \in E$, then there exists $z$ such that
$(x,z)\in E$ and $(z,y)\in E$
and for every $w\in X$, $(z,w)\in E$
implies $(x,w)\in E$
and $(w,z)\in E$
implies $(w,y) \in E$,
\end{enumerate}
where
$xE=\{\, y \in X \mid (x,y) \in E \,\}$\ and $Ex=\{\, y \in X \mid (y,x) \in E  \,\}$.
\end{definition}

We sketch the duality between finite TiRS-graphs and finite lattices, and refer the reader to \cite{CGH15} for a more general treatment. For any finite lattice $\bba$, its associated TiRS-graph is given by $\bba_{\bullet}: = (\mathbb{X}, E)$ where $X = \{ (x,y) \mid x \in \jty(\bba), y \in \mty(\bba) \ \mbox{ and } \ x \not \leq y\}$ and $E \subseteq X \times Y$ such that $(x,y)E(x',y')$ iff $x \nleq y$.

Conversely, any finite TiRS-graph $\X = (X, E)$  defines the RS-polarity $\rho(\X) = (X,X, E^{c})$ where $E^{c}$ is the relative complement of $E$ in $X \times X$.  Hence the lattice $\X^{\bullet}$ associated with $\X$ is the one given by the Galois-stable sets of $\rho(\X)$.
Notice that this correspondence induces a representation of the states $\z$ of a TiRS-graph as tuples $(\z_J, \z_M)\subseteq \jty(\X^\bullet)\times \mty(\X^\bullet)$ such that  $\z_J\nleq\z_M$.
\begin{prop}[\cite{CGH15} Corollary 3.2]\label{Prop:Finite:Tirs:Dual}
For every finite lattice $\bba$ and finite TiRS-graph $\X$ we have
\[
\bbA \cong (\bba_{\bullet})^{\bullet} \quad \mbox{ and } \quad  \X \cong (\X^{\bullet})_{\bullet}.
\]
\end{prop}

In the light of the duality given by Proposition \ref{Prop:Finite:Tirs:Dual}, we can define TiRS-graph semantics for LML. We illustrate this only in the case of $\Diamond$ and $\Box$, and refer the reader to \cite{conradie-craig} for a more complete picture.  As in the previous section, our starting point is an assignment $v: \mathrm{LML} \rightarrow \Cc$ into a finite LML-algebra $\Cc$ the lattice-reduct of which we identify with a TiRS-graph $\X^{\bullet}$ for some finite TiRS-graph $\X$. We aim to dualize the assignment $v$ as a pair $(\Vdash, \succ)$ consisting of a satisfaction (assertion) relation $\Vdash \subseteq X \times \mathrm{LML}$ and a co-satisfaction (denial) relation $\succ \subseteq X \times \mathrm{LML}$. Analogously to the considerations for RS-frames above, the main desiderata of this dualization are that  for every $\z\in X$,

\begin{equation}\label{eq:desiderata:satisfaction:Tirs}
\z \Vdash \phi\quad \mbox{ iff }\quad \z_J\leq \overline{v}(\phi) \quad \mbox{ and } \quad \z \succ \phi\quad \mbox{ iff }\quad \overline{v}(\phi) \leq \z_M.
\end{equation}
The satisfaction and co-satisfaction relations $\Vdash_{\overline{v}}$ and $\succ_{\overline{v}}$ need to be defined inductively, taking the conditions above into account. As for the basic step,  for every $x \in \jty(\Cc)$, $y \in \mty(\Cc)$ and every $p\in \mathsf{PROP} \cup \{\top, \bot \}$, we define
\begin{equation}
\z \Vdash p\quad \mbox{ iff }\quad \z_j \leq \overline{v}(p) \quad \mbox{ and } \quad \z \succ p \quad \mbox{ iff }\quad \overline{v}(p) \leq \z_M.
\end{equation}
For the inductive step, suppose that $\phi = \Box \psi$, and (\ref{eq:desiderata:satisfaction:Tirs}) holds for any $\psi$ of strictly lower complexity than $\phi$. Since $\Cc$ is a finite (and hence perfect) lattice and $\overline{v}(\Box\psi)\in \Cc$, we get $\overline{v}(\psi) = \bigwedge\{y' \in \mty(\Cc) \mid \overline{v}(\psi) \leq y'\} = \bigwedge\{\z_M\in \mty(\Cc)\mid \z \succ \psi\}$.\footnote{Note that $\Cc$ being perfect implies that  for any $y' \in \mty(\Cc)$ there exists at least one $x \in \jty(\Cc)$ such that $(x,y') \in \Cc_{\bullet}$, i.e., $x \nleq y'$. Indeed, otherwise $x = \bigwedge \mty(\Cc) =  \bot$. This implies that for every $y'\in \mty(\Cc)$ there exists at least one $\z\in \X$ such that $y' = \z_M$. Furthermore, if $\z \succ \psi$, then $\z' \succ \psi$ for every $\z'\in \X$ such that $\z'_M = \z_M$. Similar considerations apply to every $x'\in \jty(\Cc)$.} Since by assumption $\overline{v}$ is a LML-homomorphism, $\overline{v}(\Box \psi) = \Box^{\Cc}\overline{v}(\psi) = \Box^{\Cc}(\bigwedge\{\z'_M\in \mty(\Cc)\mid \z' \succ \psi\})$, and since $\Box^{\Cc}$ is (completely) meet-preserving, we get:
\[
\overline{v}(\Box \psi) = \bigwedge\{\Box^{\Cc}\z'_M \mid y'\in \mty(\Cc)\mbox{ and } \z' \succ \psi\}.
\]
Hence, for any $\z \in \X$, we have

\begin{center}
\begin{tabular}{c l}
$\z_J \leq \overline{v}(\Box \psi) $ &iff $\quad\z_J \leq \bigwedge\{\Box^{\Cc}\z'_M \mid y'\in \mty(\Cc)\mbox{ and } \z' \succ \psi\}$\\
&iff  $\quad\forall \z'[\z' \succ \psi \,\Rightarrow\, \z_J \leq \Box^{\Cc} \z'_M]$\\
&iff $\quad\forall \z'[\z_J \nleq \Box^{\Cc} \z'_M\ \Rightarrow \z' \not{\succ} \psi]$\\
&iff $\quad\forall \z'[\z R_{\Box} \z'\ \Rightarrow \z' \not{\succ} \psi],$\\
\end{tabular}
\end{center}
where $R_\Box\subseteq X\times X$ is defined by  $\z R_{\Box} \z'$ iff $\z_J \nleq \Box^{\Cc} \z'_M$. The last line of the displayed chain of equivalences above can be taken as the defining clause of the satisfaction relation for $\Box$-formulas.
As for the co-satisfaction, for any $\z \in \X$, we have

\begin{center}
\begin{tabular}{c l}
$\overline{v}(\Box \psi)\leq\z_M  $ &iff $\bigvee\{\z'_J\in \jty(\Cc)\mid \z'_J\leq \overline{v}(\Box \psi)\}\leq \z_M$\\
&iff $\bigvee\{\z'_J\in \jty(\Cc)\mid  \z'\Vdash \Box \psi\}\leq \z_M$\\
&iff $\forall \z'[ \z'\Vdash \Box \psi\ \Rightarrow\  \z'_J\leq \z_M]$\\
&iff $\forall \z'[\z'_J\nleq \z_M \ \Rightarrow\  \z'\not \Vdash \Box \psi]$\\
&iff $\forall \z'[\z'E\z \ \Rightarrow\  \z'\not \Vdash \Box \psi]$.\\
\end{tabular}
\end{center}

Summing up, any interpretation $\overline{v}:\mathrm{LML}\to \Cc$ into a finite LML-algebra can be dually characterized as a pair of (satisfaction and co-satisfaction) relations $(\Vdash, \succ)$ on the states of the TiRS-graph associated with $\Cc$. Again, as observed in the previous section, this dual characterization is in fact itself a correspondence argument.

The TiRS-graph semantics for the full LML signature, together with the definition of the relational duals of each operation,  can be derived by  arguments similar to the one given above. We are not going to present it in full here, but in the remainder of this section we are only going to consider the $\{\Diamond, \Box\}$-fragment of LML.

Informally, a \emph{TiRS-frame} for  LML is a structure $\mathbb{F} = (X, E,\mathcal{R})$ where $(X,E)$ is a TiRS-graph, and $\mathcal{R} = (R_{\Diamond}, R_{\Box}, R_{\lhd},$ $R_{\rhd}, R_{\circ}, R_{\star})$ is a tuple of additional relations on $X$ such that $R_{\Diamond}$,  $R_{\Box}$, $R_{\lhd}$, $R_{\rhd}$ are binary and $R_{\circ}$, and $R_{\star}$ are ternary. These relations satisfy additional compatibility conditions guaranteeing that the operations associated with them map Galois-stable sets to Galois-stable sets.

Because the specifics of the compatibility conditions play no role in the development of the present paper, we will not discuss them in full. The definition of $R_{\Diamond}$ in $\Cc_\bullet$  arises from a computation analogous to the one above, and is given by
\[\z R_{\Diamond}\z'  \quad \text{iff} \quad \Diamond \z'_J \nleq \z_M.\]
A \emph{TiRS-model} for LML is a tuple $\mathbb{M} = (\mathbb{X}, \mathcal{R}, V)$ where $(\mathbb{X}, \mathcal{R})$ is an TiRS-frame for LML and $V$ maps any $p \in \mathsf{PROP}$ to a subset  $V(p)\subseteq X$ which is a Galois-stable set of $\rho(\mathbb{X})$.

A \emph{TiRS-model} for LML$^+$ is a tuple $\mathbb{M}$ as above such that the map $V$ is extended to $\mathsf{NOM}$ and $\mathsf{CO}$-$\mathsf{NOM}$ as follows:  any nominal $\nomi$ is assigned to $\ell(u(\{\z\}))$ for some $\z\in X$ and any co-nominal   $\cnomm$ to $\ell(\{\z\})$ for some $\z\in X$, where $u$ and $l$ are the Galois connection maps of $\rho(\X)$.

For every TiRS-model $\mathbb{M} = (\mathbb{X}, \mathcal{R}, V)$, $\z \in X$, and any formula $\phi$ of the $\{\Diamond, \Box\}$-fragment of LML, the satisfaction and co-satisfaction relations $\mathbb{M}, \z \Vdash \phi$ and $\mathbb{M}, \z \succ \phi$ are defined by simultaneous recursion as follows:
\begin{flushleft}
\begin{tabular}{lllllll}
$\mathbb{M}, \z \Vdash \bot$ && never & &$\mathbb{M}, \z \succ \bot$ && always\\
$\mathbb{M}, \z \Vdash \top$ &&always & &$\mathbb{M}, \z \succ \top$ &&never\\
$\mathbb{M}, \z \Vdash p$ & iff & $\z\in V(p)$ & &$\mathbb{M}, \z \succ p$ & iff & $\forall \z'[\z'E\z \Rightarrow \z' \not\Vdash p]$\\
$\mathbb{M}, \z \Vdash \nomi$ & iff & $\z \in V(\nomi)$ & &$\mathbb{M}, \z \succ \nomi$ & iff & $\forall \z' [\z'E\z \Rightarrow \z' \not\Vdash \nomi]$\\
$\mathbb{M}, \z \Vdash \cnomn$ & iff & $\forall \z' [\z E\z' \Rightarrow \mathbb{M},\z' \not\succ \cnomn]$ &  &$\mathbb{M}, \z \succ \cnomn$ & iff & $\z \in V(\cnomn)$\\
$\mathbb{M}, \z \succ \phi \vee \psi$ &iff &$\mathbb{M},\z\succ \phi \text{ and } \mathbb{M},\z \succ \psi$%
& &$\mathbb{M}, \z \Vdash \phi \vee \psi$ &iff &$\forall \z' [\z E\z'   \Rightarrow \mathbb{M},\z' \not\succ \phi \vee \psi]$\\
$\mathbb{M}, \z \Vdash \phi \wedge \psi$ &iff &$\mathbb{M},\z\Vdash \phi \text{ and } \mathbb{M},\z \Vdash \psi$%
& &$\mathbb{M},\z \succ \phi \wedge\psi$ &iff &$\forall \z'[\z' E\z \Rightarrow \mathbb{M},\z'\not\Vdash \phi \wedge \psi]$\\
$\mathbb{M}, \z \succ \Diamond\phi$ &iff &$\forall \z' [\z R_{\Diamond}\z'  \Rightarrow \mathbb{M},\z' \not \Vdash \phi]$%
& &$\mathbb{M}, \z \Vdash \Diamond\phi$ &iff &$\forall \z'[\z E\z' \Rightarrow \mathbb{M},\z' \not\succ \Diamond \phi]$\\
$\mathbb{M}, \z \Vdash \Box \psi$ &iff &$\forall \z' [\z R_{\Box}\z'  \Rightarrow \mathbb{M},\z' \not\succ \psi]$%
& &$\mathbb{M}, \z \succ \Box\psi$ &iff &$\forall \z' [\z' E\z  \Rightarrow \mathbb{M},\z' \not\Vdash \Box \psi]$
\end{tabular}
\end{flushleft}

An $\mathrm{LML}^{+}$-inequality $\phi \leq \psi$ is \emph{true} in $\mathbb{M}$, denoted $\mathbb{M} \Vdash \phi \leq \psi$, if for all $\z, \z' \in X$, if $\mathbb{M}, \z \Vdash \phi$ and $\mathbb{M}, \z' \succ \psi$ then $\neg (\z E z')$. An $\mathrm{LML}^{+}$-inequality $\phi \leq \psi$ is \emph{valid} in $\mathbb{F}$ if it is true in every model based on $\mathbb{F}$.

Let $L_1$ be the first-order language with equality  built over a denumerable set of individual variables $Z$, with binary relation symbols $E$, $R_{\Diamond}$, $R_{\Box}$, $R_{\lhd}$, $R_{\rhd}$, ternary relation symbols $R_{\circ}$, $R_{\star}$ and a predicate symbol $P$ for each propositional variable $p \in \mathsf{PROP}$.\footnote{The intended interpretation of a predicate symbol $P$ in a TiRS-model $\mathbb{M} = (\X, \mathcal{R}, V)$, seen as a first-order structure, is $V(p)$ which, as indicated above, is a Galois-stable set.  Hence the interpretation of universal second-order quantification is also restricted to range over such sets.} We will further assume that $L_1$ contains denumerably many individual variables $i, j, \ldots$ corresponding to the nominals $\nomi, \nomj, \ldots \in \mathsf{NOM}$ and $n, m, \ldots$ corresponding to the co-nominals $\cnomn, \cnomm \in \mathsf{CO\text{-}NOM}$. Let $L_0$ be the sub-language which does not contain the unary predicate symbols corresponding to the propositional variables.

We are now ready to present the standard translation of the $\{\Diamond, \Box\}$-fragment of LML over TiRS-models. In this setting we need two standard translations, $\STP$ and $\STN$, respectively corresponding to satisfaction and co-satisfaction. These are defined by simultaneous induction as follows: 

\begin{center}
\begin{tabular}{ll}
$\STP_z(\bot) := z \neq z$ & $\STN_z(\bot) := z  = z$ \\
$\STP_z(\top) := z = z$ & $\STN_z(\top) := z \neq z$\\
$\STP_z(p) := P(z)$& $\STN_z(p) := \forall z' [z' E z \rightarrow \neg \STP_{z'}(p)] $\\
$\STP_z(\nomi) := \forall z'[\neg (i E z') \rightarrow \neg (z E z')]$ & $\STN_z(\nomi) := \forall z' [z' E z \rightarrow \neg \STP_{z'}(\nomi)] $\\
$\STP_z(\cnomn) := \forall z' [z E z' \rightarrow \neg \STN_{z'}(\cnomn)]$ & $\STN_z(\cnomn) :=  \neg (z E n)$\\
$\STP_z(\phi \vee \psi) := \forall z'[z E z' \rightarrow \neg \STN_{z'}(\phi \vee \psi)]$& $\STN_z(\phi \vee \psi) := \STN_{z}(\phi) \wedge \STN_{z}(\psi)$\\
$\STP_z(\phi \wedge \psi) := \STP_z(\phi) \wedge \STP_z(\psi)$& $\STN_z(\phi \wedge \psi) := \forall z'[z'Ez \rightarrow \neg \STP_{z'}(\phi \wedge \psi)]$\\
$\STP_z(\Diamond \phi) := \forall z' [z E z' \rightarrow \neg \STN_{z'}(\Diamond \phi)]$& $\STN_z(\Diamond \phi) := \forall z'[z R_{\Diamond} z' \rightarrow \neg \STP_{z'}(\phi)]$\\
$\STP_z(\Box \phi) := \forall z' [z R_{\Box} z' \rightarrow \neg \STN_{z'}(\phi)]$ & $\STN_{z}(\Box \phi) := \forall z' [z' E z \rightarrow \neg \STP_{z'}(\Box \phi)]$\\
\end{tabular}
\end{center}

\begin{lemma}\label{lemma:Standard:Trnsltn:TiRS}
For any TiRS-model $\mathbb{M} = (X,E,\mathcal{R}, V)$ for LML, any TiRS-frame $\mathbb{F}$ for LML, $\z \in X$, $z' \in Y$ and for all $LML^+$-formulas $\phi$ and $\psi$, the following hold:
\begin{enumerate}

\item $\mathbb{M}, \z \Vdash \phi$ iff $\mathbb{M} \models \STP_z(\phi)[z:= \z]$

\item $\mathbb{M}, \z \succ \psi$ iff $\mathbb{M} \models \STN_{z}(\psi)[z:= \z]$

\item
\begin{eqnarray*}
\mathbb{M} \Vdash \phi \leq \psi &\mbox{iff} &\mathbb{M} \models \forall z \forall z' [(\STP_z(\phi) \wedge  \STN_{z'}(\psi)) \rightarrow \neg zEy]\\
&\mbox{iff} &\mathbb{M} \models \forall z [\STP_z(\phi) \rightarrow \forall z'[ \STN_{z'}(\psi) \rightarrow \neg zEy]]\\
&\mbox{iff} &\mathbb{M} \models \forall z' [\STN_{z'}(\psi) \rightarrow \forall z  [\STP_z(\phi) \rightarrow \neg zEy]].
\end{eqnarray*}

\item
\begin{eqnarray*}
\mathbb{F} \Vdash \phi \leq \psi &\mbox{iff} &\mathbb{F} \models \forall \overline{P} \forall \overline{j} \forall \overline{m}\forall z \forall z' [(\STP_z(\phi) \wedge  \STN_{z'}(\psi)) \rightarrow \neg zEy]\\
&\mbox{iff} &\mathbb{F} \models \forall \overline{P} \forall \overline{j} \forall \overline{m}\forall z [\STP_z(\phi) \rightarrow \forall z'[ \STN_{z'}(\psi) \rightarrow \neg zEy]]\\
&\mbox{iff} &\mathbb{F} \models \forall \overline{P} \forall \overline{j} \forall \overline{m}\forall z' [\STN_{z'}(\psi) \rightarrow \forall z  [\STP_z(\phi) \rightarrow \neg zEy]].
\end{eqnarray*}
where $\overline{P}$, $\overline{j}$, and $\overline{m}$ are, respectively, the vectors of all predicate symbols corresponding to propositional variables, individual variables corresponding to nominals, and individual variables corresponding to co-nominals, occurring in $\STP_{z}(\phi)$ and $\STN_{z'}(\psi)$.

\end{enumerate}

\end{lemma}

\begin{example}
Consider again the LML-inequality $p\leq \Diamond p$ and the simplified pure inequality $\nomi \leq \Diamond \nomi$ resulting from running ALBA on it. We illustrate the standard translation on TiRS-frames by instantiating the statement of Lemma \ref{lemma:Standard:Trnsltn:TiRS}4 with $\nomi$ for $\phi$ and $\Diamond \nomi$ for $\psi$:

\begin{center}
\begin{tabular}{c l}
&$\mathbb{F} \Vdash \nomi \leq \Diamond \nomi$\\
iff &$\mathbb{F} \models \forall z \forall z' \forall i [(\STP_{z}(\nomi) \wedge \STN_{z'}(\Diamond \nomi)) \rightarrow \neg (z E z')]$\\
iff &$\mathbb{F} \models \forall z \forall z' \forall i [(\forall z''[\neg (i E z'') \rightarrow \neg (z E z'')] \wedge \forall u [z' R_{\Diamond} u \rightarrow \neg \STP_{u}(\nomi)]) \rightarrow \neg (z E z')]$\\
iff &$\mathbb{F} \models \forall z \forall z' \forall i [(\forall z''[\neg (i E z'') \rightarrow \neg (z E z'')] \wedge \forall u [z' R_{\Diamond} u \rightarrow \neg \forall v[\neg (i E v) \rightarrow \neg (u E v)]]) \rightarrow \neg (z E z')]$\\
iff &$\mathbb{F} \models \forall z \forall z' \forall i [(\forall z''[z E z'' \rightarrow i E z''] \wedge \forall u [z' R_{\Diamond} u \rightarrow \exists v[\neg (i E v) \wedge u E v]]) \rightarrow \neg (z E z')]$\\

\end{tabular}
\end{center}

The  theory developed in the present paper (cf.\ Theorems  \ref{Crctns:Theorem} and \ref{Thm:ALBA:Canonicity}) guarantees that the LE-logic obtained by adding $p\leq \Diamond p$ to the basic logic $\mathbf{L}_\mathrm{LML}$ (cf.\ Definition \ref{def:DLE:logic:general}) is sound and complete w.r.t.\ the elementary class of RS-frames for LML defined by the sentence above.
\end{example}

\section{Inductive and Sahlqvist inequalities}\label{Inductive:Fmls:Section}
In this section we  introduce the $\mathcal{L}_\mathrm{LE}$-analogue of the inductive inequalities in \cite{ALBAPaper}. This class of inequalities 
enjoys  canonicity and elementarity properties.  We will not give a direct proof that all inductive inequalities are elementary and canonical, but this will follow from the facts that they are all reducible by the {\sf ALBA}-algorithm and that all inequalities so reducible are elementary and canonical.

\subsection{Inductive inequalities}
		
		In the present subsection, we report on the definition of {\em inductive} $\mathcal{L}_\mathrm{LE}$-inequalities.
		on which the algorithm ALBA defined in Section \ref{Spec:Alg:Section} will be shown to succeed. 

		\begin{definition}[\textbf{Signed Generation Tree}]
			\label{def: signed gen tree}
			The \emph{positive} (resp.\ \emph{negative}) {\em generation tree} of any $\mathcal{L}_\mathrm{LE}$-term $s$ is defined by labelling the root node of the generation tree of $s$ with the sign $+$ (resp.\ $-$), and then propagating the labelling on each remaining node as follows:
			\begin{itemize}
				\item For any node labelled with $ \lor$ or $\land$, assign the same sign to its children nodes.
				\item For any node labelled with $h\in \mathcal{F}\cup \mathcal{G}$ of arity $n_h\geq 1$, and for any $1\leq i\leq n_h$, assign the same (resp.\ the opposite) sign to its $i$th child node if $\varepsilon_h(i) = 1$ (resp.\ if $\varepsilon_h(i) = \partial$).
			\end{itemize}
			Nodes in signed generation trees are \emph{positive} (resp.\ \emph{negative}) if are signed $+$ (resp.\ $-$).
		\end{definition}
		
		Signed generation trees will be mostly used in the context of term inequalities $s\leq t$. In this context we will typically consider the positive generation tree $+s$ for the left-hand side and the negative one $-t$ for the right-hand side. We will also say that a term-inequality $s\leq t$ is \emph{uniform} in a given variable $p$ if all occurrences of $p$ in both $+s$ and $-t$ have the same sign, and that $s\leq t$ is $\epsilon$-\emph{uniform} in a (sub)array $\vec{p}$ of its variables if $s\leq t$ is uniform in $p$, occurring with the sign indicated by $\epsilon$, for every $p$ in $\vec{p}$\footnote{\label{footnote:uniformterms}The following observation will be used at various points in the remainder of the present paper: if a term inequality $s(\vec{p},\vec{q})\leq t(\vec{p},\vec{q})$ is $\epsilon$-uniform in $\vec{p}$ (cf.\ discussion after Definition \ref{def: signed gen tree}), then the validity of $s\leq t$ is equivalent to the validity of $s(\overrightarrow{\top^{\epsilon(i)}},\vec{q})\leq t(\overrightarrow{\top^{\epsilon(i)}},\vec{q})$, where $\top^{\epsilon(i)}=\top$ if $\epsilon(i)=1$ and $\top^{\epsilon(i)}=\bot$ if $\epsilon(i)=\partial$. }.
		
		For any term $s(p_1,\ldots p_n)$, any order type $\epsilon$ over $n$, and any $1 \leq i \leq n$, an \emph{$\epsilon$-critical node} in a signed generation tree of $s$ is a leaf node $+p_i$ with $\epsilon_i = 1$ or $-p_i$ with $\epsilon_i = \partial$. An $\epsilon$-{\em critical branch} in the tree is a branch from an $\epsilon$-critical node. The intuition, which will be built upon later, is that variable occurrences corresponding to $\epsilon$-critical nodes are \emph{to be solved for}, according to $\epsilon$.
		
		For every term $s(p_1,\ldots p_n)$ and every order type $\epsilon$, we say that $+s$ (resp.\ $-s$) {\em agrees with} $\epsilon$, and write $\epsilon(+s)$ (resp.\ $\epsilon(-s)$), if every leaf in the signed generation tree of $+s$ (resp.\ $-s$) is $\epsilon$-critical.
		In other words, $\epsilon(+s)$ (resp.\ $\epsilon(-s)$) means that all variable occurrences corresponding to leaves of $+s$ (resp.\ $-s$) are to be solved for according to $\epsilon$. We will also write $+s'\prec \ast s$ (resp.\ $-s'\prec \ast s$) to indicate that the subterm $s'$ inherits the positive (resp.\ negative) sign from the signed generation tree $\ast s$. Finally, we will write $\epsilon(\gamma) \prec \ast s$ (resp.\ $\epsilon^\partial(\gamma_h) \prec \ast s$) to indicate that the signed subtree $\gamma$, with the sign inherited from $\ast s$, agrees with $\epsilon$ (resp.\ with $\epsilon^\partial$).
		\begin{definition}
			\label{def:good:branch}
			Nodes in signed generation trees will be called \emph{$\Delta$-adjoints}, \emph{syntactically left residual (SLR)}, \emph{syntactically right residual (SRR)}, and \emph{syntactically right adjoint (SRA)}, according to the specification given in Table \ref{Join:and:Meet:Friendly:Table}.
			A branch in a signed generation tree $\ast s$, with $\ast \in \{+, - \}$, is called a \emph{good branch} if it is the concatenation of two paths $P_1$ and $P_2$, one of which may possibly be of length $0$, such that $P_1$ is a path from the leaf consisting (apart from variable nodes) only of PIA-nodes, and $P_2$ consists (apart from variable nodes) only of Skeleton-nodes. 
A branch is \emph{excellent} if it is good and in $P_1$ there are only SRA-nodes. A good branch is \emph{Skeleton} if the length of $P_1$ is $0$ (hence Skeleton branches are excellent), and  is {\em SLR}, or {\em definite}, if  $P_2$ only contains SLR nodes.
			\begin{table}[\here]
				\begin{center}
                \bgroup
                \def\arraystretch{1.2}
					\begin{tabular}{| c | c |}
						\hline
						Skeleton &PIA\\
						\hline
						$\Delta$-adjoints & Syntactically Right Adjoint (SRA) \\
						\begin{tabular}{ c c c c c c}
							$+$ &$\vee$ &\\
							$-$ &$\wedge$ \\
							\hline
						\end{tabular}
						&
						\begin{tabular}{c c c c }
							$+$ &$\wedge$ &$g$ & with $n_g = 1$ \\
							$-$ &$\vee$ &$f$ & with $n_f = 1$ \\
							\hline
						\end{tabular}
						\\
						Syntactically Left Residual (SLR) &Syntactically Right Residual (SRR)\\
						\begin{tabular}{c c c c }
							$+$ &  &$f$ & with $n_f \geq 1$\\
							$-$ &  &$g$ & with $n_g \geq 1$ \\
						\end{tabular}
						&\begin{tabular}{c c c c}
							$+$ & &$g$ & with $n_g \geq 2$\\
							$-$ &  &$f$ & with $n_f \geq 2$\\
						\end{tabular}
						\\
						\hline
					\end{tabular}
                \egroup
				\end{center}
				\caption{Skeleton and PIA nodes for $\mathrm{LE}$.}\label{Join:and:Meet:Friendly:Table}
				\vspace{-1em}
			\end{table}
		\end{definition}
\begin{remark}\label{Remark:Pos:VS:Neg:Classification}
The classification above follows the general principles of unified correspondence as discussed in \cite{UnifiedCor}. As the names suggest, the subclassification of nodes as SLR, SRR, SRA and $\Delta$-adjoints refers to the inherent order theoretic properties of the operations interpreting these connectives, whereas the grouping of these classifications into Skeleton and PIA\footnote{The acronym PIA stands for ``Positive Implies Atomic'', and was introduced by van Benthem in \cite{vanbenthem2005}. The salient property of PIA-formulas is the intersection property, which means that, as term functions, they preserve certain meets.}  Indeed, as we will see later, the reduction strategy involves roughly two tasks, namely approximation and display. The order theoretic properties of Skeleton nodes facilitate approximation while those of PIA nodes facilitate display.  This will be further discussed and illustrated in sections \ref{Examples:Section} and \ref{Complete:For:Inductive:Section}. 

In \cite{ALBAPaper}, following \cite{GNV}, the nodes of the signed generation trees were classified according to the  \emph{choice} and \emph{universal} terminology. The reader is referred to \cite[Section 1.7.2]{UnifiedCor} for an expanded comparison of these two approaches.

The convention of considering the positive generation tree of the left-hand side and the negative generation tree of the right-hand side of an inequality also dates from \cite{GNV}. Although this might seem counter-intuitive at first glance, it is by now well established in this line of research, and we therefore  maintain it to facilitate easier comparisons.
\end{remark}

\begin{definition}[Inductive inequalities]\label{Inducive:Ineq:Def}
For any order type $\epsilon$ and any irreflexive and transitive relation $\Omega$ on $p_1,\ldots p_n$, the signed generation tree $*s$ $(* \in \{-, + \})$ of a term $s(p_1,\ldots p_n)$ is \emph{$(\Omega, \epsilon)$-inductive} if
			\begin{enumerate}
				\item for all $1 \leq i \leq n$, every $\epsilon$-critical branch with leaf $p_i$ is good (cf.\ Definition \ref{def:good:branch});
				\item every $m$-ary SRR-node occurring in the critical branch is of the form $ \circledast(\gamma_1,\dots,\gamma_{j-1},\beta,\gamma_{j+1}\ldots,\gamma_m)$, where for any $h\in\{1,\ldots,m\}\setminus j$: 
\begin{enumerate}
\item $\epsilon^\partial(\gamma_h) \prec \ast s$ (cf.\ discussion before Definition \ref{def:good:branch}), and
%
\item $p_k <_{\Omega} p_i$ for every $p_k$ occurring in $\gamma_h$ and for every $1\leq k\leq n$.
\end{enumerate}
	\end{enumerate}
			
			We will refer to $<_{\Omega}$ as the \emph{dependency order} on the variables. An inequality $s \leq t$ is \emph{$(\Omega, \epsilon)$-inductive} if the signed generation trees $+s$ and $-t$ are $(\Omega, \epsilon)$-inductive. An inequality $s \leq t$ is \emph{inductive} if it is $(\Omega, \epsilon)$-inductive for some $\Omega$ and $\epsilon$.
		\end{definition}
		
		In what follows, we will find it useful to refer to formulas $\phi$ such that only PIA nodes occur in $+\phi$ (resp.\ $-\phi$) as {\em positive} (resp.\ {\em negative}) {\em PIA-formulas}, and to formulas $\xi$ such that only Skeleton nodes occur in $+\xi$ (resp.\ $-\xi$) as {\em positive} (resp.\ {\em negative}) {\em Skeleton-formulas}\label{page: positive negative PIA}.

\begin{definition}\label{Sahlqvist:Ineq:Def}
Given an order type $\epsilon$, the signed generation tree $\ast s$, $\ast \in \{-, + \}$, of a term $s(p_1,\ldots p_n)$ is \emph{$\epsilon$-Sahlqvist} if every $\epsilon$-critical branch is excellent. An inequality $s \leq t$ is \emph{$\epsilon$-Sahlqvist} if the trees $+s$ and $-t$ are both $\epsilon$-Sahlqvist.  An inequality $s \leq t$ is \emph{Sahlqvist} if it is $\epsilon$-Sahlqvist for some $\epsilon$.
\end{definition}
\begin{example}\label{Inductive:And:Sahlqvist:Example}
The LML-inequality $\phi_1 \leq \psi_1:= (\Box p_1 \circ \Box {\lhd} p_1) \circ \Diamond p_2 \leq {\lhd} p_1 \star \Diamond\Box p_2$ is $\epsilon$-Sahlqvist for $\epsilon = (\partial,1)$, and for no other order type.

The LML-inequality $\phi_2 \leq \psi_2:= \Box ({\lhd}p_1 \star p_2) \wedge \Box p_1 \leq \Diamond (p_1 \circ p_2)$ is $(\Omega, \epsilon)$-inductive with $p_1 <_{\Omega} p_2$ and $\epsilon = (1,1)$, and also for $p_1 <_{\Omega} p_2$ and $\epsilon = (1,\partial)$.

The LML-inequality $\phi_3 \leq \psi_3:= \Diamond(\Box {\lhd}(q \circ r)  \wedge \Box (p \star \Box q)) \leq {\lhd}\Box (p \wedge r) \vee \Diamond p$ is $(\Omega, \epsilon)$-inductive with $p <_{\Omega} q$ and $r$  unrelated to $p$ and $q$ by $<_{\Omega}$, and $\epsilon_p = \partial$, $\epsilon_q = \epsilon_r = 1$.

The LML-inequality $\phi_4 = s \leq t := \Box ({\lhd} q \star p) \leq \Diamond ({\lhd} q \circ \Box p)$ is \emph{not} inductive. Indeed, for every $\epsilon$-critical branch in $+s$ and $-t$ to be good, the only possible $\epsilon$ is $(1,1)$. Given this $\epsilon$, if $+s$ is to be $(\Omega, \epsilon)$-inductive, it will have to be the case that $q <_{\Omega} p$, and similarly, for $-t$ to be $(\Omega, \epsilon)$-inductive, it will have to be the case that $p <_{\Omega} q$, which is impossible if $<_{\Omega}$ is to be a strict partial order.
\end{example}

\begin{example}
\label{ex:grishin interact princ}
The following Grishin interaction principles \cite{Grishin1983} are formulated in the language $\mathcal{L}_{\mathrm{LE}}(\mathcal{F}, \mathcal{G})$, with $\mathcal{F} = \{\circ, \starfor, \starback\}$ and $\mathcal{G} = \{\star, \circfor, \circback\}$:

\begin{center}
\begin{tabular}{c r c l c c r  c l  c c }
(a)&$(q\star r)\circ p$ &$\vdash$&$q\star(r\circ p)$ &\quad& (d) &$(p\starback q)\circback r $&$\vdash$& $q\circback (p\star r)$&$\quad$&\\
(b)&$p\star (r\circfor q) $&$\vdash$&$(p\star r)\circfor q$ &\quad& (e) &$(r\circ q)\starfor p $&$\vdash$&$r\starfor (p\circfor q)$ &$\quad$& (I)\\
(c)&$ p\starback (r\circ q)$&$\vdash$&$(p\starback r)\circ q$ &\quad& (f) &$ p\circ(r\circback q)$&$\vdash$&$(p\starfor r)\star q $&$\quad$&\\
   & &      & &           & &      & &       &\\
(a)&$(q\circback r)\circ p$ &$\vdash$& $q\circback (r\circ p)$&$\quad$& (d) & $(p\circ q)\circback r$&$\vdash$&$q\circback(p\circback r)$ &$\quad$&\\
(b)& $p\circback (r \circfor q)$&$\vdash$& $(p \circback r)\circfor q$ &$\quad$& (e) & $(p\circfor q)\circfor r$&$\vdash$&$p\circfor (r \circ q)$ &$\quad$& (II)\\
(c)&$p\circ (r\circ q)$ &$\vdash$&$(p\circ r)\circ q$ &$\quad$& (f) &$p\circ (r\circback q)$ &$\vdash$& $(r \circfor p)\circback q$ &$\quad$&\\
   & &      & &           & &      & &       &\\
(a)&$p\starfor(r\star q)$ &$\vdash$&$(p\starback r)\star q$ &$\quad$& (d) &$q\starfor (r\star p)$ &$\vdash$& $(q\starfor p)\starfor r$ &$\quad$&\\
(b)&$(p\star r)\star q$ &$\vdash$&$p\star (r\star q)$ &$\quad$& (e) & $p\starback (q\starback r)$ & $\vdash$& $(q\star p) \starback r$&$\quad$& (III)\\
(c)& $(p\starback r)\starfor q$&$\vdash$& $p\starback(r\starfor q)$&$\quad$& (f) &$(r\starfor q)\starback p$ &$\vdash$& $q\star (r\starback p)$ &$\quad$&\\
   & &      & &           & &      & &       &\\
(a)&$(q\circback r)\starfor p$ &$\vdash$&$q\circback(r\starfor a)$ &$\quad$& (d) &$(p \circback r)\starback q$ &$\vdash$& $r\starback (p\circ q)$ &$\quad$&\\
(b)&$q\circback(r\star p)$ &$\vdash$& $(q \circback r)\star p$ &$\quad$& (e) & $(p\star q)\circfor r$&$\vdash$&$p \circfor(r\starfor q)$ &$\quad$& (IV)\\
(c)&$p\circ(r\starfor q)$ &$\vdash$& $(p\circ r) \starfor q$ &$\quad$& (f) &$p\starfor (q\starback r)$ &$\vdash$&$(r\circfor p)\circback q$ &$\quad$&\\
\end{tabular}
\end{center}
The canonicity and correspondence of these axioms has been computed on a case-by-case basis in \cite{Mai-Lorijn-lambek-grishin}. The theory developed in the present paper subsumes the results in \cite{Mai-Lorijn-lambek-grishin}. Indeed, it is not difficult to see that all these axioms are either $\epsilon$-Sahlqvist or $(\epsilon, \Omega)$-inductive for at least one order type $(\epsilon_p, \epsilon_q, \epsilon_r)$ and strict partial order $\Omega$. For instance, (II c), (III b) and (III e) are Sahlqvist for $(1,1,1)$, $(\partial, \partial, \partial)$ and $(\partial, \partial, 1)$ respectively; (I a) is not Sahlqvist for any order-type, but is inductive e.g.\ for  $(1, \partial,1)$ with $q< r$, and for $(\partial, \partial,1)$ with $r< p$.
\end{example}

The following example has also been discussed in \cite[Example 2.24]{MaZhao15}.

\begin{example}
\label{ex:kurtonina}
In \cite[Definition 2.2.4, Theorem 2.2.5]{kurtonina:PhD}, Kurtonina defines the syntactic shape of a class of sequents in the $\{\circ, \backslash, /\}$-fragment of full Lambek calculus for which  she proves a  Sahlqvist-type correspondence result with the minimal valuation proof-strategy (there referred to as `canonical valuation'). It is not difficult to see that this class can be described using the notational conventions of the present paper as those sequents $\alpha\vdash \beta$ detailed below. Firstly, each non-leaf node of the signed generation tree $-\beta$ is required to have labels in  $\{+\circ, - /, - \backslash\}$.  This implies that there is exactly one  negative variable occurrence $q$ in $-\beta$ (i.e.\ the argument of the positive coordinate of the innermost node in $-\beta$ with label in $\{- /, - \backslash\}$). Let $\overline{p}$ be the vector of the  variables in $\beta$ different from $q$. The above requirement also implies that each variable in $\overline{p}$ occurs positively in $-\beta$. Hence, $+\alpha$ is further required to agree with some order-type $\epsilon^\partial$ such that $\epsilon(p) = 1$ for every $p$ in $\overline{p}$.

It is clear from this description that these requirements identify a proper subclass of Sahlqvist inequalities: indeed,  $-\beta$ is a Skeleton-formula/tree, and hence each of its branches is excellent. Each sequent described above corresponds then to an $\epsilon$-Sahlqvist inequality for the order-type $\epsilon$ determined by its antecedent $\alpha$ as described above. In Example \ref{ex:kurtonina:ALBA}, we provide a general description of the ALBA runs on such sequents, which  encodes the steps in the proof of \cite[Theorem 2.2.5]{kurtonina:PhD}.
\end{example}

\begin{example}
\label{rem:ghilardi suzuki}  In \cite{Suzuki:RSL:2013}, Suzuki gives a Sahlqvist-type theorem for full Lambek calculus. The basic logic treated in \cite{Suzuki:RSL:2013} is the normal LE logic with the additional requirements that $\backslash$ (resp.\ $/$) are right  residuals of $\circ$ (in the first and in the second coordinate, respectively). The Sahlqvist result in \cite{Suzuki:RSL:2013} covers  a syntactically defined class of sequents, described in terms of the  families defined by simultaneous recursion as follows:
\[ \cup\mbox{-terms}\quad s:: = b  \mid s\vee s\mid s\circ s\]
\[ \cap\mbox{-terms}\quad  t:: = d\mid  t\wedge t \mid s\backslash t \mid t/ s\]
\[ \quad\quad b:: = p\mid \top \mid \bot \mid b\wedge b\mid b/ c\mid c\backslash b\mid c/ d\mid d\backslash c \]
\[ \quad\quad  d:: = p\mid \top \mid \bot\mid d\vee d\mid d\circ c \mid c\circ d\]

where $c$ stands for any constant term.
The inequalities proven to be canonical in \cite[Theorem 5.10]{Suzuki:RSL:2013}   are of the form $\phi\leq \psi$, such that  \[\phi = s(\overline{x},  \overline{a_1(\overline{x})}/\overline{z})\quad \quad \psi = t(\overline{x},  \overline{a_2(\overline{x})}/\overline{z})\] where $s(\overline{x}, \overline{z})$ is a $\cup$-term, $t(\overline{x}, \overline{z})$ is a $\cap$-term, and  there exists some order-type $\epsilon$ on $\overline{x}$ such that $\epsilon^{\partial}(a_1(\overline{x}))\prec +\phi$ and $\epsilon^{\partial}(a_2(\overline{x}))\prec -\psi$ for each $a_1$ in $\overline{a_1}$ and each $a_2$ in $\overline{a_2}$.

From the description   given above, it is not difficult to see that the inequality $\phi\leq \psi$ described above falls under the definition of $\epsilon$-Sahlqvist inequalities (cf.\ Definition \ref{Sahlqvist:Ineq:Def}), when one takes into account the following observation.  Holding constant all coordinates but one in an operation corresponding to an SRR node turns it into an SRA node. Therefore, the trees $+b$ and $-d$ can be recognized as consisting entirely of PIA-nodes, and in particular  entirely of SRA nodes. Further, the generation trees $+s$ and $-t$ of $\cup$- and $\cap$-terms are constructed by taking trees consisting entirely of skeleton nodes and inserting subtrees $+b \prec +s, -t$ and $-d \prec +s, -t$ at leaves. Thus, \emph{all} branches in $+s$ and $-t$ are excellent. Substituting $\overline{a_1(\overline{x})}$ and $\overline{a_2(\overline{x})}$ into $s$ and $t$ as indicated possibly introduces non-good branches into $+ s(\overline{x},  \overline{a_1(\overline{x})}/\overline{z})$ and $- t(\overline{x},  \overline{a_2(\overline{x})}/\overline{z})$, but these will be non-critical according to $\epsilon$, and therefore both trees will be $\epsilon$-Sahlqvist.

The same analysis applies to the scope of the Sahlqvist canonicity result of Ghilardi and Meloni \cite{GhMe97} (cf.\ \cite[Remark 12]{CGPSZ14}). Indeed, 
Suzuki's treatment   extracts the syntactic definition from Ghilardi-Meloni's order-theoretic insights and transfers it to the setting of general lattice expansions. 
%
%
\end{example}

\subsection{The distributive setting}

When interpreting our language on perfect distributive lattice expansions (DLEs), the logical disjunction is interpreted by means of the coordinatewise completely $\wedge$-preserving join operation of the lattice, and the logical conjunction with the coordinatewise completely $\vee$-preserving meet operation of the lattice. Hence we are justified in listing $+\wedge$ and $-\vee$ among the SLRs, and $+\vee$ and $-\wedge$ among the SRRs, as is done in table \ref{Distr:Join:and:Meet:Friendly:Table}. Moreover, in the distributive setting, nominals and co-nominals are interpreted as completely join- (resp.\ meet-) prime elements of the perfect algebra. For reasons that will be discussed in Example \ref{Examp:Hopeless}, this makes it possible to apply $\Delta$-adjunction on $+\wedge$ and $-\vee$ nodes as part of the approximation task (cf.\ remark \ref{Remark:Pos:VS:Neg:Classification}), which justifies listing them among the Skeleton nodes in table \ref{Distr:Join:and:Meet:Friendly:Table}.

Consequently, we obtain enlarged classes of Sahlqvist and inductive inequalities
by simply applying definitions \ref{def:good:branch}, \ref{Sahlqvist:Ineq:Def} and \ref{Inducive:Ineq:Def} with respect to table \ref{Join:and:Meet:Friendly:Table:DLE}.

\begin{table}[\here]
				\begin{center}
					\begin{tabular}{| c | c |}
						\hline
						Skeleton &PIA\\
						\hline
						$\Delta$-adjoints & SRA \\
						\begin{tabular}{ c c c c c c}
							$+$ &$\vee$ &$\wedge$ &$\phantom{\lhd}$ & &\\
							$-$ &$\wedge$ &$\vee$\\
							\hline
						\end{tabular}
						&
						\begin{tabular}{c c c c }
							$+$ &$\wedge$ &$g$ & with $n_g = 1$ \\
							$-$ &$\vee$ &$f$ & with $n_f = 1$ \\
							\hline
						\end{tabular}
						\\
						SLR &SRR\\
						\begin{tabular}{c c c c }
							$+$ & $\wedge$  &$f$ & with $n_f \geq 1$\\
							$-$ & $\vee$ &$g$ & with $n_g \geq 1$ \\
						\end{tabular}
						&\begin{tabular}{c c c c}
							$+$ & $\vee$ &$g$ & with $n_g \geq 2$\\
							$-$ & $\wedge$  &$f$ & with $n_f \geq 2$\\
						\end{tabular}
						\\
						\hline
					\end{tabular}
				\end{center}
				\caption{Skeleton and PIA nodes for $\mathcal{L}_\mathrm{DLE}$.}\label{Join:and:Meet:Friendly:Table:DLE}
				\vspace{-1em}
			\end{table}

\paragraph{Distributive LML and DML.} The classification of nodes adopted in \cite{GNV} and \cite{ALBAPaper}  appears on the left half of the following table \ref{Distr:Join:and:Meet:Friendly:Table}.  In the right half,  we have specialized the table \ref{Join:and:Meet:Friendly:Table:DLE} above to the LML-signature, which is the running example in the present paper:

\begin{table}[\here]
\begin{center}
\begin{tabular}{| c | c || c | c |}
\hline
choice & universal & Skeleton  &PIA\\
\hline
&&$\Delta$-adjoints  & SRA \\
\begin{tabular}{ c  c  c  c }
$+$ &$\vee$ &$\Diamond$ &$\lhd$\\

$-$ &$\wedge$ &$\Box$ &$\rhd$\\
\hline
\end{tabular}
&
\begin{tabular}{ c  c  c}
$+$ &$\Box$ &$\rhd$\\

 $-$ & $\Diamond$ &$\lhd$\\
 \hline
\end{tabular}
&
\begin{tabular}{ c c c c c  c}
$+$ &$\vee$ &$\wedge$ &$\phantom{\lhd}$ &$\phantom{\circ}$ &$\phantom{\circ}$\\
$-$ &$\wedge$ &$\vee$\\
\hline
\end{tabular}
&
\begin{tabular}{c c c c}
$+$ &$\Box$ &$\rhd$ &$\wedge$ \\
$-$ &$\Diamond$ &$\lhd$ &$\vee$ \\
\hline
\end{tabular}

\\
&&SLR  &SRR\\
&&
\begin{tabular}{c c c c c c}
$+$ & $\wedge$ & $\Diamond$ &$\lhd$ &$\circ$\\
$-$ & $\vee$ &$\Box$ &$\rhd$ &$\star$ &\\
\end{tabular}
&\begin{tabular}{c c c c}
$+$ &$\vee$ &$\star$ &$\phantom{\wedge}$\\
$-$ & $\wedge$ &$\circ$\\
\end{tabular}
\\
\hline
\end{tabular}
\end{center}
\caption{choice and universal nodes for $\mathrm{DML}$, and Skeleton and PIA nodes for distributive $\mathrm{LML}$.}\label{Distr:Join:and:Meet:Friendly:Table}
\end{table}

\noindent The LML-signature  can be projected onto that of DML (cf.\ \cite{GNV} and \cite{ALBAPaper}) by identifying the occurrences of $\star$ and $\circ$ with $\vee$  and $\wedge$ respectively. Let us denote this projection by $\pi: \mathrm{LML} \rightarrow \mathrm{DML}$. Under this projection, the LML-inductive inequalities coincide with the DML-inductive inequalities of \cite{ALBAPaper}.

\begin{prop}
The projection $\pi$ maps the inductive and Sahlqvist $\mathrm{LML}$-inequalities onto the inductive and Sahlqvist $\mathrm{DML}$-inequalities, respectively.
\end{prop}
\begin{proof}
We sketch the proof in the case of inductive inequalities.
Let $s \leq t$ be an  $(\Omega, \epsilon)$-inductive $\mathrm{LML}$-inequality, and let $s' \leq t'$ be its image under $\pi$.
Given that all choice nodes are Skeleton, that all universal nodes are PIA, and that all $\epsilon$-critical paths in $+s$ and $-t$ are good, it follows that the only choice nodes in the scope of universal nodes on $\epsilon$-critical paths in $+s'$ and $-t'$ must be $+\vee$ and $-\wedge$ nodes resulting from the translation of $+\star$ and $- \circ$ nodes. But the condition imposed on these $+\star$ and $- \circ$ nodes by definition \ref{Inducive:Ineq:Def} are clearly equivalent to those imposed on $+\vee$ and $-\wedge$ nodes by the definition in the DML case. Hence $s' \leq t'$ is $(\Omega, \epsilon)$-inductive.

In order to show that every inductive $\mathrm{DML}$-inequality is in the range of $\pi$, let $s' \leq t'$ be an $(\Omega, \epsilon)$-inductive $\mathrm{DML}$-inequality. On all $\epsilon$-critical paths in the generation trees $+s'$ and $-t'$, begin by replacing every $+\vee$ and $-\wedge$ which has a universal node as ancestor with $+\star$ and $-\circ$ respectively. Then, on these same branches, replace every occurrence of $+\wedge$ and $-\vee$ which has a choice node as descendant  with $+\circ$ and $-\star$, respectively.  Let $t, s \in \mathrm{LML}$ be the terms corresponding to the resulting trees.

It is immediate that $\pi(s) = s'$ and $\pi(t) = t'$. It is also easy to see that all $\epsilon$-critical branches in $+s$ and $-t$ are good. Moreover, the introduced binary SRR-nodes $+\star$ and $-\circ$ replace binary choice nodes which were descendants of universal nodes and hence satisfied the conditions imposed by the definition of  inductive $\mathrm{DML}$-inequalities. Hence the introduced $+\star$ and $-\circ$ nodes will satisfy definition \ref{Inducive:Ineq:Def}.
\end{proof}

\begin{example}\label{Projection:Pi:Example}
Consider the DML-inequality $\Diamond (\Diamond p \wedge \Diamond q) \leq \Diamond \Box(p \wedge q)$, which is $(1,1)$-Sahlqvist according to definition 3.5 in \cite{ALBAPaper}. Note that it is not an LML-Sahlqvist inequality, so it cannot be its own image under $\pi$ restricted to Sahlqvist inequalities. However, it is the $\pi$-image of the Sahlqvist LML-inequality $\Diamond (\Diamond p \circ \Diamond q) \leq \Diamond \Box(p \wedge q)$, obtained by applying the strategy outlined in the proof above.
\end{example}

\section{Non-distributive ALBA}\label{Spec:Alg:Section}

ALBA takes an $\mathcal{L}_\mathrm{LE}$-inequality $\phi \leq \psi$ as input and then proceeds in three stages. The first stage preprocesses $\phi \leq \psi$ by eliminating all uniformly  occurring propositional variables, and applying distribution and splitting rules exhaustively. This produces a finite set of inequalities, $\phi'_i \leq \psi'_i$, $1 \leq i \leq n$.

Now ALBA forms the \emph{initial quasi-inequalities} $\bigamp S_i \Rightarrow \sf{Ineq}_i$, compactly represented as tuples $(S_i, \sf{Ineq}_i)$ referred as \emph{systems},  with each $S_i$ initialized to the empty set and $\sf{Ineq}_i$ initialized to $\phi'_i \leq \psi'_i$.

The second stage (called the reduction  stage) transforms $S_i$ and $\mathsf{Ineq}_i$ through the application of transformation rules, which are listed below. The aim is to eliminate all ``wild'' propositional variables from $S_i$ and $\mathsf{Ineq}_i$ in favour of ``tame'' nominals and co-nominals (for an expanded discussion on the general reduction strategy, the reader is referred to \cite{UnifiedCor,ConPalSur}). A system for which this has been done will be called \emph{pure} or \emph{purified}. The actual eliminations are effected through the Ackermann-rules, while the other rules are used to bring $S_i$ and $\mathsf{Ineq}_i$ into the appropriate shape which make these applications possible. Once all propositional variables have been eliminated, this phase terminates and returns the pure quasi-inequalities $\bigamp S_i \Rightarrow \mathsf{Ineq}_i$.

The third stage either reports failure if some system could not be purified, or else returns the conjunction of the pure quasi-inequalities $\bigamp S_i \Rightarrow \mathsf{Ineq}_i$, which we denote by $\mathsf{ALBA}(\phi \leq \psi)$.

We now outline each of the three stages in more detail:

\subsection{Stage 1: Preprocessing and initialization} ALBA receives an $\mathrm{LE}$-inequality $\phi \leq \psi$ as input. It applies the following {\bf rules for elimination of monotone variables}  to $\phi \leq \psi$ exhaustively, in order to eliminate any propositional variables which occur uniformly:

\begin{prooftree}
\AxiomC{$\alpha(p) \leq \beta(p)$}\UnaryInfC{$\alpha(\bot) \leq \beta(\bot)$}
\AxiomC{$\gamma(p) \leq \delta(p)$}\UnaryInfC{$\gamma(\top) \leq \delta(\top)$}
\noLine\BinaryInfC{}
\end{prooftree}
for $\alpha(p) \leq \beta(p)$ positive and $\gamma(p) \leq \delta(p)$ negative in $p$, respectively (see footnote \ref{footnote:uniformterms}).

Next, ALBA exhaustively distributes $f\in \mathcal{F}$ over $+\vee$, and  $g\in \mathcal{G}$ over $-\wedge$, so as to bring occurrences of $+\vee$ and $-\wedge$ to the surface wherever this is possible, and then eliminate them via exhaustive applications of {\em splitting} rules.
\paragraph{Splitting-rules.}

\begin{prooftree} \AxiomC{$\alpha \leq \beta \wedge \gamma
$}\UnaryInfC{$\alpha \leq \beta \quad \alpha \leq \gamma$}
\AxiomC{$\alpha \vee \beta \leq \gamma$}\UnaryInfC{$\alpha \leq \gamma \quad \beta \leq \gamma$}
\noLine\BinaryInfC{}
\end{prooftree}

This gives rise to  a set of inequalities $\{\phi_i' \leq \psi_i'\mid 1\leq i\leq n\}$. Now ALBA forms the \emph{initial quasi-inequalities} $\bigamp S_i \Rightarrow \sf{Ineq}_i$, compactly represented as tuples $(S_i, \sf{Ineq}_i)$ referred as \emph{systems},  with each $S_i$ initialized to the empty set and $\sf{Ineq}_i$ initialized to $\phi'_i \leq \psi'_i$. Each initial system is passed separately to stage 2, described below, where we will suppress indices $i$.

\subsection{Stage 2: Reduction and elimination}\label{Sec:ReductionElimination}

The aim of this stage is to eliminate all occurring propositional variables from a given system $(S, \mathsf{Ineq})$. This is done by means of the following \emph{approximation rules}, \emph{residuation rules}, \emph{splitting rules}, and \emph{Ackermann-rules}, collectively called \emph{reduction rules}. The terms and inequalities in this subsection are from $\mathcal{L}_\mathrm{LE}^{+}$.

\paragraph{Approximation rules.} There are four approximation rules. Each of these rules functions by simplifying $\mathsf{Ineq}$ and adding an inequality to $S$.

\begin{description}
\item[Left-positive approximation rule.] $\phantom{a}$
\begin{center}
\AxiomC{$(S, \;\; \phi'(\gamma / !x)\leq \psi)$}
\RightLabel{$(L^+A)$}
\UnaryInfC{$(S\! \cup\! \{ \nomj \leq \gamma\},\;\; \phi'(\nomj / !x)\leq \psi)$}
\DisplayProof
\end{center}
with $+x \prec +\phi'(!x)$,  the branch of $+\phi'(!x)$ starting at $+x$ being SLR (cf.\ definition \ref{def:good:branch}), $\gamma$ belonging to the original language $\mathcal{L}_\mathrm{LE}$
%
    %
    %
    and $\nomj$ being the first nominal variable not  occurring in $S$ or  $\phi'(\gamma / !x)\leq \psi$.
\item[Left-negative approximation rule.]$\phantom{a}$
\begin{center}
\AxiomC{$(S, \;\; \phi'(\gamma / !x)\leq \psi)$}
\RightLabel{$(L^-A)$}
\UnaryInfC{$(S\! \cup\! \{ \gamma\leq\cnomm\},\;\; \phi'(\cnomm / !x)\leq \psi)$}
\DisplayProof
\end{center}

with $-x \prec +\phi'(!x)$, the branch of $+\phi'(!x)$ starting at $-x$ being SLR, $\gamma$ belonging to the original language $\mathcal{L}_\mathrm{LE}$ and
     $\cnomm$ being the first co-nominal not  occurring in $S$ or $\phi'(\gamma / !x)\leq \psi$.
\item[Right-positive approximation rule.]$\phantom{a}$
\begin{center}
\AxiomC{$(S, \;\; \phi\leq \psi'(\gamma / !x))$}
\RightLabel{$(R^+A)$}
\UnaryInfC{$(S\! \cup\! \{ \nomj \leq \gamma\},\;\; \phi\leq \psi'(\nomj / !x))$}
\DisplayProof
\end{center}

with $+x \prec -\psi'(!x)$,  the branch of $-\psi'(!x)$ starting at $+x$ being SLR, $\gamma$ belonging to the original language $\mathcal{L}_\mathrm{LE}$ and
     $\nomj$ being the first nominal not  occurring in $S$ or $\phi\leq \psi'(\gamma / !x)$.
\item[Right-negative approximation rule.] $\phantom{a}$
\begin{center}
\AxiomC{$(S, \;\; \phi\leq \psi'(\gamma / !x))$}
\RightLabel{$(R^-A)$}
\UnaryInfC{$(S\! \cup\! \{ \gamma\leq \cnomm\},\;\; \phi\leq \psi'(\cnomm / !x))$}
\DisplayProof
\end{center}
with $-x \prec -\psi'(!x)$, the branch of $-\psi'(!x)$ starting at $-x$ being SLR, $\gamma$ belonging to the original language $\mathcal{L}_\mathrm{LE}$ and  $\cnomm$ being the first co-nominal not  occurring in $S$ or $\phi\leq \psi'(\gamma / !x))$.
\end{description}

\noindent The approximation rules above, as stated, will be shown to be sound both under admissible and under arbitrary assignments (cf.\ Proposition \ref{Rdctn:Rls:Crctnss:Prop}). However, their liberal application gives rise to topological complications in the proof of canonicity. Therefore, we will restrict the applications of approximation rules to nodes $!x$ giving rise to {\em maximal} SLR branches. Such applications will be called {\em pivotal}. Also, executions of ALBA in which approximation rules are applied only pivotally will be referred to as {\em pivotal}.\label{pivotal:approx:rule:application}

\paragraph{Residuation rules.} These rules operate on the inequalities in $S$, by rewriting a chosen inequality in $S$ into another inequality. For every $f\in \mathcal{F}$ and $g\in \mathcal{G}$, and any $1\leq i\leq n_f$ and $1\leq j\leq n_g$,

\begin{prooftree}
\AxiomC{$f(\phi_1,\ldots,\phi_i,\ldots,\phi_{n_f}) \leq \psi $}
\RightLabel{$\epsilon_f(i) = 1$}
\UnaryInfC{$\phi_i\leq f^\sharp_i(\phi_1,\ldots,\psi,\ldots,\phi_{n_f})$}
\AxiomC{$f(\phi_1,\ldots,\phi_i,\ldots,\phi_{n_f}) \leq \psi $}
\RightLabel{$\epsilon_f(i) = \partial$}
\UnaryInfC{$f^\sharp_i(\phi_1,\ldots,\psi,\ldots,\phi_{n_f})\leq \phi_i$}
\noLine\BinaryInfC{}
\end{prooftree}

\begin{prooftree}
\AxiomC{$\psi\leq g(\phi_1,\ldots,\phi_i,\ldots,\phi_{n_g})$}
\RightLabel{$\epsilon_g(i) = 1$}
\UnaryInfC{$g^\flat_i(\phi_1,\ldots,\psi,\ldots,\phi_{n_g})\leq \phi_i$}
\AxiomC{$\psi\leq g(\phi_1,\ldots,\phi_i,\ldots,\phi_{n_g})$}
\RightLabel{$\epsilon_g(i) = \partial$}
\UnaryInfC{$\phi_i\leq g^\flat_i(\phi_1,\ldots,\psi,\ldots,\phi_{n_g})$}
\noLine\BinaryInfC{}
\end{prooftree}

\paragraph{Right Ackermann-rule.} $\phantom{a}$
\begin{center}
\AxiomC{$(\{ \alpha_i \leq p \mid 1 \leq i \leq n \} \cup \{ \beta_j(p)\leq \gamma_j(p) \mid 1 \leq j \leq m \}, \;\; \mathsf{Ineq})$}
\RightLabel{$(RAR)$}
\UnaryInfC{$(\{ \beta_j(\bigvee_{i=1}^n \alpha_i)\leq \gamma_j(\bigvee_{i=1}^n \alpha_i) \mid 1 \leq j \leq m \},\;\; \mathsf{Ineq})$}
\DisplayProof
\end{center}
where:
\begin{itemize}
\item $p$ does not occur in $\alpha_1, \ldots, \alpha_n$ or in $\mathsf{Ineq}$,
\item $\beta_{1}(p), \ldots, \beta_{m}(p)$ are positive in $p$, and
\item $\gamma_{1}(p), \ldots, \gamma_{m}(p)$ are negative in $p$.

\end{itemize}

\paragraph{Left Ackermann-rule.}$\phantom{a}$
\begin{center}
\AxiomC{$(\{ p \leq \alpha_i \mid 1 \leq i \leq n \} \cup \{ \beta_j(p)\leq \gamma_j(p) \mid 1 \leq j \leq m \}, \;\; \mathsf{Ineq})$}
\RightLabel{$(LAR)$}
\UnaryInfC{$(\{ \beta_j(\bigwedge_{i=1}^n \alpha_i)\leq \gamma_j(\bigwedge_{i=1}^n \alpha_i) \mid 1 \leq j \leq m \},\;\; \mathsf{Ineq})$}
\DisplayProof
\end{center}
where:
\begin{itemize}
\item $p$ does not occur in $\alpha_1, \ldots, \alpha_n$ or in $\mathsf{Ineq}$,
\item $\beta_{1}(p), \ldots, \beta_{m}(p)$ are negative in $p$, and
\item $\gamma_{1}(p), \ldots, \gamma_{m}(p)$ are positive in $p$.

\end{itemize}

\subsection{Stage 3: Success, failure and output}

If stage 2 succeeded in eliminating all propositional variables from each system, the algorithm returns the conjunction of these purified quasi-inequalities, denoted by $\mathsf{ALBA}(\phi \leq \psi)$. Otherwise, the algorithm reports failure and terminates.

\section{Examples}\label{Examples:Section}

In the present section, we collect some examples of the execution of ALBA on various inequalities.

\begin{example}
\label{example: reduction involution}
The inequality $p^{\bot\bot}\leq p$ discussed in Example \ref{example:linear involution} is not $\epsilon$-inductive for $\epsilon_p = 1$, but  is clearly $\epsilon$-Sahlqvist for $\epsilon_p = \partial$. The initial system is given by

\[
S_0 = \emptyset \textrm{ \ and \ } \mathsf{Ineq}_0 = p^{\bot\bot}\leq p.
\]
Applying the left-positive and right-positive approximation rules yields
\[
S_1 = \{ \nomj\leq p^{\bot\bot}, p\leq\cnomm  \} \textrm{ \ and \ } \mathsf{Ineq}_1 = \nomj\leq \cnomm.
\]

The Left Ackermann-rule may now be applied to eliminate $p$:
\[
S_2 = \{ \nomj\leq \cnomm^{\bot\bot} \} \textrm{ \ and \ } \mathsf{Ineq}_2 = \nomj\leq \cnomm.
\]

Thus we can output the `purified' quasi inequality
\[
 \nomj\leq \cnomm^{\bot\bot} \Rightarrow \nomj\leq \cnomm,
\]
which is equivalent to the pure inequality $\cnomm^{\bot\bot} \leq \cnomm$.
\end{example}

\begin{example}
The Grishin interaction axiom (I a) of Example \ref{ex:grishin interact princ} $(q\star r)\circ p\leq q\star(r\circ p)$ is $(\Omega, \epsilon)$-inductive for  $\epsilon(p, q, r) = (1, \partial,1)$ and $q<_\Omega r$. The initial system is given by

\[
S_0 = \emptyset \textrm{ \ and \ } \mathsf{Ineq}_0 = (q\star r)\circ p\leq q\star(r\circ p).
\]
Applying the left-positive and right-positive approximation rules yields
\[
S_1 = \{ \nomj\leq q\star r, q\leq\cnomm  \} \textrm{ \ and \ } \mathsf{Ineq}_1 = \nomj\circ p\leq \cnomm\star (r \circ p).
\]
Applying the Left Ackermann rule  yields
\[
S_2 = \{ \nomj\leq \cnomm\star r  \} \textrm{ \ and \ } \mathsf{Ineq}_2 = \nomj\circ p\leq \cnomm\star (r \circ p).
\]
Applying the right-positive approximation rule yields
\[
S_3 = \{ \nomj\leq \cnomm\star r, r\circ p\leq \cnomn \} \textrm{ \ and \ } \mathsf{Ineq}_3 = \nomj\circ p\leq \cnomm\star \cnomn.
\]
Applying the residuation rule for $\star$ yields
\[
S_4 = \{ \cnomm\starback\nomj\leq r, r\circ p\leq \cnomn \} \textrm{ \ and \ } \mathsf{Ineq}_4 = \nomj\circ p\leq \cnomm\star \cnomn.
\]
Applying the Right Ackermann rule we can eliminate $r$:
\[
S_5 = \{ (\cnomm\starback\nomj)\circ p\leq \cnomn \} \textrm{ \ and \ } \mathsf{Ineq}_5 = \nomj\circ p\leq \cnomm\star \cnomn.
\]
Finally, applying the left-positive approximation rule yields
\[
S_6 = \{ (\cnomm\starback\nomj)\circ p\leq \cnomn, \nomi\leq p \} \textrm{ \ and \ } \mathsf{Ineq}_5 = \nomj\circ \nomi\leq \cnomm\star \cnomn.
\]
Applying the Right Ackermann rule we can eliminate $p$:
\[
S_7 = \{ (\cnomm\starback\nomj)\circ \nomi\leq \cnomn \} \textrm{ \ and \ } \mathsf{Ineq}_5 = \nomj\circ \nomi\leq \cnomm\star \cnomn.
\]
Thus we can output the `purified' quasi inequality
\[
 (\cnomm\starback\nomj)\circ \nomi\leq \cnomn\Rightarrow \nomj\circ \nomi\leq \cnomm\star \cnomn,
\]
which is equivalent to
\[
 (\cnomm\starback\nomj)\circ \nomi\leq \cnomn\Rightarrow \cnomm\starback(\nomj\circ \nomi)\leq \cnomn,
\]
which is equivalent to the pure inequality $\cnomm\starback(\nomj\circ \nomi)\leq (\cnomm\starback\nomj)\circ \nomi$.
\end{example}

\begin{remark}
Notice that the pure inequality in output in the example above has the same shape as axiom (I c), the only difference being that the second-order variables $p, q, r$ have been uniformly replaced by nominals and co-nominals. An analogous behaviour has been observed of the Fischer Servi axioms of intuitionistic modal logic (cf.\ \cite[Lemma 27]{ma2014algebraic}),\footnote{In \cite{ma2014algebraic}, this observation is used to give a very compact proof of the fact that  every perfect intuitionistic modal algebra $(\bbA, \Box, \Diamond)$ is a Fischer Servi algebra iff $(\bbA, \blacksquare, \Diamondblack)$ is (cf.\ \cite[Proposition 28]{ma2014algebraic}). So in particular, similar facts can be proved of lattices in the Lambek-Grishin signature.} and of other Grishin interaction axioms in
 \cite[Section 5.3.2]{Mai-Lorijn-lambek-grishin}, where it is observed:
\begin{quote}
Notice that the first-order condition [resulting from the reduction of (IV b)] is precisely the shape of the interaction axiom (e)
listed in (IV). What has happened however is that the second-order variables [$p,q,r$] in the
join-preserving coordinates of the operations have been replaced by variables from $X$, while
the variables in the meet-preserving coordinates have been replaced by variables from $Y$. It is
interesting that there is some kind of calculus behind this that gives the same replacements in
axioms (b), (c) and (d).
\end{quote}
The phenomenon of which all these observations are instances has been systematically explored with the tools of unified correspondence theory in \cite{GMPTZ}---where the algorithm ALBA serves as the calculus mentioned in the quotation above---and applied to the design of a methodology which effectively computes analytic structural rules of a proper display calculus from given axioms/inequalities in the language of any normal DLE-logic. Also, the second tool of unified correspondence, i.e., the uniform definition of Sahlqvist and inductive inequalities across normal DLE-signatures, has been used as a basis for  the syntactic  characterization of the class of those axioms/inequalities from which analytic rules can be extracted in a way which is guaranteed to preserve logical equivalence (these are the so-called {\em analytic inductive inequalities}, cf.\ \cite[Definition 53]{GMPTZ}).  In particular, all Grishin interaction axioms can be easily seen to be analytic inductive, and hence the general procedure can be applied to them. In \cite{Moortgat}, a blue-print was given to transform Grishin axioms into analytic structural rules. These same  rules can be obtained by instantiating the ALBA-based  procedure defined in \cite{GMPTZ}. For instance, the inverse of the first rule in \cite[(10)]{Moortgat} is obtained by observing that the pure inequality in output in the example above is definite left-primitive (cf.\ \cite[Definition 26]{GMPTZ}), hence its validity  is equivalent to that of  $p\starback(q\circ r)\leq (p\starback q)\circ r$ (cf.\ \cite[Proposition 35]{GMPTZ}),  and applying the following transformation steps to it we get the required rule:
\[p\starback(q\circ r)\leq (p\starback q)\circ r \quad \rightsquigarrow \quad \frac{(p\starback q)\circ r\vdash z}{p\starback(q\circ r)\vdash z} \quad \rightsquigarrow\quad \frac{p\starback q\vdash z\circfor r}{q\circ r\vdash p\star z}\]
\end{remark}

\begin{example}
\label{ex:kurtonina:ALBA} In Example \ref{ex:kurtonina}, we described the class of categorial principles treated by Kurtonina in \cite{kurtonina:PhD} in terms of a certain subclass of Sahlqvist inequalities $\alpha(\overline{p}, q)\leq \beta(\overline{p}, q)$ such that $-\beta$ is a Skeleton formula, all the variable occurrences of which are positive but for exactly one occurrence of the variable $q$, and $\alpha$ is an $\epsilon$-monotone term with some further requirements on $\epsilon$ guaranteeing the applicability of the Ackermann rules.
In what follows, we outline an alternative proof of \cite[Theorem 2.2.5]{kurtonina:PhD} in the form of a general ALBA  run on each such inequality.
Let us assume that the input inequality is not uniform in $q$. This means that either $+\alpha$ is positive in $q$ or there are positive occurrences of $q$ in $-\beta$. We proceed by cases: (a) if $+\alpha$ is positive in $q$, then  we solve for $\epsilon(\overline{p}, q) = (\overline{1}, \partial)$
After having eliminated all uniform variables, if any, the initial system is given by

\[
S_0 = \emptyset \textrm{ \ and \ } \mathsf{Ineq}_0 = \alpha(\overline{p}, q)\leq \beta(\overline{p}, q).
\]
Applying the right-positive approximation rules to  each positive variable occurrence in $-\beta$ (including positive occurrences of $q$ in $\beta$, if any),  the right-negative approximation rule to  the negative occurrence of $q$, and the left-positive approximation rule to the left-hand side of the inequality, we obtain:
\[
S_1 = \{ \nomi\leq \alpha(\overline{p}, q), \overline{\bigvee_{I_{i_p}}\nomj}\leq \overline{p}, \bigvee_{I_{i_q}}\nomj\leq q,  q\leq\cnomm \}   \textrm{ \ and \ } \mathsf{Ineq}_1 = \nomi\leq \beta(\overline{\nomj}, \cnomm).
\]
The assumptions on $+\alpha$ make it possible to apply the right Ackermann-rule to eliminate each $p$ in $\overline{p}$ and the left Ackermann-rule to eliminate $q$:
\[
S_2 = \{ \nomi\leq \alpha(\overline{\bigvee_{I_i}\nomj}, \cnomm), \bigvee_{I_{i_q}}\nomj\leq\cnomm \}   \textrm{ \ and \ } \mathsf{Ineq}_2 = \nomi\leq \beta(\overline{\nomj}, \cnomm).
\]
If (b) if $+\alpha$ is not positive in $q$, then the assumption that the input inequality is not uniform in $q$ implies that there are positive occurrences of $q$ in $-\beta$. Then we solve for $\epsilon(\overline{p}, q) = (\overline{1}, 1)$, in a similar way to that described above. If the input inequality is  uniform in $q$, then $q$ is eliminated during the preprocessing, and the remaining variables in $\overline{p}$ are eliminated as described above.
\end{example}

\begin{example}
Consider the $(1,1)$-Sahlqvist inequality $(\Box p_1\circ \Box {\lhd} p_1) \circ \Diamond p_2 \leq {\lhd} p_1 \star \Diamond\Box p_2$, (cf.\ Example \ref{Inductive:And:Sahlqvist:Example}). There are no uniform variables to eliminate, so the initial system is given by
\[
S_0 = \emptyset \textrm{ \ and \ } \mathsf{Ineq}_0 = (\Box p_1\circ \Box {\lhd} p_1) \circ \Diamond p_2 \leq {\lhd} p_1 \star \Diamond\Box p_2.
\]
Applying the left-positive approximation rule yields
\[
S_1 = \{\nomj_1 \leq  \Box p_1 \} \textrm{ \ and \ } \mathsf{Ineq}_1 = (\nomj_1 \circ \Box {\lhd} p_1) \circ \Diamond p_2 \leq {\lhd} p_1 \star \Diamond\Box p_2,
\]
which another application turns into
\[
S_2 = \{\nomj_1 \leq  \Box p_1,  \nomj_2 \leq \Box {\lhd} p_1\} \textrm{ \ and \ } \mathsf{Ineq}_2 = (\nomj_1 \circ \nomj_2) \circ \Diamond p_2 \leq {\lhd} p_1 \star \Diamond\Box p_2.
\]
Applying the right-negative approximation rule now yields
\[
S_3 = \{\nomj_1 \leq  \Box p_1,  \nomj_2 \leq \Box {\lhd} p_1, {\lhd} p_1 \leq \cnomm_1\} \textrm{ \ and \ } \mathsf{Ineq}_3 = (\nomj_1 \circ \nomj_2) \circ \Diamond p_2 \leq \cnomm_1\star \Diamond\Box p_2.
\]
Applying the residuation rules for $\Box$ and $\lhd$ turns this into
\[
S_4 = \{\Diamondblack \nomj_1 \leq  p_1,  \nomj_2 \leq \Box {\lhd} p_1, {\blacktriangleleft} \cnomm_1 \leq p_1\} \textrm{ \ and \ } \mathsf{Ineq}_4 = (\nomj_1 \circ \nomj_2) \circ \Diamond p_2 \leq \cnomm_1\star \Diamond\Box p_2,
\]
to which the Right Ackermann-rule may be applied to eliminate $p_1$, thus:
\[
S_5 = \{\nomj_2 \leq \Box {\lhd} (\Diamondblack \nomj_1 \vee {\blacktriangleleft} \cnomm_1)\} \textrm{ \ and \ } \mathsf{Ineq}_5 = (\nomj_1 \circ \nomj_2) \circ \Diamond p_2 \leq \cnomm_1\star \Diamond\Box p_2.
\]
Applying the left-positive and right-negative approximation rules yields
\[
S_6 = \{\nomj_2 \leq \Box {\lhd} (\Diamondblack \nomj_1 \vee {\blacktriangleleft} \cnomm_1), \ \nomj_3 \leq p_2, \ \Diamond\Box p_2 \leq \cnomm_2 \} \textrm{ \ and \ } \mathsf{Ineq}_6 = (\nomj_1 \circ \nomj_2) \circ \Diamond  \nomj_3 \leq \cnomm_1 \star \cnomm_2,
\]
which an application of the Ackermann-rule turns into
\[
S_7 = \{\nomj_2 \leq \Box {\lhd} (\Diamondblack \nomj_1 \vee {\blacktriangleleft} \cnomm_1), \ \Diamond\Box \nomj_3 \leq \cnomm_2 \} \textrm{ \ and \ } \mathsf{Ineq}_7 = (\nomj_1 \circ \nomj_2) \circ \Diamond  \nomj_3 \leq \cnomm_1 \star \cnomm_2.
\]
Thus we can output the `purified' quasi inequality
\[
\nomj_2 \leq \Box {\lhd} (\Diamondblack \nomj_1 \vee {\blacktriangleleft} \cnomm_1) \: \& \: \Diamond\Box \nomj_3 \leq \cnomm_2 \Rightarrow (\nomj_1 \circ \nomj_2) \circ \Diamond  \nomj_3 \leq \cnomm_1 \star \cnomm_2.
\]
\end{example}

\begin{example}
Consider the inductive inequality $\phi_3 \leq \psi_3:= \Diamond(\Box {\lhd}(q \circ r)  \wedge \Box (p \star \Box q)) \leq {\lhd}\Box(p \wedge r) \vee \Diamond p$ from example \ref{Inductive:And:Sahlqvist:Example}. Recall that this inequality is $(\Omega, \epsilon)$-inductive with $p <_{\Omega} q$ and $ r$ unrelated to $p$ and $q$ via$\leq_{\Omega}$, and $\epsilon_p = \partial$, $\epsilon_q = \epsilon_r = 1$. Once again there are no uniform variables to eliminate. The initial system is
\[
S_0 = \emptyset \textrm{ \ and \ } \mathsf{Ineq}_0 = \Diamond(\Box {\lhd}(q \circ r)  \wedge \Box (p \star \Box q)) \leq {\lhd}\Box(p \wedge r) \vee \Diamond p.
\]
Applying the left-positive and right-negative approximation rules give
\[
S_1 = \{\nomj \leq  \Box {\lhd}(q \circ r)  \wedge \Box (p \star \Box q), {\lhd}\Box(p \wedge r) \vee \Diamond p \leq \cnomm \}  \textrm{ \ and \ } \mathsf{Ineq}_1 = \Diamond \nomj \leq \cnomm.
\]
Proceeding in accordance with $\Omega$ and $\epsilon$, eliminate $p$ first, solving for its negative occurrences in the inequalities in $S$. Applying the $\vee$-splitting rule to ${\lhd}\Box(p \wedge r) \vee \Diamond p \leq \cnomm$ and then the $\Diamond$-residuation rule yields
\[
S_2 = \{\nomj \leq  \Box {\lhd}(q \circ r)  \wedge \Box (p \star \Box q), {\lhd}\Box(p \wedge r) \leq \cnomm, p \leq \blacksquare \cnomm \}  \textrm{ \ and \ } \mathsf{Ineq}_2 = \Diamond \nomj \leq \cnomm.
\]
Note that the first two inequalities in $S_2$ are positive in $p$. This, together with the shape of the third, implies that the left Ackerman-rule is applicable to eliminate p:
\[
S_3 = \{\nomj \leq  \Box {\lhd}(q \circ r)  \wedge \Box (\blacksquare \cnomm \star \Box q), {\lhd}\Box(\blacksquare \cnomm \wedge r) \leq \cnomm \}  \textrm{ \ and \ } \mathsf{Ineq}_3 = \Diamond \nomj \leq \cnomm.
\]
We next want to eliminate $q$. Applying the $\wedge$-splitting rule to the first inequality in $S_3$ followed by the $\Box$-,  $\star$-, and again $\Box$-residuation rules yields
\[
S_4 = \{\nomj \leq  \Box {\lhd}(q \circ r), \Diamondblack (\blacksquare  \cnomm \starback \Diamondblack \nomj) \leq  q, {\lhd}\Box(\blacksquare \cnomm \wedge r) \leq \cnomm \}  \textrm{ \ and \ } \mathsf{Ineq}_4 = \Diamond \nomj \leq \cnomm,
\]
to which the right Ackermann-rule is applicable with respect to $q$, resulting in the system
\[
S_5 = \{\nomj \leq  \Box {\lhd}(\Diamondblack (\blacksquare  \cnomm \starback \Diamondblack \nomj) \circ r),  {\lhd}\Box(\blacksquare \cnomm \wedge r) \leq \cnomm \}  \textrm{ \ and \ } \mathsf{Ineq}_5 = \Diamond \nomj \leq \cnomm.
\]
Only $r$ remains to be eliminated. Applying the ${\lhd}$-residuation rule to ${\lhd}\Box(\blacksquare \cnomm \wedge r) \leq \cnomm$ yields
\[
S_6 = \{\nomj \leq  \Box {\lhd}(\Diamondblack (\blacksquare  \cnomm \starback \Diamondblack \nomj) \circ r),  {\blacktriangleleft} \cnomm \leq  \Box(\blacksquare \cnomm \wedge r) \}  \textrm{ \ and \ } \mathsf{Ineq}_6 = \Diamond \nomj \leq \cnomm.
\]
Now applying the $\Box$-residuation rule to ${\blacktriangleleft} \cnomm \leq  \Box(\blacksquare \cnomm \wedge r)$ followed by the $\wedge$-splitting rule, gives
\[
S_7 = \{\nomj \leq  \Box {\lhd}(\Diamondblack (\blacksquare  \cnomm \starback \Diamondblack \nomj) \circ r),  \Diamondblack {\blacktriangleleft} \cnomm \leq  \blacksquare \cnomm,  \Diamondblack {\blacktriangleleft} \cnomm \leq  r \}  \textrm{ \ and \ } \mathsf{Ineq}_7 = \Diamond \nomj \leq \cnomm,
\]
to which the right Ackermann-rule is applicable with respect to $r$ resulting in the purified system
\[
S_8 = \{\nomj \leq  \Box {\lhd}(\Diamondblack (\blacksquare  \cnomm \starback \Diamondblack \nomj) \circ \Diamondblack {\blacktriangleleft} \cnomm ),  \Diamondblack {\blacktriangleleft} \cnomm \leq  \blacksquare \cnomm  \}  \textrm{ \ and \ } \mathsf{Ineq}_8 = \Diamond \nomj \leq \cnomm.
\]
\end{example}

\begin{example}
Consider the non-inductive inequality $\Box ({\rhd} q \star p) \leq \Diamond ({\lhd} q \circ \Box p)$ from example \ref{Inductive:And:Sahlqvist:Example}. The initial system is
\[
S_0 = \emptyset \textrm{ \ and \ } \mathsf{Ineq}_0 = \Box ({\rhd} q \star p) \leq \Diamond ({\lhd} q \circ \Box p).
\]
The only rules applicable to this system are the left-positive and right-negative approximation rules. Applying these gives
\[
S_1 = \{\nomj \leq  \Box ({\rhd} q \star p),  \Diamond ({\lhd} q \circ \Box p) \leq \cnomm \} \textrm{ \ and \ } \mathsf{Ineq}_1 =  \nomj \leq \cnomm.
\]
This system can be solved for either $p$ or $q$. Solving for $p$ we apply the $\Box$- and $\star$-residuation rules, giving us a system ready for the application of the right-Ackermann rule:
\[
S_2 = \{ {\rhd} q \starback \Diamondblack \nomj \leq  p,  \Diamond ({\lhd} q \circ \Box p) \leq \cnomm \} \textrm{ \ and \ } \mathsf{Ineq}_2 =  \nomj \leq \cnomm.
\]
Applying the Ackermann-rule gives
\[
S_3 = \{ \Diamond ({\lhd} q \circ \Box {\rhd} q \starback \Diamondblack \nomj) \leq \cnomm \} \textrm{ \ and \ } \mathsf{Ineq}_3 =  \nomj \leq \cnomm.
\]
This cannot be solved for $q$, whichever rules we apply to $\Diamond ({\lhd} q \circ \Box {\rhd} q \starback \Diamondblack \nomj) \leq \cnomm$, there will always be an occurrence of $q$ on both sides of the resulting inequalities. As the reader can check, one ends up in a similar situation with respect to $p$ when solving for $q$ first.
\end{example}

\begin{example}\label{Examp:Hopeless}
Consider the inequality $\Diamond(\Diamond p \wedge \Diamond q) \leq \Diamond \Box(p \wedge q)$ from example \ref{Projection:Pi:Example}. As remarked there, this is a DML Sahlqvist inequality and hence the distributive version of ALBA is guaranteed to succeed on it. It follows that it is elementary and canonical in the setting of DML.  However, this is not LML Sahlqvist (or even inductive) and, as we will see now, non-distributive ALBA fails on it. As far as the authors are aware, it is not known whether it is elementary and/or canonical in the setting of LML. In what follows we adopt a more compact notation.

\begin{center}
\begin{tabular}{cll}
&$\Diamond(\Diamond p \wedge \Diamond q) \leq \Diamond \Box(p \wedge q)$\\
&$(\emptyset, \; \;\Diamond(\Diamond p \wedge \Diamond q) \leq \Diamond \Box(p \wedge q))$ \\
&$(\{\nomj \leq \Diamond p \wedge \Diamond q \}, \; \;\Diamond \nomj \leq \Diamond \Box(p \wedge q))$ &($L^+A$)\\
&$(\{\nomj \leq \Diamond p, \; \nomj \leq \Diamond q \}, \; \;\Diamond \nomj \leq \Diamond \Box(p \wedge q))$ &($\wedge$-split)
\end{tabular}
\end{center}

At this point we are stuck. Had we been working in the distributive setting we could have proceeded further as follows:

\begin{center}
\begin{tabular}{cll}
&$(\{\nomj \leq \Diamond p, \; \nomj \leq \Diamond q \}, \; \;\Diamond \nomj \leq \Diamond \Box(p \wedge q))$\\
&$(\{\nomj \leq \Diamond \nomk, \; \nomk \leq p, \; \nomj \leq \Diamond \nomi, \; \nomi \leq q \}, \; \;\Diamond \nomj \leq \Diamond \Box(p \wedge q))$ &(DML approximation)\\
&$(\{\nomj \leq \Diamond \nomk, \; \nomk \leq p, \; \nomj \leq \Diamond \nomi, \; \nomi \leq q, \; \Diamond \Box(p \wedge q) \leq \cnomm \}, \; \;\Diamond \nomj \leq \cnomm)$ &($R^- A$)\\
&$(\{\nomj \leq \Diamond \nomk, \; \nomj \leq \Diamond \nomi, \; \nomi \leq q, \; \Diamond \Box(\nomk \wedge q) \leq \cnomm \}, \; \;\Diamond \nomj \leq \cnomm)$ &(RAR)\\
&$(\{\nomj \leq \Diamond \nomk, \; \nomj \leq \Diamond \nomi, \; \Diamond \Box(\nomk \wedge \nomi) \leq \cnomm \}, \; \;\Diamond \nomj \leq \cnomm)$ &(RAR)\\
\end{tabular}
\end{center}

However, the DML approximation is not sound in the setting of non-distributive lattices, precisely because here join-irreducibles are not in general join prime. Notice that, in the DML case, the two branches of $+ \Diamond(\Diamond p \wedge \Diamond q)$ are both Skeleton, and accordingly we succeed in preparing for the Ackermann rule purely through approximation rules. This example shows that failure need not be the result of a loss of order theoretic properties of interpretations of connectives when passing from the distributive to the non-distributive setting, but rather that in this setting the same properties cannot be used to perform similar tasks. In particular, the problem in the current example in the LML setting is not purely the presence of $+ \wedge$ in the antecedent, but rather the diamonds occurring in its scope, to which further approximation cannot be applied even if they are still completely join-preserving.

Finally, non-distributive ALBA would succeed on both the inequalities $\Diamond(\Box p \wedge \Box q) \leq \Diamond \Box(p \wedge q)$ and $\Diamond(\Diamond p \circ \Diamond q) \leq \Diamond \Box(p \wedge q)$. In the first case, the $\wedge$-split would form part of the display task rather than the approximation task (as was the case in DML), while in the second, preparation for the Ackermann rule would be accomplished purely through approximation rules.
\end{example}


\section{Justification of correctness}\label{Crrctnss:Section}

In this section we prove that ALBA is correct, in the sense that whenever it succeeds in eliminating all propositional variables from an inequality $\phi \leq \psi$, the conjunction of the quasi-inequalities returned is equivalent on perfect $\mathcal{L}_\mathrm{LE}$-algebras to $\phi \leq \psi$, for an arbitrarily fixed language $\mathcal{L}_\mathrm{LE}$.

Fix a perfect $\mathcal{L}_\mathrm{LE}$-algebra $\bba = (A, \wedge, \vee, \bot, \top, \mathcal{F}^\bba, \mathcal{G}^\bba)$ for the remainder of this section. We first give the statement of the correctness theorem and its proof, and subsequently prove the lemmas needed in the proof.

\begin{theorem}[Correctness]\label{Crctns:Theorem} If ALBA succeeds in reducing an $\mathcal{L}_\mathrm{LE}$-inequality $\phi \leq \psi$ and yields $\mathsf{ALBA}(\phi \leq \psi)$, then $\bba \models \phi \leq \psi$ iff $\bba \models \mathsf{ALBA}(\phi \leq \psi)$.
\end{theorem}
\begin{proof}
Let $\phi_i \leq \psi_i$, $1 \leq i \leq n$ be the quasi-inequalities produced by preprocessing $\phi \leq \psi$. Consider the chain of statements (\ref{Crct:Eqn1}) to (\ref{Crct:Eqn5}) below. The proof will be complete if we can show that they are all equivalent.
\begin{eqnarray}
&&\bba \models \phi \leq \psi\label{Crct:Eqn1}\\
&&\bba \models \phi_i \leq \psi_i,  \quad 1 \leq i \leq n \label{Crct:Eqn2}\\
&&\bba \models \bigamp \varnothing \Rightarrow  \phi_i \leq \psi_i, \quad 1 \leq i \leq n \label{Crct:Eqn3}\\
&&\bba \models \bigamp S_i \Rightarrow \mathsf{Ineq}_i, \quad 1 \leq i \leq n \label{Crct:Eqn4}\\
& &\bba \models \mathsf{ALBA}(\phi \leq \psi)\label{Crct:Eqn5}
\end{eqnarray}
For the equivalence of (\ref{Crct:Eqn1}) and (\ref{Crct:Eqn2}) we need to verify that the rules for the elimination of uniform variables, distribution and splitting preserve validity on $\bba$. Distribution and splitting are immediate. As to elimination, if $\alpha(p) \leq \beta(p)$ is positive in $p$, then for all $a \in \bba$ it is the case that $\alpha(a) \leq \alpha(\bot)$ and  $\beta(\bot) \leq \beta(a)$. Hence if $\alpha(\bot) \leq \beta(\bot)$ then $\alpha(a) \leq \beta(a)$, and hence $\bba \models \alpha(p) \leq \beta(p)$ iff $\bba \models \alpha(\bot) \leq \beta(\bot)$. The case for $\gamma(p) \leq \delta(p)$ negative in $p$ is similar.

That (\ref{Crct:Eqn2}) and (\ref{Crct:Eqn3}) are equivalent is immediate. The bi-implication between (\ref{Crct:Eqn3}) and (\ref{Crct:Eqn4}) follows from proposition \ref{Rdctn:Rls:Crctnss:Prop}, while (\ref{Crct:Eqn4}) and (\ref{Crct:Eqn5}) are the same by definition.
\end{proof}

\begin{lemma}[Distribution lemma]\label{Distribution:Lemma}
If  $\phi(!x), \psi(!x), \xi(!x), \chi(!x) \in \mathcal{L}_\mathrm{LE}^{+}$  and $\{ a_j \}_{j \in I} \subseteq A$, then
\begin{enumerate}
\item $\phi(\bigvee_{j \in I} a_j) = \bigvee\{ \phi(a_j) \mid j \in I \}$, when $+x \prec +\phi(!x)$ and in $+\phi(!x)$ the branch ending in $+x$ is SLR;
\item $\psi(\bigwedge_{j \in I} a_j) = \bigvee\{ \psi(a_j) \mid j \in I \}$, when $-x \prec +\psi(!x)$ and in $+\psi(!x)$ the branch ending in $-x$ is SLR;
\item $\xi(\bigwedge_{j \in I} a_j) = \bigwedge\{ \xi(a_j) \mid j \in I \}$, when $-x \prec -\xi(!x)$ and in $-\xi(!x)$ the branch ending in $-x$ is SLR;
\item $\chi(\bigvee_{j \in I} a_j) = \bigwedge\{ \chi(a_j) \mid j \in I \}$, when $+x \prec -\chi(!x)$ and in $-\chi(!x)$ the branch ending in $+x$ is SLR.
\end{enumerate}
\end{lemma}
\begin{proof}
The proof is by simultaneous induction on $\phi$, $\psi$, $\xi$ and $\chi$. The base cases for $\bot$, $\top$, and $x$, when applicable, are trivial. We check the inductive cases for $\phi$, and list the other inductive cases, which all follow in a similar way.

\begin{description}
    \item[$\phi$ of the form $f(\phi_1,\ldots,\phi_i(!x),\ldots,\phi_{n_f})$ with $f\in \mathcal{F}$ and $\epsilon_f(i) = 1$:] By the assumption of a unique occurrence of $x$ in $\phi$, the variable $x$ occurs in $\phi_i$ for exactly one index $1\leq i\leq n_f$. 
        The assumption that $\epsilon_f(i) = 1$ implies that $+ x \prec + \phi_i$.
        Then
        \begin{center}
        \begin{tabular}{r c l}
        $\phi(\bigvee_{j \in I} a_j)$ &$ =$&$ f(\phi_1,\ldots,\phi_i(\bigvee_{j \in I} a_j)\ldots,\phi_{n_f})$\\
        &$ =$&$  f(\phi_1,\ldots, \bigvee_{j \in I}\phi_i( a_j)\ldots,\phi_{n_f})$\\
         &$ =$&$  \bigvee_{j \in I}f(\phi_1,\ldots,\phi_i(a_j)\ldots,\phi_{n_f})$\\
         &$ =$&$  \bigvee_{j \in I} \phi(a_j)$,
          \end{tabular}
          \end{center}where the second equality holds by the inductive hypothesis, since the branch of $+\phi$ ending in $+x$ is SLR, and it traverses $+\phi_i$.
\item[$\phi$ of the form $f(\phi_1,\ldots,\psi_i(!x),\ldots,\phi_{n_f})$ with $f\in \mathcal{F}$ and $\epsilon_f(i) = \partial$:] By the assumption of a unique occurrence of $x$ in $\phi$, the variable $x$ occurs in $\psi_i$ for exactly one index $1\leq i\leq n_f$. The assumption that $\epsilon_f(i) = \partial$ implies that $- x \prec + \psi_i$. 
    Then \begin{center}
        \begin{tabular}{r c l}
        $\phi(\bigvee_{j \in I} a_i)$ &$ =$&$ f(\phi_1,\ldots,\psi_i(\bigvee_{j \in I} a_j)\ldots,\phi_{n_f})$\\
        &$ =$&$  f(\phi_1,\ldots, \bigwedge_{j \in I}\psi_i( a_j)\ldots,\phi_{n_f})$\\
         &$ =$&$  \bigvee_{j \in I}f(\phi_1,\ldots,\psi_i(a_j)\ldots,\phi_{n_f})$\\
         &$ =$&$  \bigvee_{j \in I} \phi(a_j)$,
          \end{tabular}
          \end{center}
     where the second equality holds by the inductive hypothesis, since the branch of $+\phi$ ending in $+x$ is SLR, and it traverses $-\psi_i$.
    %




\item[$\psi$ of the form]  $f(\psi_1,\ldots,\psi_i(!x),\ldots,\psi_{n_f})$ with $f\in \mathcal{F}$ and $\epsilon_f(i) = 1$ or $f(\psi_1,\ldots,\xi_i(!x),\ldots,\phi_{n_f})$ with $f\in \mathcal{F}$ and $\epsilon_f(i) = \partial$.

\item[$\xi$ of the form] $g(\xi_1,\ldots,\xi_i(!x),\ldots,\xi_{n_g})$ with $g\in \mathcal{G}$ and $\epsilon_g(i) = 1$ or $g(\xi_1,\ldots,\chi_i(!x),\ldots,\xi_{n_g})$ with $g\in \mathcal{G}$ and $\epsilon_g(i) = \partial$. 


\item[$\chi$ of the form] $g(\chi_1,\ldots,\chi_i(!x),\ldots,\chi_{n_g})$ with $g\in \mathcal{G}$ and $\epsilon_g(i) = 1$ or $g(\chi_1,\ldots,\phi_i(!x),\ldots,\xi_{n_g})$ with $g\in \mathcal{G}$ and $\epsilon_g(i) = \partial$. 
\end{description}
\end{proof}

\begin{lemma}[Right Ackermann Lemma]\label{Ackermann:Right:Lemma}
Let $\alpha_1, \ldots, \alpha_n \in \mathcal{L}_\mathrm{LE}^{+}$ with $p \not \in \mathsf{PROP}(\alpha_i)$, $1 \leq i \leq n$, let $\beta_{1}(p), \ldots, \beta_{m}(p) \in \mathcal{L}_\mathrm{LE}^{+}$ be positive in $p$, and let $\gamma_{1}(p), \ldots, \gamma_{m}(p) \in \mathcal{L}_\mathrm{LE}^{+}$ be negative in $p$. Let $v$ be any assignment on a perfect LE $\bba$. Then
\[
\bba, v \models \beta_i(\alpha_1 \vee \cdots \vee \alpha_n / p) \leq \gamma_i(\alpha_1 \vee \cdots \vee \alpha_n / p), \textrm{ for all } 1 \leq i \leq m
\]
iff there exists a variant $v' \sim_p v$ such that
\[
\bba, v' \models \alpha_i \leq p \textrm{ for all } 1 \leq i \leq n, \textrm{ and } \bba, v' \models \beta_i(p) \leq \gamma_i(p), \textrm{ for all } 1 \leq i \leq m.
\]
\end{lemma}
\begin{proof}
For the implication from top to bottom, let $v'(p) = v(\alpha_1 \vee \cdots \vee \alpha_n)$. Since the $\alpha_i$ does not contain $p$, we have $v(\alpha_i) = v'(\alpha_i)$, and hence that $v'(\alpha_i) \leq v(\alpha_1) \vee \cdots \vee v(\alpha_n) = v(\alpha_1 \vee \cdots \vee \alpha_n) = v'(p)$. Moreover, $v'(\beta_i(p)) = v(\beta_i(\alpha_1 \vee \cdots \vee \alpha_n / p)) \leq v(\gamma_i(\alpha_1 \vee \cdots \vee \alpha_n / p)) = v'(\gamma_i(p))$.

For the implication from bottom to top, we make use of the fact that the $\beta_i$ are monotone (since positive) in $p$, while the $\gamma_i$ are antitone (since negative) in $p$. Since $v(\alpha_i) = v'(\alpha_i) \leq v'(p)$ for all $1 \leq n \leq n$, we have $v(\alpha_1 \vee \cdots \vee \alpha_n) \leq v'(p)$, and hence $v(\beta_i(\alpha_1 \vee \cdots \vee \alpha_n / p)) \leq v'(\beta_i(p)) \leq v'(\gamma_i(p)) \leq (\gamma_i(\alpha_1 \vee \cdots \vee \alpha_n / p))$.
\end{proof}

The proof of the following version of the lemma is similar.

\begin{lemma}[Left Ackermann Lemma]\label{Ackermann:Left:Lemma}
Let $\alpha_1, \ldots, \alpha_n \in \mathcal{L}_\mathrm{LE}^{+}$ with $p \not \in \mathsf{PROP}(\alpha_i)$, $1 \leq i \leq n$, let $\beta_{1}(p), \ldots, \beta_{m}(p) \in \mathcal{L}_\mathrm{LE}^{+}$ be negative in $p$, and let $\gamma_{1}(p), \ldots, \gamma_{m}(p) \in \mathcal{L}_\mathrm{LE}^{+}$ be positive in $p$. Let $v$ be any assignment on a perfect LE $\bba$. Then
\[
\bba, v \models \beta_i(\alpha_1 \wedge \cdots \wedge \alpha_n / p) \leq \gamma_i(\alpha_1 \wedge \cdots \wedge \alpha_n / p), \textrm{ for all } 1 \leq i \leq m
\]
iff there exists a variant $v' \sim_p v$ such that
\[
\bba, v' \models p \leq \alpha_i  \textrm{ for all } 1 \leq i \leq n, \textrm{ and } \bba, v' \models \beta_i(p) \leq \gamma_i(p), \textrm{ for all } 1 \leq i \leq m.
\]
\end{lemma}

\begin{prop}\label{Rdctn:Rls:Crctnss:Prop}
If a system $(S,\mathsf{Ineq})$ is obtained from a system $(S_0,\mathsf{Ineq}_0)$ by the application of reduction rules, then
\[
\bba \models \forall \overline{var_0} \: [\bigamp S_0 \Rightarrow \mathsf{Ineq}_0] \textrm{ \ iff \ } \bba \models \forall \overline{var} \: [\bigamp S \Rightarrow \mathsf{Ineq}],
\]
where $\overline{var_0}$ and $\overline{var}$ are the vectors of all variables occurring in $\bigamp S_0 \Rightarrow \mathsf{Ineq}_0$ and $\bigamp S \Rightarrow \mathsf{Ineq}$, respectively.
\end{prop}
\begin{proof}
It is sufficient to verify that each rule preserves this equivalence, i.e., that if $S$ and $\mathsf{Ineq}$ are obtained from $S'$ and $\mathsf{Ineq}'$ by the application of a single transformation rule then
\[
\bba \models \forall \overline{var'} \: [\bigamp S' \Rightarrow \mathsf{Ineq}'] \textrm{ \ iff \ } \bba \models \forall \overline{var} \: [\bigamp S \Rightarrow \mathsf{Ineq}].
\]
\paragraph{Left-positive approximation rule:} Let $\mathsf{Ineq}$ be $ \phi'(\gamma / !x) \leq \psi$, with $+x \prec +\phi'(!x)$, and the branch of $+\phi'(!x)$ starting at $+x$ being SLR. Then, under any assignment to the occurring variables, $\phi'(\gamma / !x) = \phi'(\bigvee \{j \in \jira \mid j \leq \gamma \}) = \bigvee \{\phi'(j) \mid \gamma \geq j \in \jira\}$, where the latter equality holds by lemma \ref{Distribution:Lemma}.1. But then
\[
\bba \models \forall \overline{var'} \: [\bigamp S' \Rightarrow \phi'(\gamma/!x) \leq \psi] \textrm{ \ iff \ } \bba \models \forall \overline{var'} \forall \nomj \: [\bigamp S' \& \; \nomj \leq \gamma \Rightarrow \phi'(\nomj) \leq \psi].
\]
The other approximation rules are justified in a similar manner, appealing to the other clauses of lemma \ref{Distribution:Lemma}.

The residuation rules are justified by the fact that (the algebraic interpretation of) every $f\in \mathcal{F}$ (resp.\ $g\in \mathcal{G}$) is a left (resp.\ right) residual in each positive coordinate and left (resp.\ right) Galois-adjoint in each negative coordinate.  

The Ackermann-rules are justified by the Ackermann lemmas \ref{Ackermann:Right:Lemma} and \ref{Ackermann:Left:Lemma}.
\end{proof}
\begin{remark}
\label{rem: towards constructive can}
Notice that, in the proof above, we have not used the fact that nominals and co-nominals are interpreted as completely join-irreducible and meet-irreducible elements respectively. We only used the fact that completely join-irreducibles (resp.\ meet-irreducibles) completely join-generate (resp.\ meet-generate) the algebra $\bba$. This is a notable difference with the distributive setting of \cite{ALBAPaper}, where the soundness of the approximation rules essentially depends upon the complete join-primeness (resp.\ meet-primeness) of the interpretation of nominals and co-nominals respectively. This observation will be crucially put to use in the conclusions, where we discuss how the constructive canonicity theory of \cite{GhMe97,Suzuki:RSL:2013} can be derived from the results of the present paper.
\end{remark}

\begin{remark}\label{Adapt:To:Cncl:Ext:Remark}
In the next section we will prove that, for each language $\mathcal{L}_\mathrm{LE}$, all $\mathcal{L}_\mathrm{LE}$-inequalities on which ALBA succeeds are canonical. For this it is necessary to show that ALBA transformations preserve validity with respect to admissible assignments on canonical extensions of LEs. Towards this, it is an easy observation that  all ALBA transformations other than the Ackermann rules preserve validity under admissible assignments.  The potential difficulty with the Ackermann rules stems from the fact that they,  unlike all other rules, involve changing the assignments to propositional variables.
\end{remark}

\section{Canonicity}
\label{section:canonicity}

This section is devoted to proving that all inequalities on which non-distributive ALBA succeeds are canonical. In e.g.\ \cite{ALBAPaper}, the analogous proof for the distributive setting was given in terms of descriptive general frames. Here, however, we will proceed purely algebraically, and our motivation for this is twofold. Firstly, all relevant considerations are inherently algebraic and order-theoretic, and are therefore most perspicuously presented as such. Secondly, as discussed previously, the logics of the present paper do not have a single, established relational semantics and, moreover, the available options for relational semantics are rather  involved.

Fix an $\mathcal{L}_\mathrm{LE}$-algebra  $\bba$ and let $\bbas$ be its canonical extension. We write $\bbas \models_{\mathbb{A}} \phi \leq \psi$ to indicate that $\bbas, v \models \phi \leq \psi$ for all \emph{admissible} assignments $v$, as defined in Subsection \ref{Subsec:Expanded:Land}, page \pageref{admissible:assignment}. Recall that pivotal executions of ALBA are defined on page \pageref{pivotal:approx:rule:application} in Section \ref{Sec:ReductionElimination}.

\begin{theorem}\label{Thm:ALBA:Canonicity}
All $\mathcal{L}_\mathrm{LE}$-inequalities on which ALBA succeeds pivotally are canonical.
\end{theorem}
\begin{proof}
Let $\phi \leq \psi$ be an $\mathcal{L}_\mathrm{LE}$-inequality on which ALBA succeeds pivotally. The required canonicity proof is summarized in the following U-shaped diagram:\\
\begin{center}
\begin{tabular}{l c l}
$\bba \ \ \models \phi \leq \psi$ & &$\bbas \models \phi \leq \psi$\\
$\ \ \ \ \ \ \ \ \ \ \ \Updownarrow$ \\
$\bbas \models_{\bba} \phi \leq \psi$ & & \ \ \ \ \ \ \ \ \ \ \  $\Updownarrow $\\
$\ \ \ \ \ \ \ \ \ \ \ \Updownarrow$\\
$\bbas \models_{\bba} \mathsf{ALBA}(\phi \leq \psi)$
&\ \ \ $\Leftrightarrow$ \ \ \ &$\bbas \models \mathsf{ALBA}(\phi \leq \psi)$
\end{tabular}
\end{center}
The uppermost bi-implication on the left is given by the definition of validity on algebras and $\bba$ being  a subalgebra of $\bbas$. The lower bi-implication on the left is given by Proposition \ref{Propo:Crtns:CanExtns} below. The horizontal bi-implication follows from the facts that, by assumption, $\mathsf{ALBA}(\phi \leq \psi)$ is pure, and that, when restricted to pure formulas, the ranges of admissible and arbitrary assignments coincide. The bi-implication on the right is given by Theorem \ref{Crctns:Theorem}.
\end{proof}
Towards the proof of  Proposition \ref{Propo:Crtns:CanExtns}, the following definitions and lemmas will be useful:
\begin{definition}\label{Syn:Opn:Clsd:Definition}
An $\mathcal{L}_\mathrm{LE}^{+}$-term is \emph{syntactically closed} if in it all occurrences of nominals and $f^\ast\in \mathcal{F}^\ast\setminus \mathcal{F}$ 
are positive, while all occurrences of co-nominals and $g^\ast\in \mathcal{G}^\ast\setminus \mathcal{G}$ 
are negative.

Similarly, an $\mathcal{L}_\mathrm{LE}^{+}$-term is \emph{syntactically open} if in it all occurrences of nominals and $f^\ast\in \mathcal{F}^\ast\setminus \mathcal{F}$ 
are negative, while all occurrences of co-nominals and $g^\ast\in \mathcal{G}^\ast\setminus \mathcal{G}$  
are positive.
\end{definition}
In the following two lemmas, $\alpha$, $\beta_1,\ldots,\beta_n$ and $\gamma_1,\ldots,\gamma_n$ are $\mathcal{L}_\mathrm{LE}^{+}$-terms. We work under the assumption that the values of all parameters occurring in them (propositional variables, nominals and conominals) are given by some fixed admissible assignment.
\begin{lemma}[Righthanded Ackermann lemma for admissible assignments]\label{Ackermann:Dscrptv:Right:Lemma}
Let $\alpha$ be syntactically closed, $p \not \in \mathsf{PROP}(\alpha)$, let $\beta_{1}(p), \ldots, \beta_{n}(p)$ be syntactically closed and positive in $p$, and let $\gamma_{1}(p), \ldots, \gamma_{n}(p)$ be syntactically open and negative in $p$.  Then
\[
{\beta_i}(\alpha) \leq {\gamma_i}(\alpha), \textrm{ for all } 1 \leq i \leq n
\]
iff there exists some $u \in \bba$ such that
\[
\alpha \leq u \textrm{ and } {\beta_i}(u) \leq {\gamma_i}(u), \textrm{ for all } 1 \leq i \leq n.
\]
\end{lemma}

\begin{lemma}[Lefthanded Ackermann lemma for admissible assignments]\label{Ackermann:Dscrptv:Left:Lemma}
Let $\alpha$ be syntactically open, $p \not \in \mathsf{PROP}(\alpha)$, let $\beta_{1}(p), \ldots, \beta_{n}(p)$ be syntactically closed and negative in $p$, and let $\gamma_{1}(p), \ldots, \gamma_{n}(p)$ be syntactically open and positive in $p$.  Then
\[
{\beta_i}(\alpha) \leq {\gamma_i}(\alpha), \textrm{ for all } 1 \leq i \leq n
\]
iff there exists some $u \in \bba$ such that
\[
u \leq \alpha \textrm{ and } {\beta_i}(u) \leq {\gamma_i}(u), \textrm{ for all } 1 \leq i \leq n.
\]
\end{lemma}
The two lemmas above are proved in section \ref{Sec: topological Ackermann}.

\begin{lemma}\label{Syn:Shape:Lemma}
If the system $(S, \mathsf{Ineq})$ is obtained by running ALBA pivotally on some $\mathcal{L}_\mathrm{LE}$-inequality $\phi \leq \psi$, then the left-hand side of  $\mathsf{Ineq}$ as well as the left-hand side of every non-pure inequality  in $S$ is syntactically closed, while the corresponding right-hand sides  are syntactically open.
\end{lemma}
\begin{proof}
Since $\phi \leq \psi$ comes from the base language $\mathcal{L}_\mathrm{LE}$, it is immediate that $\phi$ and $\psi$ are both syntactically open and closed. Since preprocessing does not introduce any symbols not in $\mathcal{L}_\mathrm{LE}$, in any inequality $\phi' \leq \psi'$ resulting from the preprocessing, $\phi'$ and $\psi'$ are both syntactically open and closed. Thus, the claim holds for each initial system $(\varnothing, \phi' \leq \psi')$. In order to complete the proof, it now remains to check that each reduction rule preserves the desired syntactic shape. This is straightforward for all residuation rules. By way of illustration we will consider the right-negative approximation rule and the righthand Ackermann-rule.

The right-negative approximation rule transforms a system $(S, s \leq t'(\gamma / !x))$ with  $-x \prec -t'(!x)$ into $(S \cup \{\gamma \leq \cnomm \}, s \leq t'(\cnomm/!x))$. If $\gamma$ belongs to the original language, it is both syntactically open and closed. Hence, $\gamma \leq \cnomm$ is of the right shape. Moreover,  $-x \prec -t'(!x)$ implies that $x$ occurs positively in $t'$, which by assumption is syntactically open. Hence, $ t'(\cnomm/!x)$ is syntactically open. If $\gamma$ does not belong to $\mathcal{L}_\mathrm{LE}$, then the assumption that all approximation rules are applied pivotally guarantees that $\gamma$ must be a conominal $\cnomn$ (which has been introduced by some previous application of the same approximation rule). Hence $\gamma\leq \cnomm$ is pure.

As for the righthand Ackermann rule, it transforms a system
\[
\{\alpha_1 \leq p, \ldots, \alpha_n \leq p, \beta_1 \leq \gamma_1, \ldots, \beta_m \leq \gamma_m \}
\]
into
\[
\{\beta_1(\alpha/p) \leq \gamma_1(\alpha / p), \ldots, \beta_m(\alpha / p) \leq \gamma_m(\alpha / p) \},
\]
where $\alpha = \alpha_1 \vee \cdots \vee \alpha_n$. Firstly, note that all the pure inequalities among the $\beta_i \leq \gamma_i$ remain unaffected by the rule, and hence remain pure. For non-pure $\beta_i \leq \gamma_i$, we have by assumption that $\beta_i$ is syntactically closed and positive in $p$ while $\gamma_i$ is syntactically open and negative in $p$. Thus, in $\beta_i(\alpha / p)$ each occurrence of a symbol within any occurrence of the subformula $\alpha$ has the same polarity as it had in $\alpha$ before substitution. Hence, since $\alpha$ is syntactically closed, $\beta_i(\alpha /p)$ is syntactically closed. Similarly, in $\gamma_i(\alpha /p)$ any occurrence of a symbol within each occurrence of the subformula $\alpha$ has the opposite polarity from that which it had in $\alpha$ before substitution. Hence, $\gamma_i(\alpha /p)$ is syntactically open. \end{proof}

\begin{prop}[Correctness of pivotal executions of \textrm{ALBA} on canonical extensions under admissible assignments]\label{Propo:Crtns:CanExtns} If \textrm{ALBA} succeeds in reducing an $\mathcal{L}_\mathrm{LE}$-inequality $\phi \leq \psi$ through some pivotal execution and yields $\mathsf{ALBA}(\phi \leq \psi)$, then $\bbas \models_{\bba} \phi \leq \psi$ iff $\bbas \models_{\bba} \mathsf{ALBA} (\phi \leq \psi)$.
\end{prop}
\begin{proof}
It has already been indicated in Remark \ref{Adapt:To:Cncl:Ext:Remark} that the proof is essentially the same as that of Theorem \ref{Crctns:Theorem}. The only difficulty that arises  is that the Ackermann-rules are generally not invertible under admissible assignments (cf.\ \cite[Example 9.1]{ALBAPaper}). However, by Lemmas \ref{Ackermann:Dscrptv:Left:Lemma} and \ref{Ackermann:Dscrptv:Right:Lemma}, in the special case that the left and right hand sides of all non-pure inequalities involved in the application of an Ackermann-rule are, respectively, syntactically closed and open, the rule is sound and invertible
under admissible assignments. By Lemma \ref{Syn:Shape:Lemma}, this requirement on the syntactic shape is always satisfied when the rule is applied in pivotal executions of ALBA.
\end{proof}

\section{ALBA successfully reduces all inductive inequalities}\label{Complete:For:Inductive:Section}

\begin{definition}
An $(\Omega, \epsilon)$-inductive inequality  is {\em definite} if in its critical branches, all Skeleton nodes are SLR nodes.
\end{definition}

\begin{lemma}\label{Pre:Process:Lemma}
Let $\{ \phi_i \leq \psi_i\}$ be the set of inequalities obtained by
preprocessing an $(\Omega, \epsilon)$-inductive $\mathcal{L}_\mathrm{LE}$-inequality $\phi \leq
\psi$. Then each $\phi_i \leq \psi_i$ is a definite $(\Omega,
\epsilon)$-inductive inequality.
\end{lemma}
\begin{proof}
Notice that the distribution during preprocessing only swap the order of Skeleton nodes on (critical) paths, and hence does not affect the goodness of critical branches. Moreover, PIA parts are entirely unaffected, and in particular the side conditions on SRR nodes of critical branches are maintained. Finally, notice that SLR nodes commute exhaustively with $\Delta$-adjoints, and hence all $\Delta$-adjoints are effectively surfaced and eliminated via splitting, thus producing definite inductive inequalities.
\end{proof}

The following definition intends to capture the state of a system after approximation rules have been applied pivotally until no propositional variable remains in   $\mathsf{Ineq}$:

\begin{definition}\label{Stripped:Def}
Call a system $(S, \mathsf{Ineq})$ \emph{$(\Omega, \epsilon)$-stripped} if $\mathsf{Ineq}$ is pure, and for each $\xi \leq \chi \in S$ the following conditions hold:
\begin{enumerate}
\item one of  $-\xi$ and $+\chi$ is pure, and the other is $(\Omega, \epsilon)$-inductive;
%
%
\item every $\epsilon$-critical branch in $-\xi$ and $+\chi$ is PIA.
\end{enumerate}
\end{definition}

\begin{lemma}\label{Stripping:Lemma}
For any definite $(\Omega, \epsilon)$-inductive inequality $\phi \leq \psi$ the system $(\emptyset, \phi \leq \psi)$ can be transformed into an $(\Omega, \epsilon)$-stripped system by the application of approximation rules.
\end{lemma}
\begin{proof}
By assumption, $+\phi$ and $-\psi$ are both definite $(\Omega, \epsilon)$-inductive.  Hence, for any propositional variable occurrences $p$, we can apply an approximation rule  on the first non SLR-node in the path which goes from the root to $p$ (this is always possible thanks to the absence of $\Delta$-adjoints in the Skeleton). These applications are all pivotal. 

Let us show that the resulting system $(S, \mathsf{Ineq})$ is $(\Omega, \epsilon)$-stripped. Clearly, the procedure reduces $\mathsf{Ineq}$ to a pure inequality. Each inequality in $S$ is generated by the application of some approximation rule, and is therefore  either of the form $\nomj \leq \alpha$ or $\beta \leq \cnomm$,  where $+\alpha$ and $-\beta$ are subtrees of $(\Omega, \epsilon)$-inductive trees,  and hence are $(\Omega, \epsilon)$-inductive, as required by item 1 of the definition.


Finally, if in an inequality in $S$ contained  a critical variable occurrence such that its associated path has an SLR node, then, because such a path is by assumption good, this would contradict the fact that the inequality has been generated by a pivotal application of an approximation rule. This shows item 2.

\end{proof}

\begin{definition}\label{Ackermann:Ready:Def}
An $(\Omega, \epsilon)$-stripped system $(S, \mathsf{Ineq})$ is \emph{Ackermann-ready} with respect to a propositional variable $p_i$ with $\epsilon_i = 1$ (respectively, $\epsilon_i = \partial$) if every inequality $\xi \leq \chi \in S$ is of one of the following forms:
\begin{enumerate}
\item $\xi \leq p$ where $\xi$ is pure (respectively, $p \leq \chi$ where $\chi$ is pure), or
\item $\xi \leq \chi$ where neither $- \xi$ nor $+ \chi$ contain any $+p_i$ (respectively, $-p_i$) leaves.
\end{enumerate}
\end{definition}

Note that the right or left Ackermann-rule (depending on whether $\epsilon_i = 1$ or $\epsilon_i = \partial$) is applicable to a system which is Ackermann-ready with respect to $p_i$. In fact, this would still have been the case had we weakened the requirement that $\xi$ and $\chi$ must be pure to simply require that they do not contain $p_i$.

\begin{lemma}\label{Ackermann:Ready:LEmma}
If $(S, \mathsf{Ineq})$ is $(\Omega, \epsilon)$-stripped and $p_i$ is $\Omega$-minimal among propositional variables occurring in $(S, \mathsf{Ineq})$, then $(S, \mathsf{Ineq})$ can be transformed, through the application of residuation- and splitting-rules, into a system which is \emph{Ackermann-ready} with respect to $p_i$.
\end{lemma}
\begin{proof}
If $\xi \leq \chi \in S$ and $- \xi$ and $+ \chi$ contain no $\epsilon$-critical $p_i$-nodes then this inequality already satisfies condition 2 of Definition \ref{Ackermann:Ready:Def}. So suppose that $- \xi$ and $+ \chi$ contain some $\epsilon$-critical $p_i$-node among them. This means $\xi \leq \chi$ is of the form $\alpha \leq \mathsf{Pure}$ with the $\epsilon$-critical $p_i$-node in $\alpha$ and $\mathsf{Pure}$ pure, or of the form $\mathsf{Pure} \leq \delta$ with $\mathsf{Pure}$ pure and the $\epsilon$-critical $p_i$-node in $\delta$. We can now prove by simultaneous induction on $\alpha$ and $\delta$ that these inequalities can be transformed into the form specified by clause 1 of definition \ref{Ackermann:Ready:Def}. 

The base cases are when $- \alpha = -p_i$ and $+\delta = +p_i$. Here the inequalities are in desired shape and no rules need be applied to them. We will only check a few of the inductive cases. If $- \alpha = - (\alpha_1 \vee \alpha_2)$, then 
applying the $\vee$-splitting rule we transform $\alpha_1 \vee \alpha_2 \leq \mathsf{Pure}$ into $\alpha_1 \leq \mathsf{Pure}$ and $\alpha_2 \leq \mathsf{Pure}$. The resulting system is clearly still $(\Omega, \epsilon)$-stripped, and we may apply the inductive hypothesis to $\alpha_1 \leq \mathsf{Pure}$ and $\alpha_2 \leq \mathsf{Pure}$.
If $- \alpha = - f(\overline{\alpha})$, then, as per definition of inductive inequalities and given that $p_i$ is by assumption $\Omega$-minimal, exactly one of the formulas in $\overline{\alpha}$  contains an $\epsilon$-critical node, and all the others (if any) are pure. Assume that the critical node is in $\alpha_j$ for $1\leq j\leq n_f$. Then, applying the appropriate $f$-residuation rule transforms $f(\overline{\alpha})\leq \mathsf{Pure}$ into either $\alpha_j\leq \mathsf{Pure}$ if $\epsilon_f = 1$ or $\mathsf{Pure}\leq \alpha_j$ if $\epsilon_f = \partial$, yielding an $(\Omega, \epsilon)$-stripped system, and the inductive hypothesis is applicable.


If $+ \delta = + g(\overline{\delta})$, then, as per definition of inductive inequalities and given that $p_i$ is by assumption $\Omega$-minimal, exactly one of the formulas in $\overline{\delta}$  contains an $\epsilon$-critical node, and all the others (if any) are pure. Assume that the critical node is in $\delta_j$ for $1\leq j\leq n_g$. Then, applying the appropriate $g$-residuation rule transforms $\mathsf{Pure}\leq g(\overline{\delta})$ into either $\mathsf{Pure}\leq \delta_j$  if $\epsilon_g = 1$ or $\delta_j\leq \mathsf{Pure}$  if $\epsilon_f = \partial$, yielding an $(\Omega, \epsilon)$-stripped system, and the inductive hypothesis is applicable.
%
%
\end{proof}

\begin{lemma}\label{App:Of:Ackermann:Lemma}
Applying the appropriate Ackermann-rule with respect to $p_i$ to an $(\Omega, \epsilon)$-stripped system which is Ackermann-ready with respect to $p_i$, again yields an $(\Omega, \epsilon)$-stripped system.
\end{lemma}
\begin{proof}
Let $(S, \phi \le \psi)$ be an $(\Omega, \epsilon)$-stripped system which is Ackermann-ready with respect to $p_i$. We only consider the case in which the right Ackermann-rule is applied, the case for the left Ackermann-rule being dual. This means that $S = \{\alpha_k\leq p \mid 1 \leq k \leq n \} \cup \{\beta_j(p_i) \leq \gamma_{j}(p_i) \mid 1 \leq j \leq m \}$ where the $\alpha$s are pure and the $-\beta$s and $+\gamma$s contain no $+ p_i$ nodes. Let denote the pure formula $\bigvee_{k=1}^{n} \alpha_k$ by $\alpha$. It is sufficient to show that for each $1 \leq j \leq m$, the trees $- \beta(\alpha / p_i)$ and $+ \gamma(\alpha / p_i)$ satisfy the conditions of Definition \ref{Stripped:Def}. Conditions 2  follows immediately once we notice that, since $\alpha$ is pure and is being substituted everywhere for variable occurrences corresponding to non-critical nodes, $- \beta(\alpha / p_i)$ and $+ \gamma(\alpha / p_i)$ have exactly the same $\epsilon$-critical paths as $- \beta(p_i)$ and $+ \gamma(p_i)$, respectively. Condition 1, namely that $- \beta(\alpha / p_i)$ and $+ \gamma(\alpha / p_i)$ are $(\Omega, \epsilon)$-inductive, also follows using additionally the observation that all new paths that arose from the substitution are variable free.
\end{proof}

\begin{theorem}
$\mathrm{ALBA}$ succeeds on all inductive inequalities, and pivotal executions suffice.
\end{theorem}
\begin{proof}
Let $\phi \leq \psi$ be an $(\Omega, \epsilon)$-inductive inequality. By Lemma \ref{Pre:Process:Lemma}, applying preprocessing yields a finite set of definite $(\Omega, \epsilon)$-inductive inequalities, each of which gives rise to  an initial system $(\emptyset, \phi' \leq \psi')$. By Lemma \ref{Stripping:Lemma}, pivotal applications of the approximation rules convert this system into an $(\Omega, \epsilon)$-stripped system, say $(S_1, \phi'' \leq \psi'')$. Now pick any $\Omega$-minimal variable occurring in $(S_1, \phi'' \leq \psi'')$, say $p$. By lemma \ref{Ackermann:Left:Lemma}, the system can be made Ackermann-ready with respect to $p$ by applying residuation and splitting rules. Now apply the appropriate Ackermann-rule to eliminate $p$ from the system. By lemma \ref{App:Of:Ackermann:Lemma}, the result is again an $(\Omega, \epsilon)$-stripped system, now containing one propositional variable less, and to which lemma \ref{Ackermann:Left:Lemma} can again be applied. This process is iterated until all occurring propositional variables are eliminated and a pure system is obtained.
\end{proof}

\section{Constructive canonicity of ALBA inequalities}
\label{sec:conclusions}
The problem of canonicity, defined as the preservation of inequalities under the canonical extension construction, can be meaningfully investigated in a constructive meta-theory, as first shown in the work of Ghilardi and Meloni \cite{GhMe97}. Indeed, the canonical extension construction, as given in \cite{GH01,DGP}, is formulated in terms of general filters and ideals, and does not depend on any form of the axiom of choice (such as the existence of `enough' optimal filter-ideal pairs). Thus, while the constructive canonical extension need not be perfect anymore, the canonical embedding retains the properties of denseness and compactness. We have already observed in Remark \ref{rem: towards constructive can} that the soundness of the approximation rules does not rely on the fact that nominals and co-nominals are interpreted as completely join-irreducible and meet-irreducible elements respectively, but only on the fact that these elements completely join-generate and completely meet-generate the canonical extension, respectively. Since, by denseness, the closed and open elements have precisely these generating properties, one can take nominals and co-nominals to range over these sets, respectively, without affecting the soundness of the rules. In fact, all the other ALBA rules also remain sound when interpreted in the constructive canonical extensions. Hence, the constructive canonicity of the inequalities on which ALBA succeeds immediately follows by the same argument illustrated in Section \ref{section:canonicity}. In particular, all inductive inequalities are constructively canonical, and since inductive inequalities include the fragments treated in \cite{GhMe97,Suzuki:RSL:2013} (cf.\ Example \ref{rem:ghilardi suzuki}), these results follow from those in the present paper.

\bibliographystyle{siam}
\bibliography{WANonDist}

\section{Appendix}\label{appendix}

\subsection{Topological properties of the modalities and their adjoints}

Fix a language $\mathcal{L}_\mathrm{LE}$, and an $\mathcal{L}_\mathrm{LE}$-algebra  $\bba = (L, \mathcal{F}^\bba, \mathcal{G}^\bba)$ for the rest of this section. This section collects the relevant properties of the additional operations of $\bba$ and their adjoints, which will be critical for the proof of the ``topological versions'' of the Ackermann lemmas in Section \ref{Sec: topological Ackermann}. These properties are 
well known to hold in closely related settings (e.g.\ \cite{ALBAPaper}), or following easily from other results (viz.\ \cite[Lemma 3.4]{DGP}).
In what follows we use the terminology $\partial$-monotone (respectively $\partial$-antitone, $\partial$-positive, $\partial$-negative, $\partial$-open, $\partial$-closed) to mean its opposite, i.e. antitone (respectively, monotone, negative, positive, closed, open). By $1$-monotone (respectively antitone, positive, negative, open, closed) we simply mean monotone (antitone, positive, negative, open, closed). Also in symbols, for example we will write $(O(\bbas))^1$ for $O(\bbas)$ and $(O(\bbas))^\partial$ for $K(\bbas)$, and similarly $(K(\bbas))^1$ for $K(\bbas)$ and $K(\bbas)^\partial$ for $O(\bbas)$. This convention generalizes to order-types and tuples in the obvious way.  Thus, for example,  $(O(\bbas))^{\epsilon}$ is the cartesian product of sets with $O(\bbas)$ as $i$th coordinate where $\epsilon_i = 1$ and $K(\bbas)$ for $j$th coordinate where $\epsilon_j = \partial$.

\begin{lemma}
\label{cor: open upset to open upset for white box} For all $f\in \mathcal{F}^\bba$, $g\in \mathcal{G}^\bba$, $\overline{k}\in (K(\bbas))^{\epsilon_f}$, and $\overline{o}\in (O(\bbas))^{\epsilon_g}$,  
\begin{enumerate}
\item $g(\overline{ o}) \in \obbas$,
\item $f(\overline{ k}) \in \kbbas$.
\end{enumerate}
\end{lemma}
\begin{proof}
These facts straightforwardly follow from  the fact that each $f\in \mathcal{F}^{\bbas}$ (resp.\  $g\in \mathcal{G}^{\bbas}$) is the $\sigma$-extension  (resp.\ $\pi$-extension) of the corresponding operation in $\bba$: for instance, $g(\overline{ o})= g^{\pi}(\overline{ o})  = \bigvee \{ g(\overline{a}) \mid \overline{a} \in \bba^{\epsilon_g} \textrm{ and } \overline{a} \leq^{\epsilon_g} \overline{o} \}$, and $g(\overline{a}) \in \bba$.
\end{proof}

\begin{lemma}
\label{cor:congenial top for whites and blacks} For all $f\in \mathcal{F}^\bba$, $g\in \mathcal{G}^\bba$, $\overline{k}\in (K(\bbas))^{\epsilon_g}$, and $\overline{o}\in (O(\bbas))^{\epsilon_f}$,
\begin{enumerate}
\item $g(\overline{ k}) \in \kbbas$,
\item $f(\overline{ o}) \in \obbas$.
\end{enumerate}
\end{lemma}
\begin{proof}
1. By assumption, $\overline{ k}  = \bigwedge \{\overline{a} \mid \overline{a} \in \bba^{\epsilon_g} \textrm{ and } \overline{k} \leq^{\epsilon_g} \overline{a}\}$. Hence $g(\overline{ k})  = g(\bigwedge \{\overline{a} \mid \overline{a} \in \bba^{\epsilon_g} \textrm{ and } \overline{k} \leq^{\epsilon_g} \overline{a}\}) = \bigwedge \{g(\overline{a}) \mid \overline{a} \in \bba^{\epsilon_g} \textrm{ and } \overline{k} \leq^{\epsilon_g} \overline{a}\}$. Since $g(\overline{a})  \in \bba$ for all $\overline{a} \in \bba^{\epsilon_g}$, we immediately have $g(\overline{ k})\in \kbbas$, as required.


2. 
is  order-dual to 1.

\end{proof}
\begin{lemma}\label{Blk:Diam:c:Clsd:Lemma}
For all $f\in \mathcal{F}^\bba$, $g\in \mathcal{G}^\bba$, $1\leq i\leq n_f$, and $1\leq j\leq n_g$,
,
\begin{enumerate}
\item If $\epsilon_g(j) = 1$, then $g^\flat_j(\overline{k})\in \kbbas$ for every $\overline{k}\in (K(\bbas))^{\epsilon_{g^\flat_j}}$;
\item If $\epsilon_g(j) = \partial$, then $g^\flat_j(\overline{o})\in \obbas$ for every $\overline{o}\in (O(\bbas))^{\epsilon_{g^\flat_j}}$;
\item If $\epsilon_f(i) = 1$, then $f^\sharp_i(\overline{o})\in \obbas$ for every $\overline{o}\in (O(\bbas))^{\epsilon_{f^\sharp_i}}$;
\item If $\epsilon_f(i) = \partial$, then $f^\sharp_i(\overline{k})\in \kbbas$ for every $\overline{k}\in (K(\bbas))^{\epsilon_{f^\sharp_i}}$.
%
%
%
%
%
%
%
%
\end{enumerate}
\end{lemma}
\begin{proof}
1. By denseness, $g^\flat_j(\overline{k}) = \bigwedge \{ o \in \obbas \mid g^\flat_j(\overline{k})  \leq o \}$. Let $Y := \{ o \in \obbas \mid g^\flat_j(\overline{k})  \leq o \}$ and $X : = \{ a \in \bba \mid g^\flat_j(\overline{k})  \leq a \}$
To show that $g^\flat_j(\overline{k}) \in \kbbas$, it is enough to show that $\bigwedge X = \bigwedge Y$.

Since clopens are opens, $X \subseteq Y$, so $\bigwedge Y \leq \bigwedge X$. In order to show that $\bigwedge X \leq \bigwedge Y$, it suffices to show that for every $o \in Y$ there exists some $a \in X$ such that $a \leq o$. Let $o \in Y$, i.e., $g^\flat_j(\overline{k})  \leq o$. By residuation, $k_j \leq g(\overline{k}[o/k_j])$, where $\overline{k}[o/k_j]$ denotes the $n_g$-array obtained by replacing the $j$th coordinate of $\overline{k}$ by $o$. Notice that $\overline{k}[o/k_j]\in (O(\bbas))^{\epsilon_g}$. This immediately follows from the fact that by assumption, $\epsilon_{g^\flat_j}(\ell) = \epsilon_g(\ell) = 1$ if $\ell = j$ and $\epsilon_{g^\flat_j}(\ell) = \epsilon_g^\partial(\ell)$ if $\ell\neq j$.

Since $k_j \in \kbbas$, and $g(\overline{k}[o/k_j]) = g^\pi(\overline{k}[o/k_j]) = \bigvee \{g(\overline{a})\mid \overline{a}\in \bba^{\epsilon_g} \textrm{ and } \overline{a} \leq^{\epsilon_g} \overline{k}[o/k_j] \}$ and $g(\overline{ a })\in \bba \subseteq \obbas$, we may apply compactness and get that
$k_j \leq g(\overline{ a_1 })\vee \cdots \vee g(\overline{ a_n })$ for some $\overline{ a_1 },\ldots, \overline{ a_n }\in \bba^{\epsilon_g}$ s.t.\ $\overline{ a_1 },\ldots, \overline{ a_n } \leq^{\epsilon_g} \overline{k}[o/k_j]$. Let $\overline{ a} = \overline{ a_1 } \vee^{\epsilon_g} \cdots \vee^{\epsilon_g} \overline{ a_n }$. The $\epsilon_g$-monotonicity of $g$ implies that $k_j \leq g(\overline{ a_1 })\vee \cdots \vee g(\overline{ a_n })\leq  g(\overline{ a })$, and hence $g^\flat_j(\overline{a}[k_j/a_j])\leq a_j$. The proof is complete if we show that
$g^\flat_j(\overline{k}) \leq g^\flat_j(\overline{a}[k_j/a_j])$. By the $\epsilon_{g^\flat_j}$-monotonicity of $g^\flat_j$, it is enough to show that $\overline{k} \leq^{\epsilon_{g^\flat_j}}\overline{a}[k_j/a_j]$. Since the two arrays coincide in their $j$th coordinate, we only need to check that this is true for every $\ell \neq j$. Recall that $\epsilon_{g^\flat_j}(\ell) = \epsilon_g^\partial(\ell)$ if $\ell\neq j$. Hence, the statement immediately follows from this and the fact that, by construction, $\overline{ a }\leq^{\epsilon_g} \overline{k}[o/k_j]$.

2.\ 3.\ and 4.\ are order-variants of 1.
\end{proof}

\begin{lemma}\label{lemma:uncongenial for the whites}
For all $f\in \mathcal{F}$ and $g\in \mathcal{G}$,
%
\begin{enumerate}
\item $g(\bigvee^{\epsilon_g(1)} \mathcal{U}_1,\ldots, \bigvee^{\epsilon_g(n_g)} \mathcal{U}_{n_g}) = \bigvee \{g(u_1,\ldots,u_{n_g})\ |\ u_j \in \mathcal{U}_j\mbox{ for every } 1\leq j\leq n_g\}$ for every  $n_g$-tuple $(\mathcal{U}_1,\ldots, \mathcal{U}_{n_g})$ such that
    $\mathcal{U}_j\subseteq \obbas^{\epsilon_g(j)}$ and $\mathcal{U}_j$ is $\epsilon_g(j)$-up-directed for each $1\leq j\leq n_g$.
\item $f(\bigwedge^{\epsilon_f(1)} \mathcal{D}_1,\ldots, \bigwedge^{\epsilon_f(n_f)} \mathcal{D}_{n_f}) = \bigwedge \{f(d_1,\ldots,d_{n_f})\ |\ d_j \in \mathcal{D}_j\mbox{ for every } 1\leq j\leq n_f\}$ for every  $n_f$-tuple $(\mathcal{D}_1,\ldots, \mathcal{D}_{n_f})$ such that
    $\mathcal{D}_j\subseteq \kbbas^{\epsilon_f(j)}$ and $\mathcal{D}_j$ is $\epsilon_f(j)$-down-directed for each $1\leq j\leq n_f$.
%
%
%
%
%
\end{enumerate}
\end{lemma}
\begin{proof}
1. The `$\geq$' direction easily follows from the $\epsilon_g$-monotonicity of $g$. Conversely, by denseness it is enough to show that if $c \in \kbbas$ and $c \leq g(\bigvee^{\epsilon_g(1)} \mathcal{U}_1,\ldots, \bigvee^{\epsilon_g(n_g)} \mathcal{U}_{n_g})$, then $c \leq g(u_1,\ldots,u_{n_g})$ for some tuple $(u_1,\ldots,u_{n_g})$ such that $u_j\in \mathcal{U}_j$ for each $1\leq j\leq n_g$. Indeed, if $c \leq g(\bigvee^{\epsilon_g(1)} \mathcal{U}_1,\ldots, \bigvee^{\epsilon_g(n_g)} \mathcal{U}_{n_g})$ then $g^\flat_1(c, \overline{\bigvee^{\epsilon_g} \mathcal{U}}) \leq^{\epsilon(1)} \bigvee^{\epsilon_g(1)} \mathcal{U}_1$, where, to enhance readability, we suppress sub- and superscripts and write $\overline{\bigvee^{\epsilon_g} \mathcal{U}}$ for $(\bigvee^{\epsilon_g(2)} \mathcal{U}_2, \ldots, \bigvee^{\epsilon_g(n_g)} \mathcal{U}_{n_g})$. If $\epsilon_g(1) = 1$, then  $\epsilon_{g^\flat_1}(1) = 1$ and $\epsilon_{g^\flat_1}(\ell) = \epsilon_{g}^\partial(\ell)$ for every $2\leq \ell\leq n_g$. Hence $\mathcal{U}_\ell\subseteq \obbas^{\epsilon_g(\ell)} = \kbbas^{\epsilon_{g^\flat_1}(\ell)}$, hence $\bigvee^{\epsilon_g(\ell)} \mathcal{U}_\ell = \bigwedge^{\epsilon_{g^\flat_1}(\ell)}\mathcal{U}_\ell\in \kbbas^{\epsilon_{g^\flat_1}(\ell)}$ for every $2\leq \ell\leq n_g$.
By Lemma \ref{Blk:Diam:c:Clsd:Lemma}(1), this implies that $g^\flat_1(c, \overline{\bigvee^{\epsilon_g} \mathcal{U}})\in \kbbas$. Hence, by  compactness,  $g^\flat_1(c, \overline{\bigvee^{\epsilon_g} \mathcal{U}})\leq \bigvee_{i=1}^n o_i$ for some $o_1,\ldots,o_n\in \mathcal{U}_1$. Since $\mathcal{U}_1$ is up-directed, $\bigvee_{i=1}^n o_i \leq u_1$ for some $u_1 \in \mathcal{U}_1$. Hence $c \leq g(u_1, \overline{\bigvee^{\epsilon_g} \mathcal{U}})$. The same conclusion can be reached via a similar argument if $\epsilon_g(1) = \partial$. Therefore,  $g^\flat_2(u_1, c, \overline{\bigvee^{\epsilon_g} \mathcal{U}}) \leq^{\epsilon_g(2)} \bigvee^{\epsilon_g(2)} \mathcal{U}_2$, where $\overline{\bigvee^{\epsilon_g} \mathcal{U}}$ now stands for $(\bigvee^{\epsilon_g(3)} \mathcal{U}_3, \ldots, \bigvee^{\epsilon_g(n_g)} \mathcal{U}_{n_g})$.  By applying the same reasoning, we can conclude that $c \leq g(u_1, u_2, \overline{\bigvee^{\epsilon_g} \mathcal{U}})$ for some $u_2\in \mathcal{U}_2$, and so on. Hence, we can then construct a sequence $u_j\in \mathcal{U}_j$ for $1\leq j\leq n_g$ such that $c\leq g(u_1,\ldots u_{n_g})$, as required.


2.\ is order-dual to 1.
%
\end{proof}

\begin{lemma}
\label{lemma: uncongenial for the blacks} For all $f\in \mathcal{F}$, $g\in \mathcal{G}$, $1\leq i\leq n_f$, and $1\leq j\leq n_g$,
\begin{enumerate}
\item If $\epsilon_g(j) = 1$, then $g^\flat_j(\bigwedge^{\epsilon_{g^\flat_j}(1)} \mathcal{D}_1,\ldots, \bigwedge^{\epsilon_{g^\flat_j}(n_g)} \mathcal{D}_{n_g}) = \bigwedge \{g^\flat_j(d_1,\ldots,d_{n_g})\ |\ d_h \in \mathcal{D}_h\mbox{ for every } 1\leq h\leq n_g\}$ for every  $n_g$-tuple $(\mathcal{D}_1,\ldots, \mathcal{D}_{n_g})$ such that
    $\mathcal{D}_h\subseteq \kbbas^{\epsilon_{g^\flat_j}(h)}$ and $\mathcal{D}_h$ is $\epsilon_{g^\flat_j}(h)$-down-directed for each $1\leq h\leq n_g$.

    \item If $\epsilon_g(j) = \partial$, then $g^\flat_j(\bigvee^{\epsilon_{g^\flat_j}(1)} \mathcal{U}_1,\ldots, \bigvee^{\epsilon_{g^\flat_j}(n_g)} \mathcal{U}_{n_g}) = \bigvee \{g^\flat_j(u_1,\ldots,u_{n_g})\ |\ u_h \in \mathcal{U}_h\mbox{ for every } 1\leq h\leq n_g\}$ for every  $n_g$-tuple $(\mathcal{U}_1,\ldots, \mathcal{U}_{n_g})$ such that
    $\mathcal{U}_h\subseteq \obbas^{\epsilon_{g^\flat_j}(h)}$ and $\mathcal{U}_h$ is $\epsilon_{g^\flat_j}(h)$-up-directed for each $1\leq h\leq n_g$.

\item If $\epsilon_f(i) = 1$, then $f^\sharp_i(\bigvee^{\epsilon_{f^\sharp_i}(1)} \mathcal{U}_1,\ldots, \bigvee^{\epsilon_{f^\sharp_i}(n_f)} \mathcal{U}_{n_f}) = \bigvee \{f^\sharp_i(u_1,\ldots,u_{n_f})\ |\ u_h \in \mathcal{U}_h\mbox{ for every } 1\leq h\leq n_f\}$ for every  $n_f$-tuple $(\mathcal{U}_1,\ldots, \mathcal{U}_{n_f})$ such that
    $\mathcal{U}_h\subseteq \obbas^{\epsilon_{f^\sharp_i}(h)}$ and $\mathcal{U}_h$ is $\epsilon_{f^\sharp_i}(h)$-up-directed for each $1\leq h\leq n_f$.
\item If $\epsilon_f(i) = \partial$, then $f^\sharp_i(\bigwedge^{\epsilon_{f^\sharp_i}(1)} \mathcal{D}_1,\ldots, \bigwedge^{\epsilon_{f^\sharp_i}(n_f)} \mathcal{D}_{n_f}) = \bigwedge \{f^\sharp_i(d_1,\ldots,d_{n_f})\ |\ d_h \in \mathcal{D}_h\mbox{ for every } 1\leq h\leq n_f\}$ for every  $n_f$-tuple $(\mathcal{D}_1,\ldots, \mathcal{D}_{n_f})$ such that
    $\mathcal{D}_h\subseteq \kbbas^{\epsilon_{f^\sharp_i}(h)}$ and $\mathcal{D}_h$ is $\epsilon_{f^\sharp_i}(h)$-down-directed for each $1\leq h\leq n_f$.
%
%
%
%
%
%
%
%
%
%
\end{enumerate}
\end{lemma}
\begin{proof}
3.\ The `$\geq$' direction easily follows from the
$\epsilon_{f^\sharp_i}$-monotonicity of $f^\sharp_i$. For the converse inequality, by denseness it is enough to show that if $c\in \kbbas$ and
$c\leq f^\sharp_i(\bigvee^{\epsilon_{f^\sharp_i}(1)} \mathcal{U}_1,\ldots, \bigvee^{\epsilon_{f^\sharp_i}(n_f)} \mathcal{U}_{n_f})$, then $c\leq f^\sharp_i(u_1,\ldots,u_{n_f})$ for some tuple $(u_1,\ldots,u_{n_f})$ such that $u_h \in \mathcal{U}_h$ for every $1\leq h\leq n_f$. By residuation,
$c\leq f^\sharp_i(\bigvee^{\epsilon_{f^\sharp_i}(1)} \mathcal{U}_1,\ldots, \bigvee^{\epsilon_{f^\sharp_i}(n_f)} \mathcal{U}_{n_f})$ implies that $f(
\bigvee^{\epsilon_{f^\sharp_i}(1)} \mathcal{U}_1,\ldots,c,\ldots, \bigvee^{\epsilon_{f^\sharp_i}(n_f)} \mathcal{U}_{n_f})\leq \bigvee^{\epsilon_{f^\sharp_i}(i)} \mathcal{U}_i$.
The assumption $\epsilon_f(i) = 1$ implies that  $\epsilon_{f^\sharp_i}(i) = 1$ and $\epsilon_{f^\sharp_i}(\ell) = \epsilon_{f}^\partial(\ell)$ for every $\ell\neq i$. Hence $\mathcal{U}_\ell\subseteq \obbas^{\epsilon_{f^\sharp_i}(\ell)} = \kbbas^{\epsilon_f(\ell)}$, and  $\mathcal{U}_\ell$ is $\epsilon_f(\ell)$-down-directed 
for every $\ell\neq i$. Recalling that $\bigvee^{\epsilon_{f^\sharp_i}(\ell)}$ coincides with $\bigwedge^{\epsilon_{f}(\ell)}$,
we can apply  Lemma \ref{lemma:uncongenial for the whites}(2) and get:
\[f(\bigvee{}^{\epsilon_{f^\sharp_i}(1)} \mathcal{U}_1,\ldots,c,\ldots, \bigvee{}^{\epsilon_{f^\sharp_i}(n_f)} \mathcal{U}_{n_f}) =
\bigwedge \{f(u_1,\ldots,c,\ldots, u_{n_f})\mid u_\ell\in \mathcal{U}_\ell \mbox{ for every } \ell\neq i\}.\]
Hence, by compactness, 
$f(\bigvee^{\epsilon_{f^\sharp_i}(1)} \mathcal{U}_1,\ldots,c,\ldots, \bigvee^{\epsilon_{f^\sharp_i}(n_f)} \mathcal{U}_{n_f})\leq \bigvee^{\epsilon_{f^\sharp_i}(i)} \mathcal{U}_i$ implies that
\[\bigwedge_{1\leq j\leq m}\{f(o_1^{(j)},\ldots,c,\ldots, o_{n_f}^{(j)})\mid o_\ell^{(j)}\in \mathcal{U}_\ell\mbox{ for all } \ell\neq i \}\leq o_i^{(1)}\vee\cdots \vee o_i^{(n)}\]
for some $o_i^{(1)},\ldots, o_i^{(n)}\in \mathcal{U}_i$.
The assumptions that $\epsilon_f(i) = 1$ and that each $\mathcal{U}_h$ is $\epsilon_{f^\sharp_i}(h)$-up-directed for every $1\leq h\leq n_f$ imply that $\mathcal{U}_i$ is up-directed and $\mathcal{U}_\ell$ is $\epsilon_f(\ell)$-down-directed for each $\ell\neq i$. Hence, some $u_1,\ldots, u_{n_f}$ exist such that  $u_\ell\leq^{\epsilon_f(\ell)}\bigwedge^{\epsilon_f(\ell)}_{1\leq j\leq m}o_\ell^{(j)}$  and $o_i^{(1)}\vee\cdots \vee o_i^{(n)}\leq u_i.$
The $\epsilon_f$-monotonicity of $f$ implies  the following chain of inequalities:
\begin{center}
\begin{tabular}{r c l}
$f(u_1,\ldots,c, \ldots, u_{n_f})$ & $\leq$ &$ f(\bigwedge^{\epsilon_f(1)}_{1\leq j\leq m}o_1^{(j)},\ldots,c, \ldots, \bigwedge^{\epsilon_f(n_f)}_{1\leq j\leq m}o_{n_f}^{(j)})$\\
& $\leq$ &$ \bigwedge_{1\leq j\leq m}\{f(o_1^{(j)},\ldots,c,\ldots, o_{n_f}^{(j)})\mid o_\ell^{(j)}\in \mathcal{U}_\ell\mbox{ for all } \ell\neq i \}$\\
& $\leq$ &$ o_i^{(1)}\vee\cdots \vee o_i^{(n)}$\\
& $\leq$ &$ u_i$,\\
\end{tabular}
\end{center}
which implies that $c\leq f^\sharp_i(u_1,\ldots,u_{n_f})$, as required.

1.\ 2.\ and 4.\ are order-variants of 3.
%
%
\end{proof}

\subsection{Proof of the restricted Ackermann lemmas (lemmas \ref{Ackermann:Dscrptv:Right:Lemma} and \ref{Ackermann:Dscrptv:Left:Lemma})}
\label{Sec: topological Ackermann}

For any $\mathcal{L}_{\mathrm{LE}}^+$-formula $\phi$, any  $\mathcal{L}_{\mathrm{LE}}$-algebra $\bba$  and assignment $V$ on $\bbas$, we write $\phi(V)$ to denote the extension of $\phi$ in $\bbas$ under the assignment $V$. Note that if $\phi$ is in the basic signature  and $V$ is  admissible, then $\phi(V) \in \mathbb{A}$. This, however, is not the case for formulas from the expanded signature or for non-admissible valuations.
Let $p$ be a propositional variable occurring in $\phi$ and $V$ be any assignment. For any $x \in \bbas$, let $V[p:= x]$ be the assignment which is identical to $V$ except that it assigns $x$ to $p$. Then $x\mapsto \phi(V[p:= x])$  defines an operation on $\bbas$, which we will denote $\phi^{V}_{p}(x)$.


\begin{lemma}\label{Syn:Opn:Clsd:Appld:ClsdUp:Lemma}
Let $\phi$ be syntactically closed and $\psi$ syntactically open. Let $V$ be an admissible assignment, $c \in K(\bbas)$  and $o \in O(\bbas)$.
\begin{enumerate}
\item
    \begin{enumerate}
    \item If $\phi(p)$ is positive in $p$, then $\phi^{V}_{p}(c) \in K(\bbas)$, and
    \item if $\psi(p)$ is negative in $p$, then $\psi^{V}_{p}(c) \in O(\bbas)$.
    \end{enumerate}
\item
    \begin{enumerate}
    \item If $\phi(p)$ is negative in $p$, then $\phi^{V}_{p}(o) \in K(\bbas)$, and
    \item if $\psi(p)$ is positive in $p$, then $\psi^{V}_{p}(o) \in O(\bbas)$.
    \end{enumerate}
\end{enumerate}
\end{lemma}
\begin{proof}
We prove (1) by simultaneous induction of $\phi$ and $\psi$. Assume that $\phi(p)$ is positive in $p$ and that $\psi(p)$ is negative in $p$. The base cases of the induction are those when $\phi$ is of the form $\top$, $\bot$, $p$, $q$ (for propositional variables $q$ different from $p$) or $\nomi$, and when $\psi$ is of the form $\top$, $\bot$, $q$ (for propositional variables $q$ different from $p$), or $\cnomm$. (Note that the $\phi$ cannot be a co-nominal $\cnomm$, since it is syntactically closed. Also, $\psi$ cannot be $p$ or a nominal $\nomi$, since is negative in $p$ and syntactically open, respectively.) These cases follow by noting (1) that $V[p:= c](\bot) = 0 \in \mathbb{A}$, $V[p:= c](\top) = 1 \in \mathbb{A}$, and $V[p:= c](q) = V(q) \in \mathbb{A}$, (2) that $V[p:= c](p) = c \in K(\bbas)$ and $V[p:= c](\nomi) \in J^{\infty}(\bbas) \subset K(\bbas)$, and (3) that $V[p:= c](\cnomm) \in M^{\infty}(\bbas) \subset O(\bbas)$ (see discussion on page \pageref{Page:JIr:Clsd:MIr:Opn}).

For the remainder of the proof we will not need to refer to the valuation $V$ and will hence omit reference to it. We will accordingly write $\phi$ and $\psi$ for $\phi^V_p$ and $\psi^V_p$, respectively.

In the cases $\phi(p) = f(\overline{\phi'(p)})$ for $f\in \mathcal{F}$ or $\phi(p) = \phi_1(p) \wedge \phi_2(p)$, both $\phi_1(p)$ and $\phi_2(p)$ are syntactically closed and positive in $p$, and each $\phi'_i(p)$ in $\overline{\phi'(p)}$ is $\epsilon_f(i)$-syntactically closed
and $\epsilon_f(i)$-positive in $p$. Hence, the claim follows by the inductive hypothesis, and lemma \ref{cor: open upset to open upset for white box}(2) and the fact that meets of closed elements are closed, respectively.

Similarly, if $\psi(p) = g(\overline{\psi'(p)})$ for $g\in \mathcal{G}$ or $\psi(p) = \psi_1(p) \vee \psi_2(p)$, then both $\psi_1(p)$ and $\psi_2(p)$ are syntactically open and negative in $p$, and each $\psi'_i(p)$ in $\overline{\psi'(p)}$ is syntactically $\epsilon_g(i)$-open and $\epsilon_g(i)$-negative in $p$. Hence,  the claim follows by the inductive hypothesis, and lemma \ref{cor: open upset to open upset for white box}(1) and the fact that joins of open elements are open, respectively.

If $\phi(p) = g_i^\flat(\overline{\phi'(p)})$ for some $g\in \mathcal{G}$ such that $\epsilon_g(i) = 1$, then each $\phi_h'(p)$ is syntactically $\epsilon_{g_i^\flat}(h)$-closed and $\epsilon_{g_i^\flat}(h)$-positive in $p$. Hence,  the claim follows by the inductive hypothesis and lemma \ref{Blk:Diam:c:Clsd:Lemma}(1).

The remaining cases are similar and are proven making use of the remaining items of lemma \ref{Blk:Diam:c:Clsd:Lemma}.

Item (2) can similarly be proved by simultaneous induction on negative $\phi$ and positive $\psi$. 
%
\end{proof}

\begin{lemma}\label{Esakia:Syn:Clsd:Opn:Lemma} Let $\phi(p)$ be syntactically closed, $\psi(p)$ syntactically open, $V$ an admissible assignment, $\mathcal{D} \subseteq \kbbas$ be down-directed, and $\mathcal{U} \subseteq \obbas$ be up-directed.
\begin{enumerate}
\item
    \begin{enumerate}
    \item If $\phi(p)$ is positive in $p$, then $\phi^V_p(\bigwedge\mathcal{D}) = \bigwedge\{ \phi^V_p (d)\mid d\in \mathcal{D}\}$, and
    \item if $\psi(p)$ is negative in $p$, then $\psi^V_p(\bigwedge\mathcal{D}) = \bigvee\{\psi^V_p (d)\mid d\in \mathcal{D}\}$.
    \end{enumerate}
\item
    \begin{enumerate}
    \item If $\phi(p)$ is negative in $p$, then $\phi^V_p(\bigvee\mathcal{U}) = \bigwedge\{\phi^V_p (u)\mid u\in \mathcal{U}\}$, and
    \item if $\psi(p)$ is positive in $p$, then $\psi^V_p(\bigvee\mathcal{U}) = \bigvee\{ \psi^V_p (u)\mid u\in \mathcal{U}\}$.
    \end{enumerate}
\end{enumerate}
\end{lemma}

\begin{proof}
We prove (1) by simultaneous induction on $\phi$ and $\psi$. The base cases of the induction on $\phi$ are those when it is of the form $\top$, $\bot$, $p$, a propositional variable $q$ other than $p$, or $\nomi$, and for $\psi$ those when it is of the form $\top$, $\bot$, a propositional variable $q$ other than $p$ or $\cnomm$. In each of these cases the claim is trivial.

For the remainder of the proof we will omit reference to the assignment $V$, and simply write $\phi$ and $\psi$ for $\phi^V_p$ and $\psi^V_p$, respectively.

In the cases in which $\phi(p) = \phi_1(p) \vee \phi_2(p)$, $\phi(p) = \phi_1(p) \wedge \phi_2(p)$, $\phi(p) = f(\overline{\phi(p)})$, $\phi(p) = g(\overline{\phi(p)})$, $\psi(p) = \psi_1(p) \wedge \psi_2(p)$, $\psi(p) = \psi_1(p) \vee \psi_2(p)$, $\psi(p) = f(\overline{\psi(p)})$, $\psi(p) = g(\overline{\psi(p)})$,  we have that $\phi_1$ and $\phi_2$
are syntactically closed and positive in $p$ and $\psi_1$ and $\psi_2$
are syntactically open and negative in $p$, and moreover, each $\phi_h(p)$ in $\overline{\phi(p)}$ is syntactically $\epsilon_g(h)$-closed and $\epsilon_g(h)$-positive in $p$, and each $\psi_j(p)$ in $\overline{\psi(p)}$ is syntactically $\epsilon_f(j)$-open and $\epsilon_f(j)$-negative in $p$.

Hence, when $\phi(p) = \phi_1(p) \wedge \phi_2(p)$ and  $\psi(p) = \psi_1(p) \vee \psi_2(p)$, the claim follows by the inductive hypothesis and the associativity of, respectively, meet and join.

If $\phi(p) = \phi_1(p) \vee \phi_2(p)$, then
\begin{center}
\begin{tabular}{r c ll}
$\phi(\bigwedge \mathcal{D})$ &$ = $& $\phi_1(\bigwedge \mathcal{D})\vee \phi_2(\bigwedge \mathcal{D})$ & \\
&$ = $& $\bigwedge \{\phi_1(c_i)\mid c_i\in \mathcal{D}\}\vee \bigwedge \{\phi_2(c_i)\mid c_i\in \mathcal{D}\}$ & (induction hypothesis)\\
&$ = $& $\bigwedge \{\phi_1(c_i)\vee \phi_2(c_j)\mid c_i, c_j\in \mathcal{D}\}$ & ($\ast$)\\
&$ = $& $\bigwedge \{\phi_1(c)\vee \phi_2(c)\mid c\in \mathcal{D}\}$ & ($\phi$ monotone and $\mathcal{D}$ down-directed)\\
&$ = $& $\bigwedge \{\phi(c)\mid c\in \mathcal{D}\}$, & \\
\end{tabular}
\end{center}
where the equality marked with ($\ast$) follows from a restricted form of distributivity enjoyed by canonical extensions of general bounded lattices (cf.\ \cite[Lemma 3.2]{GH01}), applied to the family $\{A_1, A_2\}$ such that $A_i: = \{\phi_i(c_j)\mid c_j\in \mathcal{D}\}$ for $i\in \{1, 2\}$. Specifically, the monotonicity in $p$ of $\phi_i(p)$ and $\mathcal{D}$ being down-directed imply that $A_1$ and $A_2$ are down-directed subsets, which justifies the application of \cite[Lemma 3.2]{GH01}.

If $\phi(p) = g(\overline{\phi(p)})$, then
\begin{center}
$\phi(\bigwedge\mathcal{D}) = g(\phi_1(\bigwedge\mathcal{D}),\ldots,\phi_{n_g}(\bigwedge\mathcal{D}))
= g(\bigwedge_{d \in \mathcal{D}}^{\epsilon_g(1)} \phi_1(d),\ldots, \bigwedge_{d \in \mathcal{D}}^{\epsilon_g(n_g)} \phi_{n_g}(d))$.
\end{center}
%
%
%
The second equality above holds by the inductive hypothesis. To finish the proof, we need to show that
\begin{center}
$g(\bigwedge_{d \in \mathcal{D}}^{\epsilon_g(1)} \phi_1(d),\ldots, \bigwedge_{d \in \mathcal{D}}^{\epsilon_g(n_g)} \phi_{n_g}(d)) = \bigwedge_{d \in \mathcal{D}} g( \phi_1(d)\ldots,\phi_{n_g}(d)).$
\end{center}
The `$\leq$' direction immediately follows from the $\epsilon_g$-monotonicity of $g$. For the converse inequality, by denseness, it is enough to show that if $o\in \obbas$ and $g(\bigwedge_{d \in \mathcal{D}}^{\epsilon_g(1)} \phi_1(d),\ldots, \bigwedge_{d \in \mathcal{D}}^{\epsilon_g(n_g)} \phi_{n_g}(d)) \leq o$, then $ \bigwedge_{d \in \mathcal{D}} g( \phi_1(d)\ldots,\phi_{n_g}(d)) \leq o$. Since $g\in \mathcal{G}$, we have:

\begin{center}
$g(\bigwedge_{d \in \mathcal{D}}^{\epsilon_g(1)} \phi_1(d),\ldots, \bigwedge_{d \in \mathcal{D}}^{\epsilon_g(n_g)} \phi_{n_g}(d)) = \bigwedge\{g( \phi_1(d_1)\ldots,\phi_{n_g}(d_{n_g}))\mid d_h\in \mathcal{D} \mbox{ for every } 1\leq h\leq n_g\}$.
\end{center}
By compactness (which can be applied by lemmas \ref{cor:congenial top for whites and blacks}(1) and \ref{Syn:Opn:Clsd:Appld:ClsdUp:Lemma}(1)),
\begin{center}
$\bigwedge\{g( \phi_1(d^{(i)}_1)\ldots,\phi_{n_g}(d^{(i)}_{n_g}))\mid 1\leq i\leq n\}\leq o$.
\end{center}
Let $\mathcal{D}': = \{d^{(i)}_h\mid 1\leq i\leq n\mbox{ and } 1\leq h\leq n_g\}$. Since $\mathcal{D}$ is down-directed, $d^\ast\leq \bigwedge \mathcal{D}'$ for some $d^\ast\in \mathcal{D}$. Then, since $g$ is $\epsilon_g$-monotone and each $\phi_h(p)$ is $\epsilon_g(h)$-positive in $p$, the following chain of inequalities holds
\begin{center}
\begin{tabular}{r c l}
$\bigwedge_{d \in \mathcal{D}} g( \phi_1(d)\ldots,\phi_{n_g}(d)) $& $\leq$ &
$g(\phi_1(d^\ast),\ldots,\phi_{n_g}(d^\ast))$\\
& $\leq$ & $ g(\phi_1(\bigwedge_{1\leq i\leq n}d^{(i)}_1),\ldots, \phi_{n_g}(\bigwedge_{1\leq i\leq n}d^{(i)}_{n_g}))$\\
& $\leq$ & $ \bigwedge\{g( \phi_1(d^{(i)}_1)\ldots,\phi_{n_g}(d^{(i)}_{n_g}))\mid 1\leq i\leq n\}$\\
& $\leq$ & $o.$
\end{tabular}
\end{center}


 If $\phi(p) = f(\overline{\phi(p)})$, then
\begin{center}
$\phi(\bigwedge\mathcal{D}) = f(\phi_1(\bigwedge\mathcal{D}),\ldots,\phi_{n_f}(\bigwedge\mathcal{D}))
= f(\bigwedge_{d \in \mathcal{D}}^{\epsilon_f(1)} \phi_1(d),\ldots, \bigwedge_{d \in \mathcal{D}}^{\epsilon_f(n_f)} \phi_{n_f}(d))$.
\end{center}
%
%
%
The second equality above holds by the inductive hypothesis. To finish the proof, we need to show that
\begin{center}
$f(\bigwedge_{d \in \mathcal{D}}^{\epsilon_f(1)} \phi_1(d),\ldots, \bigwedge_{d \in \mathcal{D}}^{\epsilon_f(n_f)} \phi_{n_f}(d)) = \bigwedge_{d \in \mathcal{D}} f( \phi_1(d)\ldots,\phi_{n_f}(d)).$
\end{center}
The `$\leq$' direction immediately follows from the $\epsilon_f$-monotonicity of $f$. For the converse inequality, by denseness, it is enough to show that if $o\in \obbas$ and $f(\bigwedge_{d \in \mathcal{D}}^{\epsilon_f(1)} \phi_1(d),\ldots, \bigwedge_{d \in \mathcal{D}}^{\epsilon_f(n_f)} \phi_{n_f}(d)) \leq o$, then $ \bigwedge_{d \in \mathcal{D}} f( \phi_1(d)\ldots,\phi_{n_f}(d)) \leq o$. By lemmas \ref{lemma:uncongenial for the whites}(2) and \ref{Syn:Opn:Clsd:Appld:ClsdUp:Lemma}(1), we have:

\begin{center}
$f(\bigwedge_{d \in \mathcal{D}}^{\epsilon_f(1)} \phi_1(d),\ldots, \bigwedge_{d \in \mathcal{D}}^{\epsilon_f(n_f)} \phi_{n_f}(d)) = \bigwedge\{f( \phi_1(d_1)\ldots,\phi_{n_f}(d_{n_f}))\mid d_h\in \mathcal{D} \mbox{ for every } 1\leq h\leq n_f\}$.
\end{center}
By compactness (which can be applied by lemmas \ref{cor: open upset to open upset for white box}(2) and \ref{Syn:Opn:Clsd:Appld:ClsdUp:Lemma}(1)),
\begin{center}
$\bigwedge\{f( \phi_1(d^{(i)}_1)\ldots,\phi_{n_f}(d^{(i)}_{n_f}))\mid 1\leq i\leq n\}\leq o$.
\end{center}
Let $\mathcal{D}': = \{d^{(i)}_h\mid 1\leq i\leq n\mbox{ and } 1\leq h\leq n_f\}$. Since $\mathcal{D}$ is down-directed, $d^\ast\leq \bigwedge \mathcal{D}'$ for some $d^\ast\in \mathcal{D}$. Then, since $f$ is $\epsilon_f$-monotone and each $\phi_h(p)$ is $\epsilon_f(h)$-positive in $p$, the following chain of inequalities holds
\begin{center}
\begin{tabular}{r c l}
$\bigwedge_{d \in \mathcal{D}} f( \phi_1(d)\ldots,\phi_{n_f}(d)) $& $\leq$ &
$f(\phi_1(d^\ast),\ldots,\phi_{n_f}(d^\ast))$\\
& $\leq$ & $ f(\phi_1(\bigwedge_{1\leq i\leq n}d^{(i)}_1),\ldots, \phi_{n_f}(\bigwedge_{1\leq i\leq n}d^{(i)}_{n_f}))$\\
& $\leq$ & $ \bigwedge\{f( \phi_1(d^{(i)}_1)\ldots,\phi_{n_f}(d^{(i)}_{n_f}))\mid 1\leq i\leq n\}$\\
& $\leq$ & $o.$
\end{tabular}
\end{center}

 If $\phi(p) = g^\flat_i(\overline{\phi(p)})$, and $\epsilon_g(i) = 1$ then
\begin{center}
$\phi(\bigwedge\mathcal{D}) = g^\flat_i(\phi_1(\bigwedge\mathcal{D}),\ldots,\phi_{n_g}(\bigwedge\mathcal{D}))
= g^\flat_i(\bigwedge_{d \in \mathcal{D}}^{\epsilon_{g^\flat_i}(1)} \phi_1(d),\ldots, \bigwedge_{d \in \mathcal{D}}^{\epsilon_{g^\flat_i}(n_g)} \phi_{n_g}(d))$.
\end{center}
%
%
%
The second equality above holds by the inductive hypothesis. To finish the proof, we need to show that
\begin{center}
$g^\flat_i(\bigwedge_{d \in \mathcal{D}}^{\epsilon_{g^\flat_i}(1)} \phi_1(d),\ldots, \bigwedge_{d \in \mathcal{D}}^{\epsilon_{g^\flat_i}(n_g)} \phi_{n_g}(d)) = \bigwedge_{d \in \mathcal{D}} g^\flat_i( \phi_1(d)\ldots,\phi_{n_g}(d)).$
\end{center}
The `$\leq$' direction immediately follows from the $\epsilon_{g^\flat_i}$-monotonicity of $g^\flat_i$. For the converse inequality, by denseness, it is enough to show that if $o\in \obbas$ and $g^\flat_i(\bigwedge_{d \in \mathcal{D}}^{\epsilon_{g^\flat_i}(1)} \phi_1(d),\ldots, \bigwedge_{d \in \mathcal{D}}^{\epsilon_{g^\flat_i}(n_g)} \phi_{n_g}(d)) \leq o$, then $ \bigwedge_{d \in \mathcal{D}} g^\flat_i( \phi_1(d)\ldots,\phi_{n_g}(d)) \leq o$. By lemmas \ref{lemma: uncongenial for the blacks}(1) and \ref{Syn:Opn:Clsd:Appld:ClsdUp:Lemma}(1), we have:

\begin{center}
$g^\flat_i(\bigwedge_{d \in \mathcal{D}}^{\epsilon_{g^\flat_i}(1)} \phi_1(d),\ldots, \bigwedge_{d \in \mathcal{D}}^{\epsilon_{g^\flat_i}(n_g)} \phi_{n_g}(d)) = \bigwedge\{f( \phi_1(d_1)\ldots,\phi_{n_g}(d_{n_g}))\mid d_h\in \mathcal{D} \mbox{ for every } 1\leq h\leq n_g\}$.
\end{center}
By compactness (which can be applied by lemmas \ref{Blk:Diam:c:Clsd:Lemma}(1) and \ref{Syn:Opn:Clsd:Appld:ClsdUp:Lemma}(1)),
\begin{center}
$\bigwedge\{g^\flat_i( \phi_1(d^{(i)}_1)\ldots,\phi_{n_g}(d^{(i)}_{n_g}))\mid 1\leq i\leq n\}\leq o$.
\end{center}
Let $\mathcal{D}': = \{d^{(i)}_h\mid 1\leq i\leq n\mbox{ and } 1\leq h\leq n_g\}$. Since $\mathcal{D}$ is down-directed, $d^\ast\leq \bigwedge \mathcal{D}'$ for some $d^\ast\in \mathcal{D}$. Then, since $g^\flat_i$ is $\epsilon_{g^\flat_i}$-monotone and each $\phi_h(p)$ is $\epsilon_{g^\flat_i}(h)$-positive in $p$, the following chain of inequalities holds
\begin{center}
\begin{tabular}{r c l}
$\bigwedge_{d \in \mathcal{D}} g^\flat_i( \phi_1(d)\ldots,\phi_{n_g}(d)) $& $\leq$ &
$g^\flat_i(\phi_1(d^\ast),\ldots,\phi_{n_g}(d^\ast))$\\
& $\leq$ & $ g^\flat_i(\phi_1(\bigwedge_{1\leq i\leq n}d^{(i)}_1),\ldots, \phi_{n_g}(\bigwedge_{1\leq i\leq n}d^{(i)}_{n_g}))$\\
& $\leq$ & $ \bigwedge\{g^\flat_i( \phi_1(d^{(i)}_1)\ldots,\phi_{n_g}(d^{(i)}_{n_g}))\mid 1\leq i\leq n\}$\\
& $\leq$ & $o.$
\end{tabular}
\end{center}

The remaining cases are similar, and left to the reader.

Thus the proof of item (1) is concluded. Item (2) can be proved similarly by simultaneous induction on $\phi$ negative in $p$ and $\psi$ positive in $p$.
\end{proof}

\paragraph{Proof of the Righthanded Ackermann lemma for admissible assignments (Lemma \ref{Ackermann:Dscrptv:Right:Lemma})}
To keep the notation uncluttered, we will simply write $\beta_{i}$ and $\gamma_{i}$ for ${\beta_i}^V_p$ and ${\gamma_i}^V_p$, respectively. The implication from bottom-to-top follows by the monotonicity of the $\beta_i$ and the antitonicity of the $\gamma_i$ in $p$. Indeed, if $\alpha(V) \leq u$, then, for each $1 \leq i \leq n$, ${\beta_i}(\alpha(V)) \leq {\beta_i}(u) \leq {\gamma_i}(u) \leq {\gamma_i}(\alpha(V))$.

For the sake of the converse implication assume that ${\beta_i}(\alpha(V)) \leq {\gamma_i}(\alpha(V))$ for all $1 \leq i \leq n$. By lemma \ref{Syn:Opn:Clsd:Appld:ClsdUp:Lemma}, $\alpha(V)\in \kbbas$. Hence $\alpha(V) = \bigwedge \{a \in \bba \mid \alpha(V) \leq a \}$, making it the meet of a down-directed subset of $\kbbas$.  Thus, for any $1 \leq i \leq n$, we have
\[
\beta_i(\bigwedge \{a \in \bba \mid \alpha(V) \leq a \}) \leq \gamma_i(\bigwedge \{a \in \bba \mid \alpha(V) \leq a \}).
\]
Since $\gamma_i$ is syntactically open and negative in $p$, and $\beta_i$ is syntactically closed and positive in $p$, we may apply lemma \ref{Esakia:Syn:Clsd:Opn:Lemma} and equivalently obtain
\[
\bigwedge \{\beta_i(a)  \mid a \in \bba, \: \alpha(V) \leq a \} \leq \bigvee \{\gamma_i(a) \mid a \in \bba, \: \alpha(V) \leq a \}.
\]
By lemma \ref{Syn:Opn:Clsd:Appld:ClsdUp:Lemma}, $ \beta_i(a)\in \kbbas$ and  $\gamma_i(a)\in \obbas$ for each $a \in \bba$. Hence 
by compactness 
\[
\beta_i(b_1) \wedge \cdots \beta_i(b_k) \leq \gamma_i(a_1) \vee \cdots \vee \gamma_i(a_m).
\]
for some $a_1, \ldots, a_m, b_1, \ldots b_k \in \bba$ with $\alpha(V) \leq a_j$, $1 \leq j \leq m$, and $\alpha(V) \leq b_h$, $1 \leq h \leq k$.
Let $a_i = b_1 \wedge \cdots \wedge b_{k} \wedge a_1 \wedge \cdots \wedge a_m$. Then $\alpha(V) \leq a_i \in \bba$. By the monotonicity of $\beta_i$ and the antitonicity of $\gamma_i$ it follows that
\[
\beta_i(a_i) \leq \gamma_i(a_i).
\]
Now, letting $u = a_1 \wedge \cdots \wedge a_n$, we have $\alpha(V) \leq u \in \bba$, and by the  monotonicity of the $\beta_i$ and the antitonicity of the $\gamma_i$ we get that
\[
\beta_i(u) \leq \gamma_i(u) \textrm{ for all } 1 \leq i \leq n.
\]

\paragraph{Proof of the Lefthanded Ackermann lemma for admissible assignments (Lemma \ref{Ackermann:Dscrptv:Left:Lemma})}
As in the previous lemma we will write $\beta_{i}$ and $\gamma_{i}$ for ${\beta_i}^V_p$ and ${\gamma_i}^V_p$, respectively. The implication from bottom to top follows by the antitonicity  of the $\beta_i$ and the monotonicity of the $\gamma_i$.

For the sake of the converse implication assume that ${\beta_i}^V_p(\alpha(V)) \leq {\gamma_i}^V_p(\alpha(V))$ for all $1 \leq i \leq n$. But $\alpha$ is syntactically open and (trivially) negative in $p$, hence by lemma \ref{Syn:Opn:Clsd:Appld:ClsdUp:Lemma}(2), $\alpha(V)\in \obbas$, i.e.\ $\alpha(V) = \bigvee\{a \in \bba \mid a \leq  \alpha(V)\}$. Thus, for any $1 \leq i \leq n$, it is the case that
\[
{\beta_i}(\bigvee\{a \in \bba \mid a \leq  \alpha(V)\}) \leq {\gamma_i}(\bigvee\{a \in \bba \mid a \leq  \alpha(V)\}).
\]
Hence by lemma \ref{Esakia:Syn:Clsd:Opn:Lemma} (3) and (4)
\[
\bigwedge \{ {\beta_i}(a) \mid a \in \bba, a \leq  \alpha(V) \}  \leq \bigvee \{ \gamma_i(a) \mid a \in \bba,  a \leq  \alpha(V) \}.
\]
The proof now proceeds like that of lemma \ref{Ackermann:Dscrptv:Right:Lemma}.

\end{document}